\documentclass[a4paper, 10pt]{amsart}

\usepackage[mathscr]{eucal}
\usepackage{amsmath,amssymb,amsfonts,amsthm, color, eufrak}
\usepackage{enumitem}
\usepackage{bm}
\usepackage{hyperref}
\usepackage[capitalize          
           ]{cleveref}

\newcommand{\Subref}[1]{\dosubref\Cref#1\relax}
\def\dosubref#1#2:#3\relax{#1{#2}.\ref{#2:#3}}

\newenvironment{remenumerate}[1][]
{\enumerate[leftmargin=0em,itemindent=2em,itemsep=0.2em] }
{\endenumerate}

\newtheorem{theorem}{Theorem}[section]
\newtheorem{lemma}[theorem]{Lemma}
\newtheorem{corollary}[theorem]{Corollary}
\newtheorem{proposition}[theorem]{Proposition}
\newtheorem{conjecture}[theorem]{Conjecture}

\theoremstyle{definition}

\newtheorem{definition}[theorem]{Definition}
\newtheorem{remark}[theorem]{Remark}
\newtheorem{remarks}[theorem]{\bf Remark}

 \DeclareMathOperator{\ord}{ord}
\DeclareMathOperator{\lcm}{lcm} 
 \DeclareMathOperator{\rad}{rad}
 \DeclareMathOperator{\supp}{supp}
\DeclareMathOperator{\Pic}{Pic} 
\DeclareMathOperator{\Ker}{Ker} 
\DeclareMathOperator{\End}{End} \DeclareMathOperator{\add}{add}
\DeclareMathOperator{\udim}{udim}
\DeclareMathOperator{\Ext}{Ext}
\DeclareMathOperator{\ann}{ann}
\DeclareMathOperator{\modspec}{modspec}

\providecommand{\monext}{\propto}
\providecommand{\card}[1]{\lvert#1\rvert}
\providecommand{\length}[1]{\lvert#1\rvert}

\newcommand{\N}{\mathbb N}
\newcommand{\Z}{\mathbb Z}
\newcommand{\R}{\mathbb R}
\newcommand{\Q}{\mathbb Q}

\newcommand{\tensor}{\otimes}
\newcommand{\FF}{\text{\rm FF}}

\newcommand{\tor}{\text{\rm tor}}

\newcommand{\proj}{\text{\rm proj}}
\newcommand{\red}{\text{\rm red}}
\newcommand{\rat}{\text{\rm rat}}

\newcommand{\DP}{\negthinspace : \negthinspace}

\newcommand{\la}{\langle}
\newcommand{\ra}{\rangle}
\newcommand{\be}{\begin{equation}}
\newcommand{\ee}{\end{equation}}
\newcommand{\und}{\;\mbox{ and }\;}
\newcommand{\nn}{\nonumber}
\newcommand{\ber}{\begin{eqnarray}}
\newcommand{\eer}{\end{eqnarray}}
\newcommand{\Sum}[2]{\underset{#1}{\overset{#2}{\sum}}}
\newcommand{\Summ}[1]{\underset{#1}{\sum}}

\newcommand{\Fc}{\mathcal F}
\newcommand{\vp}{\mathsf v}
\newcommand{\bdot}{{\cdot}}

\newcommand{\wtilde}{\widetilde}

\newcommand{\Ff}{\mathcal F_{\rat}}
\newcommand{\rk}{\mathsf{r}}

\newcommand{\quo}{ \mathsf{q} }
\newcommand{\cA}{\mathcal A}
\newcommand{\cB}{\mathcal B}
\newcommand{\cC}{\mathcal C}
\newcommand{\cT}{\mathcal T}
\newcommand{\cV}{\mathcal V}
\newcommand{\cW}{\mathcal W}
\newcommand{\sL}{\mathsf L}
\newcommand{\sZ}{\mathsf Z}
\newcommand{\sd}{\mathsf d}
\newcommand{\st}{\mathsf t}
\newcommand{\enumequiv}{%
  \renewcommand{\theenumi}{(\alph{enumi})}%
  \renewcommand{\labelenumi}{\theenumi}%
}

\renewcommand{\sc}{\mathsf c} 
\renewcommand{\vec}[1]{\mathbf{#1}}
\renewcommand{\t}{\, | \,}
\renewcommand{\labelenumi}{\arabic{enumi}.}

\begin{document}

\thanks{The first author was a Fulbright-NAWI Graz Visiting Professor in the Natural Sciences and supported by the Austrian-American Education Commission. The second and  fourth authors were supported by the Austrian Science Fund FWF  Projects P26036-N26 and W1230.}

\author{N.R. Baeth and A. Geroldinger and D.J. Grynkiewicz and D. Smertnig}

\address{Department of Mathematics and Computer Science \\ University of Central Missouri \\ Warrensburg, MO 64093, USA; \normalfont{ baeth@ucmo.edu}}

\address{Institut f\"ur Mathematik und Wissenschaftliches Rechnen \\
Karl-Franzens-Universit\"at Graz \\
Heinrichstra\ss e 36\\
8010 Graz, Austria; \normalfont{alfred.geroldinger@uni-graz.at,  daniel.smertnig@uni-graz.at} }

\address{Department of Mathematical Sciences \\ University of Memphis \\ Memphis, TN 38152, USA; \normalfont{diambri@hotmail.com}}



\begin{abstract}
Let $R$ be a ring and let $\mathcal C$ be a small class of right $R$-modules which is closed under finite direct sums, direct summands, and isomorphisms. Let $\mathcal V (\mathcal C)$ denote a set of representatives of isomorphism classes in $\mathcal C$ and, for any module $M$ in $\mathcal C$, let $[M]$ denote the unique element in $\mathcal V (\mathcal C)$ isomorphic to $M$. Then $\mathcal V (\mathcal C)$ is a reduced commutative semigroup with operation defined by $[M] + [N] = [M \oplus N]$, and this semigroup carries all information about direct-sum decompositions of modules in $\mathcal C$. This semigroup-theoretical point of view has been prevalent in the theory of direct-sum decompositions since it was  shown that if $\End_R (M)$ is semilocal for all $M\in \mathcal C$, then $\mathcal V (\mathcal C)$ is a Krull monoid. Suppose that the monoid $\mathcal V (\mathcal C)$ is Krull with a finitely generated  class group (for example, when $\mathcal C$ is the class of finitely generated torsion-free modules and $R$ is a one-dimensional reduced Noetherian local ring). In this case we study the arithmetic of $\mathcal V (\mathcal C)$ using new methods from zero-sum theory. Furthermore, based on module-theoretic work of Lam, Levy, Robson, and others we study the algebraic and arithmetic structure of the monoid $\mathcal V (\mathcal C)$ for certain classes of modules over Pr\"ufer rings and hereditary Noetherian prime rings.
\end{abstract}

\title[Direct-sum Decompositions and associated combinatorial problems]{A semigroup-theoretical view of direct-sum decompositions \\ and associated combinatorial problems}

\maketitle

\bigskip
\section{Introduction} \label{1}
\bigskip

The overarching goal of this manuscript is to study direct-sum decompositions of modules into indecomposable modules. Let $R$ be a ring and let $M$ be a right $R$-module. If $M$ is Noetherian or Artinian, then a well-known and simple argument shows that $M$ is a finite direct sum of indecomposable right $R$-modules. If, for example, $M$ is either both Noetherian and Artinian or if $R$ is a principal ideal domain and $M$ is finitely generated, then such a direct-sum decomposition is unique; that is, the Krull-Remak-Schmidt-Azumaya property (KRSA, for short) holds. For a simple example of non-unique direct-sum decomposition, consider a commutative domain $R$ with distinct non-principal maximal ideals $\mathfrak m_1$ and $\mathfrak m_2$. Then the epimorphism $\mathfrak m_1 \oplus \mathfrak m_2 \to R$, defined by $(m_1, m_2) \mapsto m_1+m_2$, gives rise to a split short exact sequence $0 \to \mathfrak m_1 \cap \mathfrak m_2 \to \mathfrak m_1 \oplus \mathfrak m_2 \to R \to 0$. Hence $\mathfrak m_1 \oplus \mathfrak m_2 \cong R \oplus (\mathfrak m_1 \cap \mathfrak m_2)$. Since $\mathfrak m_1$ and $\mathfrak m_2$ are not principal and since $R$ is a domain, each $\mathfrak m_i$ is indecomposable as an $R$-module and is not  isomorphic to $R$. Since the pioneering work of Krull, Remak, Schmidt, and Azumaya, direct-sum decompositions have been a classic topic in module theory. We refer the reader to \cite{Fa03a} for an overview of the celebrated Krull-Remak-Schmidt-Azumaya Theorem and related topics in direct-sum theory.

The work of  Facchini, Herbera, and Wiegand \cite{Fa-He00, MR1778501,Fa-Wi04} introduced a new semigroup-theoretical approach to the study of direct-sum decompositions of modules when KRSA fails to hold. Let $\mathcal C$ be a small class of right $R$-modules which is closed under finite direct sums, direct summands, and isomorphisms. Let $\mathcal V (\mathcal C)$ denote a set of representatives of isomorphism classes in $\mathcal C$ and, for any module $M$ in $\mathcal C$, let $[M]$ denote the unique element in $\mathcal V (\mathcal C)$ isomorphic to $M$. Then $\mathcal V (\mathcal C)$ is a reduced commutative semigroup with operation defined by $[M] + [N] = [M \oplus N]$, and this semigroup carries all information about direct-sum decompositions of modules in $\mathcal C$. In particular, $[M]$ is an irreducible element of the semigroup $\mathcal V (\mathcal C)$ if and only if $M$ is an indecomposable module, and direct-sum decompositions of modules in $\mathcal C$ are unique (equivalently, KRSA holds) if and only if $\mathcal V ( \mathcal C)$ is a free abelian monoid. Suppose $\mathcal C$ is a class of $R$-modules as just defined and that KRSA fails. Then arithmetical questions of the following type naturally arise.

\medskip

\begin{itemize}
\item [{\bf Q1:}] If $M_1 \oplus M_2 = N_1 \oplus \cdots \oplus N_l$, where $M_1, M_2, N_1, \ldots, N_l$ are indecomposable modules, does there exist an upper bound for $l$ depending only on $\mathcal C$?
\item [{\bf Q2:}] Suppose an indecomposable module $M$ is isomorphic to a direct summand of $N_1 \oplus \cdots \oplus N_l$ for indecomposable modules $N_1, \ldots , N_l$. Is there an upper bound (depending only on $\mathcal C$) for the number $|I|$ such that $M$ is already isomorphic to a direct summand of $\oplus_{i \in I} N_i$?
\end{itemize}

\medskip

We propose the following overall strategy to tackle these and other arithmetical questions regarding non-unique direct-sum decompositions of modules.

\medskip

\begin{itemize}
\item[{\bf A.}] Use module-theoretic results in order to describe the algebraic structure of the semigroup $\mathcal V (\mathcal C)$.

\item[{\bf B.}] Use factorization theory to study the arithmetic structure of the semigroup $\mathcal V (\mathcal C)$.
\end{itemize}

\medskip

This strategy is relatively new, but has been used in several recent papers for certain classes of modules (see, for example \cite{Ba-Ge14b}). In the present paper we pursue this strategy for three classes of finitely generated modules: torsion-free modules over one-dimensional reduced commutative Noetherian local rings (Section \ref{5}), modules over Pr\"ufer rings (Section \ref{6}), and right-modules over hereditary Noetherian prime rings (Section \ref{7}).

\medskip

First, suppose that $\mathcal C$ is a class of $R$-modules such that the endomorphism ring $\End_R (M)$ is semilocal for each $R$-module $M$. In this case, Facchini proved \cite[Theorem 3.4]{Fa02} that  $\mathcal V (\mathcal C)$ is a reduced Krull monoid.  Earlier results in this direction can be found in \cite{Fa-He00, MR1778501, Wi01}). A reduced Krull monoid $H$ is uniquely determined by its characteristic $(G, (m_g)_{g \in G})$ where $G$ is the class group of $H$ and, for each $g \in G$, $m_g$ is the cardinality of the set of prime divisors lying in the class $g \in G$ (see Section \ref{2}). Many arithmetical problems depend only on the set $G_p = \{ g \in G : m_g > 0 \}$ of classes containing prime divisors and, for simplicity, we restrict our discussion to this case. Therefore, in order to determine the structure of $\mathcal V (\mathcal C)$, it is required to determine the characteristic of $\mathcal V(\mathcal C)$, or at least the tuple $(G, G_P)$. In general, this is an herculean task. Indeed, even for specific classes $\mathcal C$ of modules where it is known that $\End_R(M)$ is semilocal for each $M$ in $\mathcal C$, except for in very special situations, we have limited information about $(G, G_P)$. For an overview what is known for certain classes of finitely generated modules over certain one- and two-dimensional Noetherian local rings, see \cite{Ba-Ge14b}. Nevertheless, suppose we are in a situation where we are able to determine the tuple $(G, G_P)$. Then, by a well-known transfer homomorphism (see Proposition \ref{3.1}), arithmetical problems in $\mathcal V (\mathcal C)$ can be studied in the (combinatorial) Krull monoid $\mathcal B (G_P)$, the monoid of zero-sum sequences over $G_P$. Therefore, in this setting, the study of uniqueness and non-uniqueness of direct-sum decompositions can be reduced to zero-sum theory, a flourishing subfield of combinatorial and additive number theory. Except for occasional work (see \cite{An-Ch-Sm94c} for early contributions), the focus of (arithmetical) zero-sum theory has been restricted to the case where $G$ is a finite abelian group, whereas class groups $G$ stemming from monoids of modules $\mathcal V (\mathcal C)$ are often infinite. In Section \ref{4} we study zero-sum theory over finitely generated free abelian groups using new methods from matroid theory (goal {\bf B}). The focus will be on the study of the Davenport constant which can often be used to provide bounds on important factorization-theoretic invariants. Specifically, upper and lower bounds on the Davenport constant are given in Theorem \ref{4.2}. In Section \ref{5} these results will be applied to Krull monoids. In particular, module-theoretic results will be used in order to describe $\mathcal V(\mathcal C)$ (goal {\bf A}) in such a way that we can apply the results from Section \ref{4} (see Corollaries \ref{5.5} and \ref{5.6}).

\medskip

Apart from cases where $\mathcal V(\mathcal C)$ is Krull and some trivial cases (say, where KRSA holds whence $\mathcal V(\mathcal C)$ is free abelian), the structure of $\mathcal V (\mathcal C)$ has not been studied from a semigroup-theoretical point of view. Our goal in the present paper is to take this approach for certain classes of finitely generated modules. The most simple case is the classical Theorem of Steinitz which determines the structure of direct-sum decompositions of finitely generated modules over Dedekind domains. This result has found generalizations into many directions. We consider two generalizations in the setting of certain classes of modules over Pr\"ufer rings and over hereditary Noetherian prime rings.

In Section \ref{6} we study the class of finitely generated projective modules over a class of Pr\"ufer rings. Based on work of Feng and Lam (\cite{Fe-La02a}) we show that $\mathcal V (\mathcal C_{\proj})$ is a finitely primary monoid (and goal {\bf A} is achieved). Since the arithmetic of finitely primary monoids has been well-studied, there are well-known answers to questions about the arithmetic of $\mathcal V(\mathcal C)$ (goal {\bf B}). These results are summarized in Theorems \ref{6.1} and \ref{6.3}.

In Section \ref{7} we study the class of finitely generated projective modules over hereditary Noetherian prime rings (which generalize non-commutative Dedekind prime rings). Deep module theoretic work by Levy and Robson (\cite{Le-Ro11a}) allows an algebraic characterization of the associated monoids of stable isomorphism classes of modules (goal {\bf A}). We first introduce monoids of this type in an abstract setting and study their arithmetic (Propositions \ref{monext} and \ref{acm-factorization}). We then apply these results to monoids of modules (goal {\bf B}) over hereditary Noetherian prime rings in Theorem \ref{hnp-proj}.

In both the setting of Sections \ref{6} and of Section \ref{7}, the respective monoids of modules are seen to be half-factorial (all direct-sum decompositions of a given module have the same length). However, these modules behave very differently with respect to finer arithmetical invariants including the $\omega$-invariants and the tame degrees (see Theorems \ref{6.1}, \ref{6.3}, and \ref{hnp-proj}).

\medskip

Section \ref{2} is preparatory in nature. There we gather together arithmetical concepts from factorization theory as well as required material on (generalized) Krull monoids and monoids of modules. An even more detailed description of these concepts and their relevance to module theory can be found in \cite{Ba-Ge14b}. We also refer the reader to the monograph \cite{Ge-HK06a} for more information on factorization theory, and to \cite{Ba-Wi13a}, for a friendly introduction to the interplay of factorization theory and module theory.

\bigskip
\section{Preliminaries} \label{2}
\bigskip

We denote by $\N$ the set of positive integers and set $\N_0 = \N \cup \{0\}$. For real numbers $a, b \in \R$ we set $[a, b] = \{ x \in \Z : a \le x \le b \}$. For a subset $L \subset \Z$, we denote by $\Delta (L) \subset \N$ the set of distances of $L$. This is the set $\{l-k\colon k < l \in L\text{ and }L \cap [k, l] = \{k, l\}\}$. Let $G$ be an additive abelian group and let $A, B \subset G$ be subsets. Then $A+B = \{a+b : a \in A, b \in B \}$ denotes the sumset of $A$ and $B$, $-A = \{-a\,:\, a \in A \}$ is the negative of $A$, $g+A = \{g\}+A$ for $g \in G$, and $\langle A \rangle \subset G$ denotes the subgroup of $G$ generated by $A$.

\medskip

A family $(e_i)_{i \in I}$ of elements of $G$ is said to be {\it independent} if $e_i \ne 0$ for all $i \in I$ and, for every family $(m_i)_{i \in I} \in \Z^{(I)}$,
\[
\sum_{i \in I} m_ie_i =0 \qquad \text{implies} \qquad m_i e_i =0 \quad \text{for all} \quad i \in I\,.
\]
An independent family $(e_i)_{i \in I}$ is called a {\it basis} for $G$ if $G = \bigoplus_{i \in I} \langle e_i \rangle$.  If $G$ is torsion-free, then $\mathsf r (G) = \dim_{\Q} (G \tensor_{\Z} \Q)$ denotes the {\it rank} of $G$ and
\[
G_{|I|}^+ = \Big\{ \sum_{i \in I} \epsilon_i e_i : \epsilon_i \in \{0,1\} \ \text{for all} \ i \in I, \ \text{but not all equal to zero} \Big\} \subset \oplus_{i \in I} \langle e_i \rangle \cong \Z^{(I)}
\]
denotes the set of non-zero vertices of the hypercube in $\oplus_{i \in I} \langle e_i \rangle$. This definition clearly depends on the chosen basis, but throughout we refer to  $G_{|I|}^+$ only after a  basis has been fixed.

\medskip

By a {\it monoid} we always mean a commutative semigroup with identity which satisfies the cancellation law. Thus, if $R$ is a commutative ring and $R^{\bullet}$ its set of regular elements, then $R^{\bullet}$ is a multiplicative monoid. Let $H$ be a (multiplicatively written) monoid. We denote by $\mathsf q (H)$ a quotient group of $H$, by $H^{\times}$ the group of invertible elements of $H$, and by $\widehat H$ the complete integral closure of $H$ with $H \subset \widehat H \subset \mathsf q (H)$.
An element $u \in H$ is called an \emph{atom} if $u \not \in H^\times$ and $u=ab$ with $a,b \in H$ implies that either $a \in H^\times$ or $b \in H^\times$.
The set of all atoms of $H$ is denoted by $\mathcal A (H)$. We say that a monoid $H$ with identity $1$ is reduced if $H^{\times} = \{1\}$, and we denote by $H_{\red} = \{ a H^{\times} : a \in H\}$ the associated reduced monoid of $H$.

\medskip
\noindent
{\bf Free abelian monoids and groups.} A monoid $F$ is  {\it free abelian with basis $P \subset F$}  if every $a \in F$ has a unique representation of the form
\[
a = \prod_{p \in P} p^{\mathsf v_p(a) } \quad \text{with} \quad
\mathsf v_p(a) \in \N_0 \ \text{ and } \ \mathsf v_p(a) = 0 \ \text{
for almost all } \ p \in P \,.
\]
In this case, we set $F = \mathcal F(P)$. The isomorphism $a \mapsto (\mathsf v_p(a))_{p \in P}$ from the multiplicative monoid $\mathcal F (P)$  to the additive monoid $(\N_0^{(P)}, +)$ induces an isomorphism from the quotient group $\mathsf q (F)$ to $\Z^{(P)}$. We denote by $\mathcal F_{\rat} (P)$ the multiplicative monoid isomorphic to $(\Q_{\ge 0}^{(P)}, +)$ and we tacitly assume that $\mathcal F (P) \subset \mathcal F_{\rat} (P)$. The quotient group of $\mathcal F_{\rat} (P)$ is isomorphic to $(\Q^{(P)}, +)$ and an
element $a \in \mathsf q (\mathcal F_{\rat} (P) )$ will be written in the form $a =  \prod_{p \in P} p^{\mathsf v_p(a) } \quad \text{with} \quad
\mathsf v_p(a) \in \Q \ \text{ and } \ \mathsf v_p(a) = 0 \ \text{
for almost all } \ p \in P$. We
call
\[
  |a| = \sum_{p \in P} |\mathsf v_p (a)|  \in \Q_{\ge 0} \quad \text{the \ {\it
length}} \ \text{of} \ a \, \ \text{and} \ \supp (a) = \{p \in P : \mathsf v_p (a) \ne 0 \} \subset P \quad \text{the {\it support} of} \ a \,.
\]

Clearly there are inclusions $\N_0^{(P)} \subset \Z^{(P)} \subset \Q^{(P)} \subset \Q^P$. Elements of $\Q^P$ will (usually) be written as $\boldsymbol x$ and, if $\boldsymbol x \in \Q^P$, then we tacitly assume that $\boldsymbol x = (x_p)_{p \in P}$ and $\mathsf v_p (\boldsymbol x) = x_p$ for all $p \in P$. For $p \in P$, let $\boldsymbol e_p \in \N_0^{(P)}$ denote the standard vector with $e_{p, q} = 1$ if $p=q$ and $e_{p,q} = 0$ for all $q \in P \setminus \{p\}$. Then $(\boldsymbol e_p)_{p \in P}$ is the standard basis of $\Z^{(P)}$.

\medskip
\noindent
{\bf Factorizations and sets of lengths.} Let $H$ be a monoid. The free abelian monoid $\mathsf Z (H) = \mathcal F (\mathcal A (H_{\red}))$ is called the {\it factorization monoid} of $H$ and the unique homomorphism
\[
\pi \colon \mathsf Z (H) \to H_{\red} \quad \text{satisfying} \quad
\pi (u) = u \quad \text{for each} \quad u \in \mathcal A(H_\red)
\]
is called the  {\it factorization homomorphism} of $H$. For $a
\in H$,
\begin{flalign*}
\qquad\qquad & \mathsf Z_H (a) = \mathsf Z (a)  = \pi^{-1} (aH^\times) \subset \mathsf Z (H) \text{ is the {\it set of factorizations} of $a$,} & \\
& \mathsf L_H (a) = \mathsf L (a) = \bigl\{ |z| : z \in \mathsf Z (a) \bigr\} \subset \N_0 \text{ is the {\it set of lengths} of $a$,}\text{ and} & \\
& \mathcal L (H) = \{ \mathsf L (a) : a \in H \}  \text{ is the {\it system of sets of lengths} of $H$}. &
\end{flalign*}

We say that $H$ is {\it atomic} if $\mathsf Z(a) \ne \emptyset$ for each $a \in H$, that $H$ is an {\it \FF-monoid} if $\mathsf Z (a)$ is finite and nonempty for each $a \in H$, and that $H$ is {\it factorial} if $|\mathsf Z (a)| = 1$ for each $a \in H$. For the remainder of this section we assume that $H$ is atomic.

\medskip

Among the most well-studied invariants in factorization theory are those that describe the structure of sets of lengths of elements in $H$. Let $k \in \mathbb N$. If $H \ne H^{\times}$, then
\[
      \mathcal U_k (H) \ = \ \bigcup_{k \in L \in \mathcal L (H)} L
\]
is the {\it union of  sets of lengths} containing $k$. If $
H^\times=H$, we set $\mathcal U_k (H) =\{k\}$. In either case we define
$\rho_k (H) = \sup \mathcal U_k (H)$  and $\lambda_k (H)= \min
\mathcal U_k (H)$. Clearly, $\mathcal U_1 (H) =\{1\}$ and, for each $k\in \mathbb N$, $k
\in \mathcal U_k (H)$. In particular, $k+l \in \mathcal U_k (H) + \mathcal U_l (H) \subset \mathcal U_{k+l} (H)$ for each $k, l \in \N$. With $\Delta(L)$ the set of distances of a length set $L$,
\[
\Delta (H) = \bigcup_{L \in \mathcal L (H)} \Delta (L)
\]
denotes the {\it set of distances} of $H$. By definition, $\rho_k (H) = k$ for all $k \in \N$ if and only if $\mathcal U_k (H) = \{k\}$ for all $k \in \N$ if and only if $\Delta (H) = \emptyset$. In this case, $H$ is said to be {\it half-factorial}. If $\Delta (H) = \{d\}$ for some $d \in \N$, then for all $L \in \mathcal L (H)$ and for all $k \in \mathbb N$, $L$ and $\mathcal U_k (H)$ are arithmetical progressions with difference $d$.

\medskip

Let $M \in \mathbb N_0$, $d \in \N$, and  $\{0,d\} \subset \mathcal D \subset [0,d]$.
A subset  $L \subset \Z$  is called an {\it almost
arithmetical multiprogression} ({\rm AAMP}) with {\it difference} $d$, {\it period} $\mathcal D$, and {\it bound} $M$ if
\[
L = y + (L' \cup L^* \cup L'') \, \subset \, y + \mathcal D + d \Z \text{ where}
\]
\begin{itemize}
\item  $L^*$ is finite and nonempty with $\min L^* = 0$ and $L^* =
       (\mathcal D + d \Z) \cap [0, \max L^*]$,
\item $L' \subset [-M, -1]$,
\item $L'' \subset \max L^* + [1,M]$, and
\item $y \in \Z$.
\end{itemize}
Note that every AAMP is a finite non-empty subset of $\mathbb Z$ and that an  AAMP with period $\{0,d\}$ and bound $M=0$ is a (usual) arithmetical progression with difference $d$.

\medskip
\noindent
{\bf Distance between factorizations and catenary degrees.}
Let $z,\, z' \in \mathsf Z (H)$. Then we can write
\[
z = u_1 \cdot \ldots \cdot u_lv_1 \cdot \ldots \cdot v_m \quad
\text{and} \quad z' = u_1 \cdot \ldots \cdot u_lw_1 \cdot \ldots
\cdot w_n\,,
\]
where  $l,\,m,\, n\in \N_0$ and $u_1, \ldots, u_l,\,v_1, \ldots,v_m,\,
w_1, \ldots, w_n \in \mathcal A(H_\red)$ are such that
\[
\{v_1 ,\ldots, v_m \} \cap \{w_1, \ldots, w_n \} = \emptyset\,.
\]
Then $\gcd(z,z')=u_1\cdot\ldots\cdot u_l$ and we call
\[
\mathsf d (z, z') = \max \{m,\, n\} = \max \{ |z \gcd (z, z')^{-1}|,
|z' \gcd (z, z')^{-1}| \} \in \N_0
\]
the {\it distance} between $z$ and $z'$. The  {\it catenary degree} $\mathsf c(a)$ of an element
$a \in H$ is the smallest $N \in \N_0 \cup \{ \infty\}$ such that, for
      any two factorizations $z$ and $z'$ of $a$, there exists a finite
sequence $z = z_0\,, \, z_1\,, \ldots, z_k = z'$ \ of
factorizations of \ $a$ \ such that \ $\mathsf d (z_{i-1}, z_i) \le
N \text{ for all } i \in [1,k]$.
We denote by
\[
\mathsf c(H) = \sup \{ \mathsf c(a) : a \in H\} \in \N_0 \cup\{\infty\}
\,
\]
the {\it catenary degree} of $H$. By definition, $|\mathsf Z (a)| = 1$ (i.e., $a$ has unique factorization in $H$) if and only if $\mathsf c (a) = 0$, and thus $H$ is factorial if and only if $\mathsf c (H) = 0$. Suppose now that $H$ is not factorial. Then it is an easy consequence of the definitions that $2 + \sup \Delta (H) \le \mathsf c (H)$. In particular, if $\mathsf c (H) = 2$, then $\Delta (H) = \emptyset$, and if $\mathsf c (H) = 3$, then $\Delta (H) = \{1\}$.

\medskip
\noindent
{\bf The $\omega$-invariants and tame degrees.} We now recall the $\omega$-invariants as well as local and global tame degrees. These are well-studied invariants in the factorization theory of rings and semigroups. They have also been considered in module-theoretic situations in terms of the so-called {\it semi-exchange property} (see \cite{Di07a}).

For  an element $a$ in an atomic monoid $H$,  let  $\omega (H,a)$ denote the
      smallest  $N \in \N_0 \cup \{\infty\}$ having the following property{\rm \,:}
      \begin{enumerate}
      \smallskip
      \item[] For any multiple $b$ of $a$ and any factorization $b = v_1 \cdot \ldots \cdot v_n$ of $b$,  there exists a
              subset $\Omega \subset [1,n] $ with $|\Omega | \le N $ such that
              \[
              a \Bigm| \, \prod_{\nu \in \Omega} v_\nu \,.
              \]
      \end{enumerate}
We then define
\[
\omega (H) = \sup \{ \omega (H, u) : u \in \mathcal A (H) \} \in \mathbb N_0 \cup \{\infty\} \,.
\]
It is then clear that an atom $u \in H$ is prime if and only if $\omega (H, u) = 1$ and thus $H$ is factorial if and only if $\omega (H) \le 1$.  If $H$ satisfies the ascending chain condition (ACC) on divisorial ideals (in particular, if $H$ is Krull or $H=R^\bullet$ where $R$ is a Noetherian domain), then
$\omega (H, u) < \infty$ for all $u \in \mathcal A (H)$ \cite[Theorem 4.2]{Ge-Ha08a}.

\medskip

Roughly speaking, for an element $u$ in an atomic monoid $H$, the tame degree $\mathsf t (H,u)$ is the maximum of $\omega (H,u)$ and the factorization lengths of $u^{-1} \prod_{\nu \in \Omega} v_{\nu}$ in $H$. For an atom $u \in H_{\red}$, the {\it local tame degree} $\mathsf t (H, u)$ is the smallest $N \in \N_0 \cup \{\infty\}$ with the following property:
\begin{enumerate}
\smallskip
\item[] For any multiple $a \in H_\red$ of $u$ and any factorization $z = v_1 \cdot \ldots \cdot v_n \in \mathsf Z (a)$ of $a$ which does not contain $u$, there is a short subproduct which is a multiple of $u$, say $v_1 \cdot \ldots \cdot v_m$, and a refactorization of this subproduct which contains $u$, say $v_1 \cdot \ldots \cdot v_m = u u_2 \cdot \ldots \cdot u_{\ell}$, such that $\max \{\ell, m \} \le N$.
 \smallskip
      \end{enumerate}
In particular, the local tame degree $\mathsf t (H, u)$ measures the distance between any factorization of a multiple $a$ of $u$ and a factorization of $a$ which contains $u$.  We denote by
      \[
      \mathsf t (H) = \sup \big\{ \mathsf t (H, u) : u \in \mathcal A (H_{\red}) \big\} \in \N_0 \cup \{\infty\}
      \]
the ({\it global}) {\it tame degree} of $H$, and we say that $H$ is tame if $\mathsf t (H) < \infty$. If $u$ is prime, then $\mathsf t (H, u) = 0$ and thus $H$ is factorial if and only if $\mathsf t (H) = 0$. In order to describe the relationship between the $\omega$-invariants and the tame degree, we observe that
\[
\begin{split}
  \omega(H,a) = \sup\big\{\, k \in \N_0 \cup \{\infty\} : &\text{ $b=u_1\cdot \ldots \cdot u_k \in aH$ with $k \in \N_0$, $u_1,\ldots,u_k \in \cA(H)$,} \\ &\text{ and $u \nmid u_i^{-1} b$ for all $i \in [1,k]$} \,\big\} \,,
\end{split}
\]
and we define the $\tau$-invariant as
\begin{equation} \label{eq:tau}
\begin{split}
  \tau(H,a) = \sup\big\{\, \min \sL(a^{-1}b) : &\text{ $b=u_1\cdot \ldots \cdot u_k \in aH$ with $k \in \N_0$, $u_1,\ldots,u_k \in \cA(H)$,} \\ &\text{ and $a \nmid u_i^{-1} b$ for all $i \in [1,k]$} \,\big\}.
\end{split}
\end{equation}

For each non-prime $u \in \cA(H)$, $\st(H,uH^{\times}) = \max\{\,\omega(H,u),\, \tau(H,u) + 1 \,\}$ (\cite[Theorem 3.6]{Ge-Ha08a}). Suppose that $H$ is half-factorial and $\sL(a) = \{\,l\,\}$ with $l \in \N_0$. Then $\sL(a^{-1}b) = \{\, k - l \,\}$ and hence $\omega(H,a) = \tau(H,a) + l$. In particular, if $u \in \cA(H)$, then $\omega(H,u) = \tau(H,u) + 1$.

\medskip

We recall that if $H_{\red}$ is finitely generated, then $H$ is tame (\cite[Theorem 3.1.4]{Ge-HK06a}). Tame monoids will be studied in Proposition \ref{3.1}, Theorem \ref{6.3}, and Theorem \ref{hnp-proj}. In Section \ref{5} we will characterize when the monoid $\mathcal V((\mathcal C)$ is tame for $\mathcal C$ a class of finitely generated modules over a commutative Noetherian local ring (Proposition \ref{5.1}). Further examples of tame monoids can be found in \cite{Ge-Ka10a, Ka14a}.
We now gather together several arithmetical finiteness properties of tame monoids.

\smallskip
\begin{proposition}[\bf Arithmetic of tame monoids] \label{2.1}~

Let $H$ be a tame monoid.
\begin{enumerate}
      \item If $H$ is not factorial, then  $2 + \sup \Delta (H) \le \mathsf c (H) \le \omega (H) \le \mathsf t (H) \le \omega (H)^2 < \infty$.

      \item There is a constant $M  \in \N_0$  such that every set of lengths $L \in \mathcal L (H)$ is an {\rm AAMP} with difference $d \in \Delta (H)$ and bound $M$.

      \item There is a  constant $M \in \mathbb N_0$ such that, for every $k \ge 2$, the union $\mathcal U_k (H)$ of sets of lengths is an {\rm AAMP} with period $\{0, \min \Delta (H) \}$ and bound $M$.
\end{enumerate}
\end{proposition}

\begin{proof}
Statement 1 follows easily from \cite[Theorem 1.6.3]{Ge-HK06a} and \cite[Section 3]{Ge-Ka10a}.
Statement 2 follows from \cite[Theorem 5.1]{Ge-Ka10a} and statement 3 follows from \cite[Theorems 4.2 and 3.5]{Ga-Ge09b}.
\end{proof}

\medskip
\noindent
{\bf Transfer homomorphisms.}
A monoid homomorphism \ $\theta \colon H \to B$ is called a \ {\it transfer homomorphism} if the following properties are satisfied:

\begin{enumerate}
\item[{\bf (T\,1)\,}] $B = \theta(H) B^\times$ \ and \ $\theta^{-1} (B^\times) = H^\times$.

\smallskip

\item[{\bf (T\,2)\,}] If $u \in H$, $b,c \in B$, and \ $\theta
(u) = bc$, then there exist $v,w \in H$  such that $u = vw$, $\theta (v) \simeq b$, and $\theta (w) \simeq c$.
\end{enumerate}

Transfer homomorphisms are a central tool in Factorization Theory and allow one to lift arithmetical results from a (simpler) monoid $B$ to the monoid $H$ (of original interest). These homomorphisms will be used throughout this manuscript (see, in particular, Propositions \ref{3.1}, \ref{monext}, and \ref{acm-factorization}). Each transfer homomorphism naturally gives rise to a homomorphism $\overline \theta$ of factorization monoids that extends $\theta$. Let $\theta \colon H \to B$ be a transfer homomorphism of atomic monoids. Then $\theta$ induces a homomorphism $\overline \theta \colon \mathsf Z (H) \to \mathsf Z (B)$ satisfying $\overline \theta (u H^{\times}) = \theta (u) B^{\times}$ for all $u \in \mathcal A (H)$.

\medskip

We now recall how various factorization-theoretic invariants are preserved by transfer homomorphisms. Let $\theta:H \rightarrow B$ and $\overline \theta: \mathsf Z(H) \rightarrow \mathsf Z(B)$ be as above. For $a \in H$, we denote by $\mathsf c (a, \theta)$ the smallest $N \in \N_0 \cup \{\infty\}$ with the following property:

\smallskip

\begin{enumerate}
\item[]
If $z,\, z' \in \mathsf Z_H (a)$ and $\overline \theta (z) =
\overline \theta (z')$, then there exist $k \in \N_0$ and factorizations $z=z_0,
\ldots, z_k=z' \in \mathsf Z_H (a)$ such that \ $\overline \theta (z_i) = \overline \theta (z)$ and \ $\mathsf d (z_{i-1}, z_i) \le N$ for all $i \in [1,k]$; that is,
$z$ and $z'$ can be concatenated by an $N$-chain in the fiber
\ $\mathsf Z_H (a) \cap \overline \theta ^{-1} (\overline \theta
(z))$.
\end{enumerate}
Now
\[
\mathsf c (H, \theta) = \sup \{\mathsf c (a, \theta) : a \in H \} \in \N_0 \cup \{\infty\}
\]
denotes the catenary degree in the fibres of $\theta$. The next lemma summarizes what will be needed in the sequel (a proof can be found in \cite[Section 3.2]{Ge-HK06a}).

\smallskip
\begin{lemma} \label{2.2}
Let $\theta \colon H \to B$ be a transfer homomorphism.
\begin{enumerate}
\item $H$ is atomic if and only if $B$ is atomic.

\smallskip
\item For all $a \in H$, $a$ is an atom of $H$ if and only if $\theta (a)$ is an atom of $B$.

\smallskip
\item Suppose that $H$ is atomic.
      \begin{enumerate}
      \item
      For all $a \in H$, $\mathsf L_H (a) = \mathsf L_{B} (\theta (a))$. In particular,
      \smallskip
      \begin{enumerate}
      \item $\mathcal L (H) = \mathcal L (B)$ and $\Delta (H) = \Delta (B)$, and
      \smallskip
      \item  $\mathcal U_k (H) = \mathcal U_k (B)$, $\rho_k (H) = \rho_k (B)$, and $\lambda_k (H) = \lambda_k (B)$ for each $k \in \mathbb N.$
      \end{enumerate}
      \smallskip
      \item For all $a \in H$, $\mathsf c_B ( \theta (a)) \le \mathsf c_H (a) \le \max \{ \mathsf c_B ( \theta (a)), \mathsf c (a, \theta)\}$. In particular,
      \smallskip
      \begin{enumerate}
      \item [] $\mathsf c (B) \le \mathsf c (H) \le \max \{ \mathsf c (B), \mathsf c (H, \theta)\}.$
      \end{enumerate}
      \end{enumerate}
\end{enumerate}
\end{lemma}

\medskip
\noindent
{\bf Generalized Krull monoids.} Let $H$ and $D$ be monoids. A monoid homomorphism $\varphi \colon H \to D$   is called
\begin{itemize}
\smallskip
\item a  {\it divisor homomorphism} if $\varphi(a)\mid\varphi(b)$ implies that $a \mid b$  for all $a,b \in H$.

\smallskip
\item  a {\it divisor theory} (for $H$) if $D = \mathcal F(P)$ for some set $ P\subset D$, $\varphi$ is a divisor homomorphism, and, for every
$a \in \mathcal{F}( P)$, there exists a finite nonempty subset $X \subset H$ such that $a = \gcd \bigl( \varphi(X) \bigr)$.

\smallskip
\item {\it cofinal} if for every $\alpha \in D$ there exists $a \in H$ such that $\alpha \t \varphi (a)$.
\end{itemize}
If $H \subset D$ is a submonoid, then $H \subset D$ is said to be {\it saturated} (or {\it cofinal}) if the embedding $H \hookrightarrow D$ is a divisor homomorphism (or is cofinal). The monoid $H$ is called a
\begin{itemize}
\item {\it rational generalized Krull monoid} if there is a divisor homomorphism $\varphi \colon H \to \mathcal F_{\rat} (P)$ for some set $P$.

\smallskip
\item {\it Krull monoid} if there is a divisor homomorphism $\varphi \colon H \to \mathcal F (P)$ for some set $P$.
\end{itemize}
We note that every Krull monoid is a rational generalized Krull monoid and that a monoid $H$ is a (rational generalized) Krull monoid if and only if $H_{\red}$ is a (rational generalized) Krull monoid. Generalized Krull monoids and domains have been studied in \cite[Section 5]{Ge-Mo95}, \cite[Chapter 22]{HK98}, and \cite{Ch-HK-Kr02}. Specifically, \cite[Proposition 2]{Ch-HK-Kr02} guarantees that the definition of rational generalized Krull monoids we have given above coincides with the usual one.

\medskip

We note that a monoid is Krull if and only if it is completely integrally closed and $v$-noetherian and that every Krull monoid is an \FF-monoid and has a divisor theory. Let $H$ be a Krull monoid and let $\varphi \colon H \to D = \mathcal F (P)$ be a cofinal divisor homomorphism. We call $\mathcal{C}(\varphi)=\mathsf q (D)/ \mathsf q (\varphi(H))$ the {\it class group} of $\varphi $ and use additive notation for this group. For $a \in \mathsf q(D)$, we denote by $[a] = [a]_{\varphi} = a \,\mathsf q(\varphi(H)) \in \mathsf q (D)/ \mathsf q (\varphi(H))$ the class containing $a$. Since $\varphi$ is a cofinal divisor homomorphism, $\mathcal{C}(\varphi) = \{[a] : a \in D \}$ and $\varphi(H)= \{a \in D\, :\, [a]=[1]\}$. The set
\[
G_P = \{[p] = p \mathsf q (\varphi(H))\, :\, p \in P \} \subset \mathcal{C}(\varphi)
\]
of classes containing prime divisors plays a crucial role in arithmetic computations (see Proposition \ref{3.1}) and we have $[G_P] = \mathcal{C}(\varphi)$. If $\varphi$ is a divisor theory, then $\varphi$ and the class group $\mathcal C (\varphi)$ are unique up to isomorphism. Since $\mathcal C (\varphi)$ depends only on $H$, we denote it by  $\mathcal C (H)$ and call it the {\it class group of} $H$. Moreover, a reduced Krull monoid $H$ with divisor theory $H \hookrightarrow \mathcal F(P)$ is uniquely determined up to isomorphism by its \emph{characteristic} $(G, (m_g)_{g \in G})$ where $G$ is an abelian group together with an isomorphism $\Phi: G \rightarrow \mathcal C(H)$ and where $(m_g)_{g\in G}$ is a family of cardinal numbers $m_g=| P \cap \Phi(g)|$ (see \cite[Theorem 2.5.4]{Ge-HK06a}). We consider the characteristic of a monoid of modules $\mathcal V(\mathcal C)$ that is Krull in both Sections \ref{5} and \ref{7}.

\medskip

A domain $R$ is a rational generalized Krull domain if and only if $R^{\bullet}$ is a rational generalized Krull monoid (for recent work on these kinds of domains, see \cite{Li12b, Pa-Te10a}). A $v$-Marot ring (and, in particular, a domain) $R$ is a Krull ring if and only if its multiplicative monoid of regular elements $R^{\bullet}$ is a Krull monoid (\cite[Theorem 3.5]{Ge-Ra-Re14c}), and we set $\mathcal C (R) = \mathcal C (R^{\bullet})$. If $R$ is a Dedekind domain, then $\mathcal C (R)$ is the Picard group of $R$.

\medskip

We now introduce a Krull monoid of  combinatorial flavor, the monoid $\mathcal B (G_0)$ of zero-sum sequences over a subset $G_0$ of an abelian group $G$. As previously mentioned, this monoid will play a crucial role in arithmetical investigations of general Krull monoids. Section \ref{4} provides a detailed study of $\mathcal B(G_0)$ in case of subsets $G_0$ of finitely generated free abelian groups.
Let $G$ be an additive abelian group, $G_0 \subset G$ a subset, and $S = g_1 \cdot \ldots \cdot g_l \in \mathcal F (G_0)$. We call $\sigma (S) = g_1+ \cdots + g_l \in G$ the sum of $S$, and we  define, for $k \in \N$,
\ber\nn \Sigma
(S) &=& \Bigl\{ \sum_{i \in I} g_i :\;\emptyset \ne I \subset
[1,\ell] \Bigr\} \ \subset \ G
\,,
 \\ \nn \Sigma_k(S) &=& \Bigl\{ \sum_{i \in I} g_i :\; I \subset
[1,\ell] ,\,|I|=k\Bigr\} \ \subset \ G \quad \text{and} \quad \Sigma_{\le k} (S) = \bigcup_{\nu \in [1, k]} \Sigma_{\nu} (S) .\eer
If $(e_1, \ldots, e_r)$ is a basis of $G$, then the set $G_r^+$ of nonzero vertices of the hypercube satisfies
\[
G_r^+ = \Sigma (e_1 \bdot \ldots \bdot e_r) \,.
\]
This set will be thoroughly studied in Sections \ref{4} and \ref{5}.

\medskip

For a map $\varphi \colon G \to G'$ between two abelian groups $G$ and $G'$, we define $\varphi (S) = \varphi (g_1) \bdot \ldots \bdot \varphi (g_\ell)$. Also, for $S \in \mathcal F (G_0)$, we set $-S = (-g_1) \cdot \ldots \cdot (-g_\ell)$. Clearly
\[
\mathcal B (G_0) = \{ S \in \mathcal F (G_0) : \sigma (S) = 0 \} \subset \mathcal F (G_0)
\]
 is a submonoid of $\mathcal F(G_0)$. Moreover, since the inclusion $\mathcal B (G_0) \hookrightarrow \mathcal F (G_0)$ is a divisor homomorphism, $\mathcal B (G_0)$ is a Krull monoid. For arithmetical invariants $* (\cdot)$, as defined previously, we write (as it is usual) $* (G_0)$ instead of $* ( \mathcal B (G_0))$. In particular, $\mathcal A (G_0)$ denotes the set of atoms of $\mathcal B (G_0)$, $\Delta (G_0)$ denotes the set of distances of $\mathcal B (G_0)$, and so forth.
Note that the atoms of $\mathcal B (G_0)$ are precisely the minimal zero-sum sequences over $G_0$ --- those zero-sum sequences having no proper subsequence that is also a zero-sum sequence --- and we denote by
\[
\mathsf D (G_0) = \sup \{ |U| : U \in \mathcal A (G_0) \} \in \mathbb N_0 \cup \{\infty\}
\]
the {\it Davenport constant} of $G_0$, a central invariant in zero-sum theory (\cite{Ge09a, Sm10a, Ge-Li-Ph12, Gr13a}). The following lemma highlights the close connection between the arithmetic of a general Krull monoid and the arithmetic of the associated monoid of zero-sum sequences. A proof can be found in \cite[Theorems 3.4.2 and  3.4.10]{Ge-HK06a}.

\smallskip
\begin{proposition} \label{3.1}
Let $H$ be a Krull monoid, $\varphi \colon H \to D = \mathcal F (P)$
a cofinal divisor homomorphism, $G = \mathcal C (\varphi)$ its class
group, and $G_{P} \subset G$ the set of classes containing prime
divisors. Let $\widetilde{\boldsymbol \beta} \colon D \to \mathcal F
(G_{ P})$ denote the unique homomorphism defined by
$\widetilde{\boldsymbol \beta} (p) = [p]$ for all $p \in P$.
\begin{enumerate}
\item The homomorphism $\boldsymbol \beta = \widetilde{\boldsymbol \beta} \circ \varphi \colon H \to \mathcal B
      (G_{P})$ is a transfer homomorphism. Moreover,
      \begin{itemize}
      \item [] $\mathsf c (H, \boldsymbol \beta ) \le 2$, $\mathsf c (G_P) \le \mathsf c (H) \le \max \{\mathsf c (G_P), 2\}$, and $\mathsf c (H) \le \mathsf D (G_P)$.
      \end{itemize}
\smallskip
\item If $G_P$ is finite, then $\mathcal A (G_P)$ is finite and hence $\mathsf D (G_P) < \infty $.

\smallskip
\item If $\mathsf D (G_P) < \infty$, then both $H$ and $\mathcal B (G_P)$ are tame.
\end{enumerate}
\end{proposition}

\medskip
\noindent
{\bf Monoids of modules.} Let $R$ be a ring and let $\mathcal C$ be a small class of right $R$-modules. That is, $\mathcal C$ has a set $\mathcal V (\mathcal C)$ of isomorphism class representatives, and for any $M \in \text{\rm Ob}(\mathcal C)$, we denote by $[M]$ the unique element of $\mathcal V (\mathcal C)$ isomorphic to $M$. In more technical terms we suppose that the full subcategory $\mathcal C$ of Mod-$R$ is skeletally small. Suppose that $\mathcal C$ is closed under finite direct sums, direct summands, and isomorphisms. Then $\mathcal V (\mathcal C)$ is a reduced commutative semigroup with operation $[M] + [N] = [M \oplus N]$,
 and all information about direct-sum decomposition of modules in $\mathcal C$ can be studied in terms of factorizations in the semigroup $\mathcal V (\mathcal C)$.

Suppose that $\mathcal V (\mathcal C)$ is a monoid and let $\mathcal C'$ be a subclass of $\mathcal C$ which is closed under isomorphisms. Then $\mathcal C'$ is closed under finite direct sums, direct summands, and isomorphisms if and only if $\mathcal V (\mathcal C') \subset \mathcal V (\mathcal C)$ is a divisor-closed submonoid. For a module $M$ in $\mathcal C$ we denote by $\add (M)$ the class of $R$-modules that are isomorphic to direct summands of direct sums of finitely many copies of $M$. Then $\mathcal V (\add (M))$ is the smallest divisor-closed submonoid generated by $[M] \in \mathcal V(\mathcal C)$. The class $\mathcal C_{\proj}$ of finitely generated projective right $R$-modules is of  special importance (\cite{Dr69a} and \cite[Section 2.2]{Fa12a}) and will be considered in Sections \ref{6} and \ref{7}.

\medskip
\begin{proposition} \label{3.2}
Let $R$ be a ring and $\mathcal C$ a small class of right  $R$-modules which is closed under finite direct sums, direct summands, and isomorphisms.
\begin{enumerate}
\item If $\End_R (M)$ is semilocal for each $M$ in $\mathcal C$, then $\mathcal V (\mathcal C)$ is a Krull monoid.
\smallskip
\item  If there exists $M$ in $\mathcal C$ such that $\End_R (M)$ is semilocal, then $\mathcal V (\add (M))$ is a finitely generated Krull monoid. Conversely, if $\mathcal V (\mathcal C)$ is a finitely generated monoid, then there is some $M$ in $\mathcal C$ such that $\mathcal V (\mathcal C) = \mathcal V (\add (M))$.

\smallskip
\item If $R$ is semilocal, then $\mathcal V (\mathcal C_{\proj})$ is a finitely generated Krull monoid. In addition, if $R$ is commutative, then $\mathcal V (\mathcal C_{\proj})$ is Krull if and only if $\mathcal V (\mathcal C_{\proj})$ is free abelian.

\end{enumerate}
\end{proposition}

\begin{proof}
See \cite[Theorem 3.4]{Fa02} for the proof of 1.

\smallskip
We now consider statement 2. Let $N_1$ and $N_2$ be right $R$-modules and recall that $\End (N_1 \oplus N_2)$ is semilocal if and only if $\End (N_1)$ and $\End (N_2)$ are semilocal (see the introduction of \cite{Fa-He04}). Thus $\End (N)$ is semilocal for all $N$ in $\add (M)$ and hence $\mathcal V (\add (M))$ is a Krull monoid by 1.
Moreover, \cite[Corollary 4.11]{Fa98a} implies that $\mathcal V (\add (M))$ is finitely generated. Conversely, suppose that $\mathcal V (\mathcal C)$ is a finitely generated monoid. Let $M_1, \ldots, M_t$ be indecomposable right $R$-modules such that $\mathcal A(\mathcal V (\mathcal C))=\{[M_1], \ldots, [M_t]\}$. Then $M = M_1 \oplus \cdots \oplus M_t$ has the required property.

\smallskip
Since $\mathcal V(\mathcal C_{\proj}) = \mathcal V(\add(R_R))$ and $\End(R_R) \cong R$, the first claim in 3 follows immediately from 2. The second statement follows from \cite[Theorem 4.2]{Fa06a}.
\end{proof}

\medskip
\begin{remarks} \label{3.3}
  \begin{remenumerate}[leftmargin=0,itemindent=2em,itemsep=0.2em]
\item[]
\item Statement 1 of Proposition \ref{3.2} is a local condition guaranteeing that $\mathcal V (\mathcal C)$ is a Krull monoid and such examples will be considered in Section \ref{5}. It is well-known that the condition ''$\End_R (M)$ is semilocal" is stronger than the condition that ''$\mathcal V (\add (M))$ is a Krull monoid". Indeed, if $H$ is any Krull monoid, then all divisor-closed submonoids are Krull as well. However, even if all divisor-closed submonoids which are generated by a single element are Krull monoids, then the monoid $H$ can fail to be Krull (see \cite{Re14a} for counterexamples).

\item In Section \ref{6} we will see that since a finitely primary monoid $H$ is Krull if and only if $H$ is factorial, the monoid of modules $\mathcal V (\mathcal C_{\proj})$ in this setting shares the same property.

\item In Section \ref{7} we will study finitely generated monoids of modules which are not Krull.

\item Every reduced Krull monoid is isomorphic to a monoid of finitely generated projective modules $\mathcal V (\mathcal C_{\proj})$ for which $\End_R (M)$ is semilocal for all $M \in \mathcal C_{\proj}$ (\cite[Theorem 2.1]{Fa-Wi04}).
  For further realization results see, for example \cite{Wi01}, \cite[Section 9]{Fa12a}, \cite{Le-Wi12a}, and \cite[Theorem 2.7.14]{Ge-HK06a}.

\item In \cite{Fa06b}, Facchini presents an extensive list of classes of modules having semilocal endomorphism rings, and thus a list of monoids of modules $\mathcal V(\mathcal C)$ which are Krull.
\end{remenumerate}
\end{remarks}

\bigskip
\section{Zero-Sum Theory in finitely generated free abelian groups} \label{4}
\bigskip

In this section we study the Davenport constant $\mathsf D(G_0)$ when $G_0$ is a subset of a finitely generated free abelian group $G$. Although the results we obtain are interesting in their own right, we have in mind an application to the study of invariants of certain Krull monoids, and in particular certain monoids of modules $\mathcal V(\mathcal C)$ that we study in Section \ref{5}.

Let $G$ be an additive abelian group and let $G_0 \subset G$ be a subset.  A sequence over $G_0$ will mean a finite sequence of terms from $G_0$ which is unordered and where repetition of terms is allowed.  Zero-Sum Theory studies such sequences, their sets of subsequence sums, and their structure under extremal conditions (see \cite{Ga-Ge06b, Ge-Ru09, Gr13a}). Much of the focus has been on sequences over finite abelian groups, but --- motivated by applications in various fields --- sequences over infinite abelian groups have recently found more attention (see \cite{Ge-Gr-Sc-Sc10, Si14a} for examples of recent work).

Our goal in this section is to present, and expand upon, known methods from matroid theory that can be used to find upper bounds for the Davenport constant $\mathsf D(G_0)$ for subsets $G_0$ of finitely generated free abelian groups. These methods from matroid theory were used by  Sturmfels with regards to toric varieties \cite{MR1363949}. Such toric varieties are, in turn, related to primitive partition identities which were further studied by Diaconis, Graham, and Sturmfels \cite{Di-Gr-St93} by combining the matroid-theoretic methods with methods from the geometry of numbers. Although no mention of either zero-sums or the Davenport constant was made in their work, Freeze and Schmid \cite[Theorem 6.5]{Fr-Sc10a} along with Geroldinger and Yuan \cite[Section 3]{Ge-Yu13a} made attempts to use these methods in order to study zero-sum problems over finite abelian groups. We continue this program with a goal of studying zero-sum problems over infinite abelian groups.

\medskip

In order to make these methods from matroid theory --- written using very different notation and terminology than is typical in zero-sum theory --- more generally available for the study of zero-sum problems over finitely generated free abelian groups, we  develop the basic theory here from scratch. In addition, we modify the arguments so that they extend to the case when $G_0$ is non-symmetric, a situation not included in the original formulations from \cite{Ro69a} and \cite{MR1363949}.

\smallskip
Let $G$ be a finitely generated free abelian group, $Q$ a $\Q$-vector space with $G \subset Q$, and $G_0$ a subset $G$. The elements $S \in \mathcal F_{\rat} (G_0)$ will be called {\it rational sequences} over $G_0$. If $S = \prod_{g \in G_0} g^{\mathsf v_g (S)} \in \mathcal F_{\rat} (G_0)$, then $\sigma (S) = \sum_{g \in G_0} \mathsf v_g (S) g  \in Q$ is called the {\it sum } of $S$. Then
\[
\mathcal B_{\rat} (G_0) = \{ S \in \mathcal F_{\rat} (G_0) : \sigma (S) = 0 \} \subset \mathcal F_{\rat} (G_0)
\]
is a saturated submonoid of $\mathcal F_{\rat}(G_0)$, and hence $\mathcal B_{\rat} (G_0)$ is a rational generalized Krull monoid. This monoid is clearly reduced and, since for any nonidentity  $B \in \mathcal B_{\rat} (G_0)$ we have $B = B^{1/2} B^{1/2}$, $\mathcal B_{\rat} (G_0)$ has no atoms. Moreover, $\mathcal B_{\rat} (G_0)$ is a generalized block monoid as introduced in \cite[Example 4.10]{Ge-HK94}. The elements of $\mathcal B_{\rat} (G_0)$ will be called {\it rational zero-sum sequences} over $G_0$. Obviously, for any rational (zero-sum) sequence $S$ there exists an integer $n \in \N$ such that $S^n$ is an ordinary (zero-sum) sequence. Thus $\mathcal B (G_0) \subset \mathcal B_{\rat} (G_0)$ and $\mathcal F (G_0) \subset \mathcal F_{\rat} (G_0)$ are root extensions (see \cite[Section 5]{Ch-HK-Kr02}).
If $G_0 \subset G_0'$ are two subsets of $G$, we assume $\mathcal F_{\rat}(G_0) \subset \mathcal F_{\rat}(G_0')$ and $\mathcal B_{\rat}(G_0) \subset \mathcal B_{\rat}(G_0')$. Specifically, we make this assumption when we consider $G_0' = G_0 \cup -G_0$.
Given a sequence $S \in \mathcal F_{\rat}(G_0)$ and $g\in G_0$, we tacitly use terms such as $\vp_{-g}(S)$ and $-S$, where we interpret $-g$ and $-S$ as elements of $\mathcal F_{\rat}(G_0 \cup -G_0)$. In particular, $\vp_{-g}(S) = 0$ if $-g \not \in G_0$.

\medskip
We begin with the central construction of a partion of $G_0$ and of an epimorphism from $\mathcal F_{\rat} (G_0)$ to some $\Q$-vector space. This construction will remain valid throughout Section \ref{4}.

\medskip

Partition the nonzero elements of $G_0$ as $G_0\setminus\{0\}=G_0^+\cup G_0^-$, where the elements of $G_0$ have been distributed so that if $g,-g \in G_0\backslash\{0\}$, then $g$ and $-g$ neither both occur in $G_0^+$ nor both occur in $G_0^-$.
Furthermore, choose such a partition so that $G_0^+$ is maximal, that is, such that $-G_0^-\subset G_0^+$. There may, of course, be many ways to achieve this, but we choose one such partition and fix it for the remainder of Section \ref{4}.

Let $$\varphi':\mathcal F_{\rat}(G_0)\rightarrow \mathcal F_{\rat}(G_0^+\setminus (-G_0^{-}))\times \mathsf q(\mathcal F_{\rat}(-G_0^{-}))\subset
\mathsf q(\mathcal F_{\rat}(G_0^+))$$
be the unique epimorphism satisfying $\varphi'(g)=g$ for each $g\in G_0^+$,  $\varphi'(-g)=g^{-1}$ for each $-g\in G_0^{-}$, and $\varphi'(0)=1$ provided $0 \in G_0$.  The arguments of this section are based in the geometry of the sets and thus can be phrased more naturally using vector notation. To translate, we use the canonical isomorphism between $\mathsf q(\mathcal F_{\rat}(G_0^+))$ and $\Q^{(G_0^+)}$ which maps $g$ to $\mathsf e_g$ for each $g \in G_0^+$ and where $(\mathsf e_g)_{g \in G_0^+}$ denotes the standard basis of $\Q^{(G_0^+)}$. Composing $\varphi'$ with this isomorphism, we obtain an epimorphism
$$\varphi:\mathcal F_{\rat}(G_0)\overset{\varphi'}{\rightarrow} \mathcal F_{\rat}(G_0^+\setminus (-G_0^{-}))\times \mathsf q(\mathcal F_{\rat}(-G_0^{-}))\rightarrow \Q_{\geq 0}^{(G_0^+\setminus (-G_0^{-}))}\oplus \Q^{(-G_0^{-})} \subset \Q^{(G_0^+)}$$ satisfying
$$\varphi\left(\prod_{g\in G_0}g^{\vp_g(S)}\right)=\Summ{g\in G_0^+} \big(\vp_g(S)-\vp_{-g}(S) \big)\mathsf e_g.$$ The sequences $S\in\ker(\varphi) \cap \mathcal F (G_0)$ are precisely those zero-sum sequences over $G_0$ having a factorization into zero-sum subsequences of length at most $2$, and an arbitrary rational sequence from $\ker(\varphi)$ is nothing more than a rational power of such a sequence. The purpose of this construction is to create a setting where we can first apply methods from linear algebra and the geometry of $\Q$-vector spaces to study $\varphi(S)$, and then to apply the results we obtain to the original sequence $S\in \mathcal F_{\rat} (G_0)$.

\medskip
For each rational sequence $S\in \mathcal F_{\rat}(G_0)$ we define the signed support of $S$ as $$\supp^+(S)=\{g\in G_0 \cup -G_0:\;\vp_g(S)-\vp_{-g}(S)\neq 0\}.$$  Observe that for each  $S\in \mathcal F(G_0)$, we have $$\supp^+(S)=\{g\in G_0 \cup -G_0:\varphi'(g)\in \supp(\varphi'(S))\}$$ and
$$\supp^+(S)=\supp^+(-S)\subset \supp(S)\cup -\supp(S)\subset G_0\cup -G_0.$$
For $S\in \Ff(G_0)$ we define \be\label{R-def}R=0^{\vp_0(S)}\prod_{g\in G^+_0}(g (-g))^{\min\{\vp_g(S),\,\vp_{-g}(S)\}}\in \Ff(G_0),\ee where $\vp_{-g}(S)=0$ whenever $-g\notin G_0$. Then $R\mid S$ and we set $S'=R^{-1} S\in \Ff(G_0)$. It is then easily noted that
$\supp^+(S)=\supp^+(S')$, that $\varphi(S)=\varphi(S')$, that  $\varphi(R)=\mathbf 0 \in \Q^{(G_0^+)}$, and for each $g\in G_0$, $g$ and $-g$ are not both contained in $\supp(S')$ for any $g\in G_0$; that is, either $\vp_g(S')=0$ or $\vp_{-g}(S')=0$.
Note that the latter condition is equivalent to $$\supp(S')\cap -\supp(S')=\emptyset.$$
In particular, $0\notin \supp(S')$. Finally, if $S\in \Fc(G_0)$ is an ordinary sequence, then  $R,\,S'\in \Fc(G_0)$ are also ordinary sequences. These observations will be used repeatedly in the sequel.

\medskip

We  say that a (rational) zero-sum sequence $S$ is {\it elementary} if $\supp^+(S)$ is nonempty and minimal; that is, there is no (rational) zero-sum sequence $T$ with $\emptyset\neq \supp^+(T)\subsetneq \supp^+(S)$. Clearly, a rational zero-sum sequence $S$ is elementary if and only if  $S^n$ is an elementary zero-sum sequence for some positive integer $n$. An \emph{elementary} atom is an atom $U\in\mathcal B(G_0)$ which is also an elementary zero-sum sequence. We let
\[
 \mathsf D^{\mathsf{elm}}(G_0) =  \sup \{ |U| : U \in \mathcal A (G_0) \ \text{is elementary} \} \in \mathbb N_0 \cup \{\infty\}
\]
denote the supremum of the lengths of elementary atoms over $G_0$ (with the convention that $\sup \emptyset = 0$), and call $\mathsf D^{\mathsf{elm}}(G_0)$ the {\it elementary Davenport constant} over $G_0$.

In our motivating applications the group $G$ is the class group of a Krull monoid, and thus we need each of the previously defined concepts in this general abstract setting. However, it is technically simpler but no restriction to work over the additive group $\Z^r$ instead of working over an abstract finitely generated free abelian group. (Note, if $\varphi\colon G \to \Z^r$ is a group isomorphism and $G_0 \subset G$, then $\mathsf D(G_0) = \mathsf D(\varphi(G_0))$.) Therefore, for the rest of this section we suppose that
\[
G_0 \subset G = \Z^r \quad \text{where} \quad r \ge 1 \,.
\]
Since the case $G_0 \subset \{0\}$ is trivial, we further assume that the set $G_0$ contains a nonzero element of $G$.
Moreover, whenever it is convenient, we may assume that $\rk( \langle G_0 \rangle)=r$, as otherwise we could replace $\Z^r$ with $\la G_0\ra\cong \Z^{\rk( \langle G_0 \rangle)}$. We start with a sequence of basic but important properties regarding elementary zero-sum sequences.

\smallskip

\begin{lemma}\label{lem-key-tech}
Let  $r\geq 1$ and let $G_0\subset \Z^r$ be a nonempty subset. Let $S,\,T\in \Ff(G_0)$ be rational sequences such that $\supp(S)\cap \supp(-S)=\emptyset$ and $\supp(T)\cap \supp(-T)=\emptyset$. Suppose $\supp^+(T)\subset \supp^+(S)$ and $\supp(T)\cap \supp(S)\neq \emptyset$. Let $$\alpha=\min\{\vp_g(S)/\vp_g(T):\;g\in \supp(S)\cap \supp(T)\}\in \Q_{>0}.$$ Define $S'=(-T)^{\alpha} S\in \Ff(G_0\cup -G_0)$, let $$R=\prod_{g\in G^+_0}(g (-g))^{\min\{\vp_g(S'),\,\vp_{-g}(S')\}}\in \Ff(G_0\cup -G_0),$$ and set
$\tilde{S'}:=S' R^{-1}$. Then $\tilde{S'}\in\Ff(G_0)$ is a rational sequence over $G_0$ with
\be\label{toosmall}\supp(\tilde{S'})\subsetneq\supp(S)\quad\und\quad \supp^+(\tilde{S'})\subsetneq \supp^+(S).\ee
\end{lemma}

\begin{proof}
By definition, $\tilde{S'}\in \mathcal B_{\rat}(G_0\cup -G_0)$ with $\supp^+(\tilde{S'})=\supp^+(S')$.
Thus, by the definition of $S'$, \be\label{kkb}\supp^+(\tilde{S'})=\supp^+(S')\subset \supp^+(-T)\cup \supp^+(S)=\supp^+(T)\cup \supp^+(S)=\supp^+(S),\ee with the final equality obtained by using the hypothesis that $\supp^+(T)\subset \supp^+(S).$ By hypothesis, $\supp^+(T)\subset \supp^+(S)$, and thus
\be\label{retya}-g\in \supp(S)\subset G_0\quad \mbox{ for every $g\in \supp(T)\setminus \supp(S)$}.\ee

Let $g_1\in \supp(S)\cap \supp(T)$ with $\vp_g(S)/\vp_g(T)$ minimal in $\Q_{>0}$ where $g\in \supp(S)\cap \supp(T)$.
By hypothesis, $g_1$ and $-g_1$ cannot both be in $\supp(S)$ nor can they both be in $\supp(T)$. Thus, since $g_1\in \supp(T)\cap \supp(S)$, we see that $-g_1\notin \supp(T)$ and $-g_1\notin \supp(S)$. Consequently, by definition of $\alpha$, we have $$\vp_{-g_1}(-T^\alpha)=\vp_{g_1}(T^\alpha)=\vp_{g_1}(S)\quad\und\quad \vp_{g_1}(-T^\alpha)=\alpha\vp_{-g_1}(T)=0=\vp_{-g_1}(S).$$  But $g_1\in \supp^+(S)$ as $g_1\in \supp(S)$ and $-g_1\notin \supp(S)$. Thus $g_1\notin \supp^+(S')$ and the inclusion in \eqref{kkb} must be strict; that is, \be\label{kkc}\supp^+(\tilde{S'})\subsetneq \supp^+(S).\ee

From the definitions of $S'$ and $\tilde{S'}$, we have \be\label{qwela} \supp(\tilde{S'})\subset \supp(S')\subset \supp(-T)\cup \supp(S).\ee
For $g\in \supp(T)\cap \supp(S)$, the definition of $\alpha$ ensures that $\vp_{-g}(-T^\alpha)\leq \vp_g(S)$. Since $g$ and $-g$ cannot both be in $\supp(T)$ nor both be in $\supp(S)$, we also have that $\vp_g(-T^\alpha)=0=\vp_{-g}(S)$. Combining this fact with the definition of $\tilde{S'}$ we see that for every $g\in \supp(T)\cap \supp(S)$, $-g\notin \supp(\tilde{S'})$. Considering this fact along with \eqref{retya} and \eqref{qwela}, it follows that \be\label{wonderstorea}\supp(\tilde{S'})\subset \supp(S)\subset G_0.\ee
In particular, we now know that $\tilde{S'}\in \Ff(G_0)$. It remains to show that the inclusion in \eqref{wonderstorea} is strict. But for this fact we need only recall that $g_1\in \supp(S)$ and that $g_1\notin \supp^+(\wtilde{S'})=\supp(\wtilde{S'})\cup -\supp(\wtilde{S'}).$ Indeed, this last equality follows immediately since $\supp(\wtilde{S'})$ contains at most one of $g$ and $-g$ for every $g\in G_0$, a fact that follows from the definition of $\supp(\wtilde{S'})$ and the observations before Lemma \ref{lem-key-tech},
\end{proof}

\smallskip
\begin{lemma}\label{prop-rational-multiple}
Let $r\geq 1$ and let $G_0\subset \Z^r$ be a nonempty subset. If $S,\,T\in \mathcal B_{\rat}(G_0)$ are both elementary rational zero-sum sequences with $\supp^+(S)=\supp^+(T)$, then either $$S R_T^\alpha= T^\alpha R_S \quad\mbox{ or }\quad  S (-R_T)^\alpha=-T^\alpha R_S\quad\mbox{ for some positive $\alpha\in \Q_{>0}$},$$ where $R_S\mid S$ and $R_T\mid T$ are the respective maximal length  rational zero-sum subsequences of $S$ or $T$ with $\varphi(R_T)=\mathbf 0$ and $\varphi(R_S)=\mathbf 0$ (defined in \eqref{R-def}).

\smallskip
In particular, if $S,\,T\in \mathcal B(G_0)$ are elementary zero-sum sequences with common signed support,
then there exist relatively prime $m,\,n\in \mathbb N$ such that either $$S^n R_T^m= T^m R_S^n\quad\mbox{ or }\quad  S^n (-R_T)^m= (-T)^m R_S^n,$$ where $R_S\mid S$ and $R_T\mid T$ are the respective maximal length zero-sum subsequences of $S$ or $T$ having a factorization into zero-sum subsequences each of length at most $2$.
\end{lemma}

\begin{proof}
Let $S,\,T\in \mathcal B_{\rat}(G_0)$ be elementary rational zero-sum sequences with common signed support $\supp^+(S)=\supp^+(T)$. In view of the observations before Lemma \ref{lem-key-tech} we may, without loss of generality, assume that $R_S$ and $R_T$ are trivial (the general case follows easily from this special case). In particular, \be\label{whispy}\mbox{$\supp(S)\cap \supp(-S)=\emptyset\,\text{ and } \supp(T)\cap \supp(-T)=\emptyset$}.\ee
Moreover, since $S$ and $T$ are elementary zero-sum sequences, $\supp^+(S)=\supp^+(T)\not=\emptyset$.

First suppose that $\supp(S)\cap \supp(T)\neq \emptyset$. Then, since $\supp^+(S)=\supp^+(T)$, we can apply Lemma \ref{lem-key-tech}. Let $\alpha\in \Q_{>0}$, $S'\in \Ff(G_0\cup -G_0)$, and $\tilde{S'}\in \Ff(G_0)$ be as in Lemma \ref{lem-key-tech}. Since $S$, $T$, and $R$ are each rational zero-sum sequences, it follows from the definition of $\tilde{S'}$ that $\tilde{S'}$ is also a rational zero-sum sequence. Thus $\tilde{S'}$ will, in view of \eqref{toosmall}, contradict that $S$ is an elementary rational zero-sum sequence unless $\supp^+(\tilde{S'})=\emptyset$. However, by definition,  $\supp^+(\tilde{S'})=\emptyset$ is only possible if $\tilde{S'}$ is trivial, in which case $(-T^\alpha) S=S'=R$. But this implies, in view of \eqref{whispy} and $\supp^+(T)=\supp^+(S)$, that $$\vp_g(T^\alpha)=\vp_g(S)\quad\mbox{ for all $g\in G_0$}.$$
Therefore $S=T^\alpha$ as desired.

We now assume that $\supp(S)\cap \supp(T) = \emptyset$. In this case, since $\supp^+(S)=\supp^+(T)$, it follows from \eqref{whispy} that $\supp(T)=-\supp(S)$. This in turn implies that $\supp(-T) \subset G_0$ and hence $-T\in \mathcal B_{\rat}(G_0)$. Repeating the arguments of the previous paragraph using $-T$ in place of $T$, we conclude that $S=(-T)^\alpha$ as desired. The in particular statement follows easily from the general statement with $\alpha=\frac{m}{n}$, where $m,\,n\in \N$ are relatively prime.
\end{proof}

\smallskip
\begin{lemma}\label{lem-atom-uniqueness}
Let  $r\geq 1$ and let $G_0\subset \Z^r$ be a nonempty subset. If $U,\,V\in \mathcal A(G_0)$ are elementary atoms with $\supp^+(U)=\supp^+(V)$, then either $U=V$ or $U=-V$.
\end{lemma}

\begin{proof}
Since $U$ and $V$ are both elementary, their signed support must be nontrivial. Moreover, since $U$ and $V$ are atoms, neither $U$ nor $V$ has a nontrivial zero-sum subsequence of length one or two. Therefore,
 applying Lemma \ref{prop-rational-multiple} to both $U$ and $V$, we find that either $U^n=V^m$ or $U^n=(-V)^m$ for relatively prime positive integers $m$ and $n$. Note that if the latter case holds, then $\supp(-V)=\supp(U)\subset G_0$. Thus, by replacing $V$ with $-V$ if need be (in which case the hypotheses of the theorem hold for $U$ and $-V$ and, if the conclusion holds for $U$ and $-V$, then it will hold for the original pair $U$ and $V$ as well), we may assume that $U^n=V^m$ for relatively prime positive integers $m$ and $n$. Therefore $n\vp_g(U)=m\vp_g(V)$ for all $g\in G_0$. Consequently, since $\gcd(m,n)=1$, it follows that $m\mid \vp_g(U)$ for each $g\in G_0$. But then $U={W}^m$  is a product decomposition of $U$, where $W=\prod_{g\in G_0}g^{\vp_g(U)/m}\in \Fc(G_0)$. Noting that $0=\sigma(U)=m\sigma(W)$, we conclude that $W$ is a zero-sum sequence. Since it was assumed that $U\in \mathcal A(G_0)$ is an atom, $m=1$. A similar argument shows that $n=1$. Now the relation $U^m=V^m$ gives the desired conclusion $U=V$.
 \end{proof}

\smallskip
\begin{lemma}\label{lem-step}
Let $r\geq 1$ and let $G_0\subset \Z^r$ be a nonempty subset. If $S\in \mathcal B(G_0)$ is an elementary zero-sum sequence, then
$$S=R U^\ell$$ for some $\ell\geq 1$, some elementary atom $U\in \mathcal A(G_0)$ with $\supp^+(U)=\supp^+(S)$, and some zero-sum sequence $R\in \ker (\varphi)$ that has a factorization involving only zero-sum subsequences all of length at most $2$.
\end{lemma}

\begin{proof}
Assume for the sake of contradiction that $S\in \mathcal B(G_0)$ is an elementary zero-sum sequence that fails to have the desired form and with $|S|$ minimal. By the minimality of $|S|$ and the observations made before Lemma \ref{lem-key-tech}, it follows that every nontrivial zero-sum subsequence of $S$ has length at least $3$. Therefore $\supp(S)\cap\supp(-S)=\emptyset$. Since the conclusion of the lemma holds for each atom, it must be the case that the chosen $S$ is not an atom. Therefore $$S=T_1 T_2$$ for nontrivial zero-sum subsequences $T_1$ and $T_2$. Now, since every zero-sum subsequence of $S$ has length at least $3$, the nontrivial zero-sum subsequences $T_1$ and $T_2$ must each have non-empty signed support. But this contradicts the fact that $S$ is an elementary zero-sum sequence unless we have \be\label{crikd}\supp^+(S)=\supp^+(T_1)=\supp^+(T_2).\ee
Also, $|T_1|<|S|$ and $|T_2|<|S|$ since $T_1$ and $T_2$ are both nontrivial subsequences of $S$. By \eqref{crikd}, each $T_i$ is a zero-sum sequence with $\supp^+(T_i)=\supp^+(S)$ and is consequently an elementary zero-sum sequence. Furthermore, since $|T_1|<|S|$ and $|T_2|<|S|$, the minimality of $|S|$ ensures that both $T_1$ and $T_2$ have the form stated in the conclusion of the lemma. Thus $T_1=U^m$ and $T_2=V^n$ for atoms $U,\,V\in \mathcal A(G_0)$ with $$\supp^+(U)=\supp^+(V)=\supp^+(T_1)=\supp^+(T_2)=\supp^+(S)$$ and positive integers $m$ and $n$. As $U$ and $V$ are both atoms with $\supp^+(U)=\supp^+(V)=\supp^+(S)$, it follows that $U$ and $V$ are both elementary atoms. Invoking Lemma \ref{lem-atom-uniqueness} we find that either $U=V$ or $U=-V$.

If $U=-V$, then there must exist $g\in \supp(U)$ with $-g\in \supp(V)$. In this case, since $T_1=U^m$, $T_2=V^n$, and $S=T_1 T_2$, it follows that $g (-g)\mid S$, contrary to conclusion above that $S$ does not have any zero-sum subsequence of length $1$ or $2$. Therefore $U\neq -V$, forcing $U=V$. Then $T_1=U^m$, \ $T_2=V^n=U^n$, and $S=T_1 T_2=U^{m+n}$ as desired.
\end{proof}

Combining the above results, we are now able to describe the form of elementary zero-sum sequences. We do so in the following proposition whose proof follows immediately from Lemmas \ref{lem-atom-uniqueness}  and \ref{lem-step}. Essentially, Proposition \ref{prop-char} states that if $X$ is the signed support of an elementary zero-sum sequence, then (up to sign) there is a unique atom $U$ with $\supp^+(U)=X$, and all other zero-sum sequences having signed support $X$ must have the form $U^\ell R$ or $(-U)^\ell R$ where $\ell\geq 1$ and $R\in \mathcal B(G_0)$ is a zero-sum sequence that has a factorization involving zero-sum subsequences all of length at most $2$.

\smallskip
\begin{proposition}\label{prop-char}
Let  $r\geq 1$ and let $G_0\subset \Z^r$ be a nonempty subset. If $X$ is the signed support of an elementary zero-sum sequence over $G_0$, then there exists a unique (up to sign)  atom $U\in \mathcal A(G_0)$ with $\supp^+(U)=\supp^+(-U)=X$ such that every elementary zero-sum sequence  $S$ with signed support $\supp^+(S)=X$ has the form $$S=R U^\ell\quad\mbox{ or }\quad S=R (-U)^\ell$$ where $R\in \ker(\varphi)$ is a zero-sum sequence and $\ell\geq 1$.
\end{proposition}

\smallskip
\begin{lemma}\label{lem-existance-subsupport}
Let  $r\geq 1$ and let $G_0\subset \Z^r$ be a nonempty subset. Suppose $S\in \mathcal B_{\rat}(G_0)$ is a rational zero-sum sequence with $\supp^+(S)$ nonempty. Then there exists some elementary atom $U\in\mathcal A(G_0)$ such that $\supp(U)\subset \supp(S)$.
\end{lemma}

\begin{proof}
Assume for the sake of contradiction that $S\in \mathcal B_{\rat}(G_0)$ is a counterexample with $\supp^+(S)\neq \emptyset$ minimal. By removing rational  zero-sum subsequences of the form $((-g) g)^{\min\{\vp_g(S),\,\vp_{-g}(S)\}}$ and $0^{\vp_0(S)}$ from $S$ as defined in \eqref{R-def}, we may, without loss of generality, assume that $\supp(S)\cap \supp(-S)=\emptyset$. Let $U\in \mathcal B(G_0)$ be an elementary zero-sum sequence with $$\supp^+(U)\subset \supp^+(S).$$ Note that such an elementary zero-sum sequence exists with $\supp^+(U)=X$ for any minimal nonempty subset $X\subset \supp^+(S)$  provided there exists a zero-sum sequence with signed support $X$.

In view of the observations made prior to Lemma \ref{lem-key-tech}, we may assume that $U$ has no nontrivial zero-sum sequence of length $1$ or $2$. Now, if $U=U_1\bdot\ldots\bdot U_\ell$ is a factorization as a product of atoms $U_i\in \mathcal A(G_0)$, we find that $\supp^+(U_i)=\supp^+(U)$ for each $i \in [1,l]$, else $U$ is not an elementary zero-sum sequence. Therefore, replacing $U$ by some $U_i$ as need be, we may without loss of generality assume that $U\in \mathcal A(G_0)$ is an atom.

Since $\supp^+(U)\subset \supp^+(S)$, $\supp(U)$ and $\supp(S)$ are disjoint only if $-g\in \supp(S)\subset G_0$ for every $g\in \supp(U)$. In this case, $-U\in \mathcal A(G_0)$ is also an atom over $G_0$. Thus, replacing $U$ by $-U$ (in this one scenario), we may assume that $\supp(U)\cap \supp(S)\neq \emptyset$. Since an elementary atom cannot have a nontrivial zero-sum sequence of length $1$ or $2$, we may apply Lemma \ref{lem-key-tech} with $T=U$. Now let $\alpha$, \ $S'$, \ $R$ and $\tilde{S'}$ be as in Lemma \ref{lem-key-tech}. Since $S$, \ $U$ and $R$ are each zero-sum sequences, it follows from the definition of $\tilde{S'}$ that $\tilde{S'}$ is also a zero-sum sequence. Therefore $$\tilde{S'}\in \mathcal B_{\rat}(G_0),$$ an improvement over $\tilde{S'}\in \Ff(G_0)$ given in Lemma \ref{lem-key-tech}.

If $\supp^+(\tilde{S'})=\supp^+(S')$ is nonempty, then by the strict inclusion $\supp^+(\tilde{S'})\subsetneq \supp^+(S)$ in \eqref{toosmall}, we may apply the induction hypothesis to $\tilde{S'}\in \mathcal B_{\rat}(G_0)$ and find an elementary atom
$V\in\mathcal A(G_0)$ such that $\supp(V)\subset \supp(\tilde{S'})$. The proof is then complete if one considers the other inclusion in \eqref{toosmall}. Therefore we may assume the alternative, that \be\label{fillup}\supp^+(\tilde{S'})=\supp^+(S')=\emptyset.\ee Recalling that $S'=(-U)^\alpha S$, that $\supp^+(U)\subset \supp^+(S)$, and that $\supp(S)\cap \supp(-S)=\supp(U)\cap \supp(-U)=\emptyset$, we see that \eqref{fillup} is possible only if $\supp(U)=\supp(S)$ with $\vp_{g}(U^\alpha)=\vp_g(S)$ for every $g\in \supp(U)=\supp(S)$.
\end{proof}

\medskip
Let $G_0\subset \Z^r\subset \Q^r$ be a finite subset. Note that the chosen partition of $G_0\setminus \{0\}$ gives rise to a unique partition of $G_0 \cup -G_0 \setminus \{0\}$ with $(G_0 \cup -G_0)^+ = G_0^+$. To this partition we may again associate a map $\Ff(G_0 \cup -G_0) \to \Q^{(G_0^+)}$ which we also denote by $\varphi$.
Then $\varphi\colon \Ff(G_0) \to \Q^{(G_0^+)}$ is simply the restriction of $\varphi\colon \Ff(G_0 \cup -G_0) \to \Q^{(G_0^+)}$ to $\Ff(G_0)$.
By construction, we have
\[
\varphi(\mathcal B(G_0))\subset \Z^{(G_0^+)} \subset \Q^{(G_0^+)} \,.
\]
It is easily checked that $\varphi(\mathcal B_{\rat}(G_0))$ is an additive monoid closed under multiplication by nonnegative rational numbers. The $\Q$-vector space spanned by $\varphi(\mathcal B_{\rat}(G_0))$ is then $\varphi(\mathcal B_{\rat}(G_0\cup -G_0))$, which is also the $\Q$-vector space spanned by $\varphi(\mathcal B(G_0))=\varphi(\mathcal B_{\rat}(G_0))\cap\Z^{(G_0^+)}$.

Note that vector $(\alpha_g)_{g\in G_0^+}\in \Q^{(G_0^+)}$ is an element of $\varphi(\mathcal B_{\rat}(G_0\cup -G_0))$ precisely when $$\Summ{g\in G_0^+}\alpha_gg=0.$$
Thus, if we let $M$ denote the $r\times |G_0^+|$ matrix whose columns are the vectors $g\in G_0^+\subset \Z^r\subset\Q^r$, we see that $\varphi(\mathcal B_{\rat}(G_0\cup -G_0))$ is the kernel of the matrix $M$. The set $\varphi(\mathcal B_{\rat}(G_0))$ can also be described via $M$; It is the subset consisting of all vectors $(\alpha_g)_{g\in G_0^+}\in \ker(M)$ that satisfy the following sign restrictions:
\[
  \alpha_g \ge 0\,\text{ unless }\, {-g} \in G_0.
\]
(Recall that we always assume $-G_0^-\subset G_0^+$, and thus $\alpha_g > 0$ is allowed for every $g \in G_0^+$.)
It is well known that the kernel of a matrix $M$ is the orthogonal space for the row space of the same matrix $M$, and that the row and column space of $M$ have the same dimension, which in this case is equal to the dimension of the $\Q$-vector space spanned by the vectors from $G_0\subset \Z^r\subset \Q^r$. This latter number is simply $\rk( \langle G_0 \rangle)$ and thus we conclude that
\be\label{dim-equation}|G_0^+|=\rk( \langle G_0 \rangle)+\dim_{\Q}\big( \langle \varphi(\mathcal B(G_0)) \rangle \big).\ee

\medskip
The next theorem (essentially due to Rockafellar \cite{Ro69a} in a matroid formulation) shows that elementary zero-sum sequences can be useful for decomposing a zero-sum sequence via rational product decomposition. Indeed, an arbitrary zero-sum sequence always has a product decomposition into a bounded number of rational powers of elementary atoms. It also shows that an upper bound for $\mathsf D(G_0)$ can be found using an upper bound for $\mathsf{D^{elm}}(G_0)$.
It is important to note that, even if $S\in \mathcal B(G_0)$ is an atom, $S$ may still have a nontrivial product decomposition into rational powers of elementary atoms if it is not itself elementary.

\smallskip
\begin{theorem}\label{thm-elm-atom-bound}
Let $r\geq 1$, let $G_0\subset \Z^r$ be a nonempty subset, and let $S\in \mathcal B_{\rat}(G_0)$ be a rational zero-sum sequence. Then there exist a nonnegative integer $\ell\geq 0$, elementary atoms $U_1,\ldots,U_\ell\in \mathcal A(G_0)$, positive rational numbers $\alpha_0,\,\alpha_1,\ldots,\alpha_\ell\in\Q_{>0}$, and a zero-sum sequence $R\in \ker(\varphi)$ (possibly trivial) having a factorization into zero-sum subsequences each of length $1$ or $2$. Moreover, \begin{align*}&S=R^{\alpha_0} U^{\alpha_1}_1\bdot \ldots \bdot U^{\alpha_\ell}_\ell, \quad \supp^+(U_j)\nsubseteq \supp^+\left(S \left(\prod_{i=1}^j U_i^{\alpha_i}\right)^{-1}\right)\mbox{ for all $i\in [1,\ell]$,}\quad \und\\ & \ell\leq \min\left\{\frac12|\supp^+(S)|,\,|G_0^+|-\rk( \langle G_0 \rangle)\right\}.\end{align*} In particular, $$\mathsf D(G_0)\leq \sup\{2,\;\hat \ell \;\mathsf {D^{elm}}(G_0)\}\leq \sup\{2,\,\min\{\eta,\,|G_0^+|-\rk( \langle G_0 \rangle)\}\; \mathsf {D^{elm}}(G_0)\}\leq\sup\{2,\; |G_0\setminus\{0\}|\;\mathsf {D^{elm}}(G_0)\},$$ where $\eta=\sup\{|\supp(U)|:\; U\in \mathcal A(G_0)\}$ and $\hat \ell$ denotes the supremum over all $\ell$ as $S$ ranges over $\cA(G_0)$.
\end{theorem}

\begin{proof}
We first construct the rational product decomposition for $S\in \mathcal B_{\rat}(G_0)$ and then deduce the upper bound for $\mathsf D(G_0)$. To this end, in view of the observations made prior to Lemma \ref{lem-key-tech}, we may without loss of generality assume that $S$ is nontrivial and that \be \label{thecondition} \{g\in G_0^+\, :\, g,-g \in \supp(S)\}=\emptyset. \ee It suffices to prove the theorem for such rational sequences.

Our goal is to show that there exist elementary atoms $U_1,\ldots,U_\ell\in \mathcal A(G_0)$ and positive rational numbers $\alpha_1,\ldots,\alpha_\ell\in\Q_{>0}$ such that \begin{align}\label{wolff}S=U^{\alpha_1}_1\bdot \ldots \bdot U^{\alpha_\ell}_\ell \quad \und\quad \supp^+(U_j)\nsubseteq \supp^+\left(S \left(\prod_{i=1}^j U_i^{\alpha_i}\right)^{-1}\right)\mbox{ for all $i\in [1,\ell]$}.\end{align} Before doing so, we explain how \eqref{wolff} forces the desired upper bound for $\ell$. The second condition from \eqref{wolff} says that each atom $U_j$ contains an element  $g_j\in\supp(S)$ (via the first condition of \eqref{wolff}) not contained in any of the $U_{j+1},\ldots,U_\ell$.  But then $\{g_1,\ldots,g_\ell\}\subset \supp(S)$ is a subset of cardinality $\ell$, implying $\ell\leq |\supp(S)|=\frac12|\supp^+(S)|$.
In view of \eqref{thecondition} and since $g_j\in\supp(S)$, we see that $-g_j\notin \supp(S)$ for each $j\in [1,\ell]$. As a result,
since each atom $U_j$ contains some element  $g_j$ not contained in any of the $U_{j+1},\ldots,U_\ell$,
no $g_j\in \supp^+(U_j)$ is contained in $\supp^+(U_{j+1}\bdot\ldots\bdot U_\ell)$. Now it is easily deduced that $\varphi(U_1),\ldots,\varphi(U_\ell)\in \varphi(\mathcal B(G_0))$ are linearly independent over $\Q$. Therefore $\ell\leq \dim_\Q( \langle\varphi(\mathcal B(G_0))\rangle)=|G_0^+|-\rk( \langle G_0 \rangle)$, with the equality following from \eqref{dim-equation}. Thus the desired bound for $\ell$ follows from \eqref{wolff}, and we now devote our attention to proving \eqref{wolff}. For this, suppose for the sake of contradiction that $S\in \mathcal B_{\rat}(G_0)$ is a counterexample to \eqref{wolff} with $|\supp(S)|$ minimal.

Since $S$ is nontrivial, \eqref{thecondition} forces $\supp^+(S)\neq \emptyset$. Thus,
by Lemma \ref{lem-existance-subsupport}, there exists an elementary atom $U_1\in \mathcal A(G_0)$ with $\supp(U_1)\subset \supp(S)$. Let $\alpha_1=\min\{\vp_g(S)/\vp_g(U_1): g\in \supp(U_1)\}$. Since $\supp(U_1)\subset \supp(S)$, $\alpha_1>0$. By definition, $\vp_{g}(U_1^{\alpha_1})\leq \vp_g(S)$ for every $g\in \supp(U_1)$, with equality holding for some $g=g_1\in\supp(U_1)\subset \supp(S)$ (attaining the minimum in the definition of $\alpha_1$).
Define $S'=S U_1^{-\alpha_1}$. Since $\vp_{g_1}(U_1^{\alpha_1})=\vp_{g_1}(S)$ with $g_1\in \supp(S)$, we conclude that $\supp(S')$ is a proper subset of $\supp(S)$. Indeed, $g_1\notin \supp(S')$. Since $S$ and $U_1$ are each zero-sum sequences, it follows that $S'$ is also a zero-sum sequence. Thus,
in view of the minimality of $|\supp(S)|$, we can apply the theorem to (the possibly trivial) rational sequence $\tilde{S'}$ to find $\tilde{S'}=U_2^{\alpha_2}\bdot\ldots\bdot U_\ell^{\alpha_\ell}$ for some positive rational numbers $\alpha_i$ and elementary atoms $U_i$ satisfying \eqref{wolff}. But now $S=U_1^{\alpha_1} S'=U_1^{\alpha_1} U_2^{\alpha_2}\bdot\ldots\bdot U_\ell^{\alpha_\ell}$ with \eqref{wolff} holding for $j\in [2,\ell]$. Since $g_1\in \supp(U_1)$ and $g_1\notin \supp(S')=\supp(U_2\bdot\ldots\bdot U_\ell)$, we see that \eqref{wolff} holds when $j=1$ and thus \eqref{wolff} is established. This completes the proof of the first part of the theorem. It remains to prove the upper bound for $\mathsf D(G_0)$.

Let $U\in \mathcal A(G_0)$ be an atom. We must show that $|U|$ is at most the bound given at the end of Theorem \ref{thm-elm-atom-bound}. If $|U|\leq 2$, this is clearly the case and so we may assume that $|U|\geq 3$. In this case we may assume that $\supp(U)\cap \supp(-U)=\emptyset$. Let $$U=U_1^{\alpha_1}\bdot\ldots\bdot U_{\ell}^{\alpha_\ell}$$ be the rational product decomposition of $U$ given by the first part of the theorem. In particular, each $U_i\in \mathcal A(G_0)$ is an elementary atom and each $\alpha_i\in \Q_{>0}$ is a positive rational number. We note that the corresponding sequence $R$ is trivial since $\supp(U)\cap \supp(-U)=\emptyset$. If $\alpha_i>1$ for some $i\in [1,\ell]$, then $U_i\mid U$ is a proper nontrivial zero-sum subsequence, contradicting that $U\in \mathcal A(G_0)$ is an atom. Thus $\alpha_i\leq 1$ for all $i\in [1,\ell]$. Now  $$|U|=\Sum{i=1}{\ell}\alpha_i|U_i|\leq \Sum{i=1}{\ell}|U_i|\leq \ell \mathsf D^{\mathsf{elm}}(G_0),$$ and, noting that $\supp(U)=\frac12|\supp^+(U)|$ (since $\supp(U)\cap \supp(-U)=\emptyset$), the desired bound for $|U|$ follows from the bound for $\ell$ given in the first part of the theorem.
\end{proof}

We now consider which subsets $X\subset G_0\subset \Z^r$ can be attained as the support of an elementary zero-sum sequence.  A related question asks which subsets $X\subset G_0\cup -G_0\subset \Z^r$ can be attained as the signed support of an elementary zero-sum sequence. Given any $U\in \mathcal B_{\rat}(G_0)$, we know that $U^n\in \mathcal B(G_0)$ for some $n\geq 1$. Applying Lemma \ref{prop-char}, we see that $X=\supp^+(U)$ for some elementary $U\in \mathcal B(G_0)$ is equivalent to $X=\supp^+(U)$ for some elementary $U\in \mathcal B_{\rat}(G_0)$ which in turn is equivalent to $X=\supp^+(U)$ for some elementary $U\in \mathcal A(G_0)$. The same is true when considering $X=\supp(U)$ for an elementary zero-sum sequence $U$ over $G_0$. Of course, if $X=\supp^+(U)$ for a zero-sum sequence $U$, then $X\subset G_0\cup -G_0$ is symmetric and so $X=Y\cup -Y$ for some $Y\subset G_0$ with $Y\cap-Y=\emptyset$ .The following lemma classifies the possibilities for $X$.

\smallskip
\begin{lemma}\label{lem-elementary-supp}
Let $r\geq 1$, let $G_0\subset \Z^r$ be a nonempty subset, and let $X\subset G_0\cup -G_0$ be a subset with $X=-X$. Then condition (a) holds if and only if conditions (b) and (c) both hold. If $G_0=-G_0$, then (a) and (b) are equivalent.

\begin{itemize}
\item[(a)]  $X=\supp^+(U)$ for some elementary zero-sum sequence $U\in \mathcal B(G_0)$.
\item[(b)] The elements of $X\cap G_0^+$ are linearly dependent over $\Q$, but any proper subset of $X \cap G_0^+$ is linearly independent over $\Q$.
\item[(c)] There exists a nontrivial zero-sum sequence $S\in \mathcal B(G_0)$ with $\emptyset\neq \supp^+(S)\subset X$.
\end{itemize}
In particular, if $U\in \mathcal B(G_0)$ is an elementary zero-sum sequence such that $\supp(S)\cap \supp(-S)=\emptyset$, then $|\supp(U)|\leq r+1$.
\end{lemma}

\begin{proof}
We set $Y = X \cap G_0^+$. Suppose that (a) holds. Then (c) holds and for each $g\in Y$ there exists a nonzero $\alpha_g\in \Q\setminus\{0\}$ such that \be\Summ{g\in Y}\alpha_gg=0\quad\mbox{ with $\alpha_g>0$ whenever $-g\notin G_0$}.\label{going}\ee  Since each $\alpha_g$ is nonzero, the elements of $Y$ are linearly dependent.

Now, if we assume that condition (b) does not hold, then there must be some nonempty, proper subset $Z\subsetneq Y$ such that the elements of $Z$ are linearly dependent. But then for each $g \in Z$ there exist $\beta_g\in \Q$ such that
\be\Summ{g\in Z}\beta_gg=0\label{stillgoing}\ee with not all $\beta_g$ zero.
Set $\beta_g=0$ for any $g\in Y\setminus Z$.
Suppose first that there is some $g \in Y$ with $\alpha_g > 0$ and $\beta_g \ne 0$.
Multiplying the $\beta_g$ by $-1$ if need be, we may assume $\alpha_g >0$ and $\beta_g > 0$.
In this case, let $\gamma=\min\{\alpha_g/\beta_g:\; \alpha_g>0,\;\beta_g>0,\; g\in Y\}>0$.
If $\beta_g=0$ whenever $\alpha_g > 0$, we set $\gamma = \min\{ \alpha_g/\beta_g:\; \alpha_g <0,\; \beta_g < 0,\; g \in Y\} > 0$.
Multiplying \eqref{stillgoing} by $\gamma$, we obtain
\be\label{getgoing}\Summ{g\in Z}\gamma\beta_gg=0.\ee Moreover, by the definition of $\gamma$ we see that $\gamma\beta_g\leq \alpha_g$ whenever $\alpha_g>0$ and $\beta_g>0$. Thus, if we subtract \eqref{getgoing} from \eqref{going}, the resulting coefficient $\alpha_g-\gamma\beta_g$ will be non-negative whenever $\alpha_g>0$. As a result, since $\alpha_g>0$ whenever $-g\notin G_0$ (by \eqref{going}), we see that
\be\label{flyll}\Summ{g\in Z}(\alpha_g-\gamma\beta_g)g=0\quad\mbox{with $\alpha_g-\gamma\beta_g\geq 0$ whenever $-g\notin G_0$}\ee
Furthermore, for an element  $g_1\in Y$  attaining the minimum in the definition of $\gamma$, we see that the coefficient $\alpha_{g_1}-\gamma\beta_{g_1}$ of $g_1$ in \eqref{flyll} is zero, while not all coefficients in \eqref{flyll} are zero since each $\alpha_g$ is nonzero and since at least one $\beta_g$ is zero ($Z$ is a proper subset of $Y$). Thus the $\Q$-linear relation given in \eqref{flyll} corresponds to  a nontrivial rational zero-sum sequence $V\in \mathcal B_{\rat}(G_0)$ whose support is strictly contained in the support of $U$, contradicting that $U\in \mathcal B(G_0)$ is an elementary zero-sum sequence. We can then conclude that all proper subsets of $Y$ are linearly independent, as desired.

\smallskip
Now suppose that (b) holds and that either $G_0=-G_0$ or (c) holds.
Clearly, there cannot be any zero-sum sequence $U$  over $G_0\subset \Z^r\subset \Q^r$ with $\supp(U)$ consisting of linearly independent elements over $\Q$. Thus, in order to show that (a) holds, it suffices to show that there exists some $U\in \mathcal B_{\rat}(G_0)$ with $\supp^+(U)=X$.

If (c) holds, then a nontrivial  $U\in \mathcal B_{\rat}(G_0)$ exists with $\emptyset \neq \supp^+(U)\subset X$. However, since $\supp^+(U)\cap G^+$ can only be linearly independent if $\supp^+(U)$ is empty, we have $\supp^+(U)=X$ as desired.

Next assume  $G_0=-G_0$. We need to show that there exists some $U\in \mathcal B_{\rat}(G_0)$ with $\supp^+(U)=X$. This  is equivalent to showing that there exists some nonzero $\alpha_1,\ldots,\alpha_\ell\in \Q\setminus \{0\}$ with $\Sum{i=1}{\ell}\alpha_ig_i=0$. As the elements of $Y$ are linearly dependent, there exist  $\alpha_1,\ldots,\alpha_\ell\in \Q$  not all zero with $\Sum{i=1}{\ell}\alpha_ig_i=0$. Thus, if $\alpha_i=0$ for some $i \in [1,\ell]$, this would give a dependence relation on the  elements of $Y\setminus\{g_i\}$, contradicting the hypothesis that every proper subset of $Y$ is linearly independent. Consequently, $\alpha_i\neq 0$ for all $i \in [1,\ell]$, and (a) follows as noted earlier.

\smallskip
If  $U\in \mathcal B(G_0)$ is an elementary zero-sum sequence such that $\supp(U)\cap\supp(-U)=\emptyset$, then the first part of the theorem implies that $\supp(U)\setminus \{g\}\subset \Q^r$ is a set of linearly independent vectors for any $g\in \supp(U)$.
Since any subset of vectors of size $r+1$ must be linearly dependent in $\Q^r$, the desired bound $|\supp(U)|\leq r+1$ follows.
\end{proof}

\smallskip
\begin{lemma}\label{lem-D3-existance}
Let $r\geq 1$ and let $G_0\subset \Z^r$ be a nonempty subset. Then $\mathsf D(G_0)\geq 3$ if and only if there exists an elementary atom $U\in \mathcal A(G_0)$. If this is the case, then $\mathsf D^{\mathsf{elm}}(G_0)\geq 3$ as well.
\end{lemma}

\begin{proof}
Since any elementary atom $U\in \mathcal A(G_0)$ must satisfy $|U|\geq 3$, one direction is clear. Suppose now that $\mathsf D(G_0)\geq 3$ and let $V\in \mathcal A(G_0)$ be an atom with $|V|\geq 3$. Then $\supp^+(V)$ is nonempty, in which case  Lemma \ref{lem-existance-subsupport} completes the proof.
\end{proof}

Recall that, given any $m\times n$ integer matrix $M$ with $m\leq n$, we can perform elementary row and column operations on $M$ (swapping rows/columns, multiplying a row/column by $\pm 1$, or adding an integer multiple of a row/column to another row/column) to obtain a diagonal integer matrix  $D=(d_{i,j})_{i,j}$ with $d_{1,1}\mid\cdots\mid d_{m,m}$ and $d_{i,i}\geq 0$ for all $i \in [1,m]$. The matrix $D$ is unique and is known as the Smith normal form of the matrix $M$ and the $d_{i,i}$ are called the \emph{elementary divisors} of $M$. If $g_1,\ldots,g_n\in \Z^m$ are the columns of $M$, then $\Z^m/\la g_1,\ldots,g_n\ra\cong \Z/d_{1,1}\Z\oplus\cdots\oplus \Z/d_{m,m}\Z$. Thus, when $M$ has full rank (whence $d_{m,m}\neq 0$), we have  $d_{m,m}=\exp(\Z^m/\la g_1,\ldots,g_n\ra)$. It is easily checked (and well-known) that for $j\in [1,m]$ \begin{equation*}\gcd\{\det(M'):\; \mbox{$M'$ is a $j\times j$ submatrix of $M$}\}\end{equation*} remains invariant under elementary row and column operations and is thus equal to $d_{1,1}\cdots d_{j,j}$. In particular, if $m=n$, so that $M$ is a square matrix, \be\label{invariant-eq}d_{m,m}=\frac{|\det(M)|}{\gcd\{|\det(M')|:\; \mbox{$M'$ is a $(m-1)\times (m-1)$ submatrix of $M$}\}}.\ee These results can be found in many standard textbooks dealing with linear algebra over $\Z$.

\medskip
We now turn our attention to finding bounds for $\mathsf D(G_0)$ where $G_0\subset \Z^r$. Let $M$ be a $r\times |G_0^+|$ matrix whose columns are the vectors $g\in G_0^+\subset \Z^r$. Using lattice theory and results from the Geometry of Numbers, Diaconis, Graham, and Sturmfels \cite{Di-Gr-St93} showed that
\begin{equation*} \mathsf D(G_0)\leq (2r)^r(r+1)^{r+1}\max\{|\det(M')|:\; \mbox{$M'$ is a $r\times r$ submatrix of $M$}\}\end{equation*} when $G_0$ is finite with full rank $\rk(\la G_0\ra)=r$. However, when $|G_0|$ is not terribly large, Theorem \ref{thm-elm-atom-bound} can be used to obtain tighter bounds.
To do so, we need to be able to bound $\mathsf D^{\mathsf{elm}}(G_0)$ and to do this we consider an  argument of  Sturmfels \cite[Chapter 4]{MR1363949} which, when  combined with additional results, allows us to give a linear algebraic description of  $\mathsf D^{\mathsf{elm}}(G_0)$.
In order to state the next theorem we first need the following definition. For a collection of $r+1$ integer vectors $g_1,\ldots,g_{r+1}\in \Z^r$, we define
$$\Delta(g_1,\ldots,g_{r+1})=
\frac{
\Sum{i=1}{r+1}|\det(g_1,\ldots,g_{i-1},g_{i+1},\ldots,g_{r+1})|}
{\gcd\{|\det(g_1,\ldots,g_{i-1},g_{i+1},\ldots,g_{r+1})|: \; i\in [1,r+1]\}}\in \N_0,$$ where $\Delta(g_1,\ldots,g_{r+1})=0$ if $\rk(\la g_1,\ldots,g_{r+1}\ra)<r$.

\smallskip
\begin{theorem}\label{thm-diaconis}
Let $r\geq 1$ and let $G_0\subset  \Z^r$ be a nonempty subset with $\rk(\la G_0\ra)=r$ and
$\mathsf D(G_0)\geq 3$.
Then \begin{eqnarray*}\nn\mathsf D^{\mathsf{elm}}(G_0)&=&\sup\{\Delta(g_1,\ldots,g_{r+1}):\;g_1,\ldots,g_{r+1}\in G_0 \ \und \ \mathsf D(\{g_1,\ldots,g_{r+1}\})\geq 3\}\\
&\leq& \sup\{\Delta(g_1,\ldots,g_{r+1}):\;g_1,\ldots,g_{r+1}\in G_0\}. \end{eqnarray*} Moreover, if $G_0=-G_0$, then equality holds.
\end{theorem}

\begin{proof}
By Lemma \ref{lem-D3-existance}, the hypothesis $\mathsf D(G_0)\geq 3$ is equivalent to the existence of an elementary atom $U\in \mathcal A(G_0)$. Since any such elementary atom $U$ satisfies $|U|\geq 3$, we conclude that \be\label{stiy}\mathsf D^{\mathsf{elm}}(G_0)\geq 3.\ee
Let $U\in \mathcal A(G_0)$ be an arbitrary elementary atom and let $X=\supp(U)$. Since $U$ is an elementary atom, $\supp(U)\cap\supp(-U)=\emptyset$. By Lemma \ref{lem-elementary-supp}, $X$ is linearly dependent over $\Q$, but any proper subset of $X$ is linearly independent. In particular, $|X|=x+1\geq 2$ with $$1\leq \rk(\la X\ra )=x\leq \rk(\la G_0\ra)=r.$$ Thus, if $x<r$, then we can find a subset $X'\subset G_0\setminus (X\cup -X)$ such that $|X'|=r-x$ and $\rk(\la X\cup X'\ra)=r$. Let $Y=X\cup X'$. Note that $\supp(U)\subset Y\subset G_0$ with $|Y|=r+1$, and that, for each $g \in G_0$, $g$ and $-g$ are not both contained in $Y$. Let $Y=\{g_1,\ldots,g_{r+1}\}$ where $g_1,\ldots,g_{x+1}$ are the elements from $X$.

Let $M$ be the $r\times (r+1)$ matrix whose columns are the vectors $g_i\in Y\subset G_0$. Then the vector of integer multiplicities $\mathbf x=(x_i)_{i\in [1,r+1]}\in \Z^{|Y|}=\Z^{r+1}$ corresponds to a zero-sum subsequence $S=\prod_{i=1}^{r+1}g_i^{x_i}\in \mathcal B(G_0)$ with $\supp(S)\subset Y$ (with $\supp(S)\cap \supp(-S)=\emptyset$) and $\vp_{g_i}(S)=x_i$ precisely when $\mathbf x$ is in the kernel of the matrix $M$ and $x_i\geq 0$ for all $i \in [1,r+1]$. Also, the vector $\mathbf x\in \Z^{r+1}$ corresponds to a zero-sum subsequence $S=\prod_{i\in I^+}g_i^{x_i} \prod_{i\in I^-}(-g_i)^{-x_i}\in \mathcal B(G_0\cup -G_0)$ with $\supp^+(S)\subset Y\cup -Y$ precisely when $\mathbf x$ is in the kernel of the matrix $M$, $I^+\subset [1,r+1]$ denotes the subset of indices $i\in [1,r+1]$ with $x_i>0$, and $I^-\subset [1,r+1]$ denotes the subset of indices $i\in [1,r+1]$ with $x_i<0$. In the latter case, we have $S\in \mathcal B(G_0)$ precisely when $-Y^-\subset G_0$ for $Y^- = \{ y_i \in Y : i \in I^- \}$.

Consider the vector $\mathbf x=\frac{1}{\delta}(x_i)_{i\in [1,r+1]}\in \Z^{r+1}$ given by $$x_i=(-1)^i\det(g_1,\ldots,g_{i-1},g_{i+1},\ldots,g_{r+1})\quad\mbox{ for $i\in [1,r+1]$,}$$
where $$\delta=\gcd\{\det(g_1,\ldots,g_{i-1},g_{i+1},\ldots,g_{r+1}): \; i\in [1,r+1]\}.$$
Since $\rk(\la Y\ra)=r$, the $x_i$ cannot all be zero, and thus $\delta>0$ is also nonzero.
Since each $g_i\in \Z^r$ is an integer-valued vector, it is clear from the the above definition that $\mathbf x\in \Z^{r+1}$ with \be\label{whoosh}\gcd\Big\{\frac{x_i}{\delta}:\;i\in [1,r+1]\Big\}=1.\ee Moreover, since $X=\{g_1,\ldots,g_{x+1}\}$ is linearly dependent, $x_i=0$ for all $i\in [x+2,r+1]$.

We now show that $\mathbf x$ is in the kernel of the matrix $M$ whose columns are the vectors $g_i\in \Z^r$.  Let $g_i=(g_{i,j})_{j\in [1,r]}$ with $g_{i,j}$ the $j$-th entry of the column vector $g_i$. With $j\in [1,r]$ arbitrary, the $j$-th entry of $M\mathbf x\in \Z^r$ is
\be\label{jcoord}\frac{1}{\delta}\Sum{i=1}{r+1}g_{i,j}x_i=\frac{1}{\delta}\Sum{i=1}{r+1}(-1)^i
g_{i,j}\det(g_1,\ldots,g_{i-1},g_{i+1},\ldots,g_{r+1}).\ee
However, by the cofactor expansion formula for the determinant of a matrix, the right hand side of \eqref{jcoord} is equal (up to sign) to the product of $\frac{1}{\delta}$ and the determinant of the $(r+1)\times (r+1)$ matrix $M'$ formed from the $r\times (r+1)$ matrix $M$ by repeating the $j$-th row $(g_{i,j})_{i\in  [1,r]}$ of $M$ and then computing the cofactor expansion about this duplicate row. As $M'$ has two duplicate rows, its determinant is zero and thus \eqref{jcoord} is zero if $j\in [1,r]$. Since the $j$-th entry of $M\mathbf x\in \Z^r$ is equal to \eqref{jcoord}, this shows that every coordinate of $M\mathbf x$ is zero. Hence $M\mathbf x=0$ and $\mathbf x$ is in the kernel of $M$ as claimed.

\smallskip
Since $\mathbf x=\frac{1}{\delta}(x_i)_{i\in [1,r+1]}\in \Z^{r+1}$ is an integer-valued vector in the kernel of $M$, it follows (as was noted earlier in the proof) that $\mathbf x$ corresponds to a nontrivial (not all $x_i$ are zero) zero-sum sequence $S\in \mathcal B(G_0\cup -G_0)$ with $\supp^+(S)\subset X\cup -X$ ($x_i=0$ for $i\in [x+2,r+1]$ and $X=\{g_1,\ldots,g_{x+1}$\}). Moreover, $\supp(S)\cap\supp(-S)=\emptyset$.
Consequently, as $X\cup -X=\supp^+(U)$ with $U\in \mathcal A(G_0)$ elementary, $\supp^+(S)=X\cup -X$. From \eqref{whoosh} we see that $S\neq T^\ell$ for any $T\in \mathcal B(G_0\cup -G_0)$ and $\ell\geq 2$. Since any proper subset of $\supp(U)$ is linearly independent, it follows that there is no nontrivial zero-sum sequence $V\in \mathcal B(G_0\cup -G_0)$ with $\emptyset\neq \supp^+(V)\subsetneq \supp^+(U)$. Thus $U$ is an elementary atom not just over $G_0$, but also over $G_0\cup -G_0$. From Lemma \ref{prop-char}, $U$ must be the unique (up to sign) elementary atom over $G_0\cup -G_0$ with signed support $X \cup -X$, and all other elementary zero-sum sequences $T$ over $G_0\cup -G_0$ with $\supp^+(T) = X \cup -X$ (for which $\supp(T)\cap\supp(-T)=\emptyset$) must be a power of either $U$ or $-U$. Applying this conclusion to $T=S$, we find that either $S=U^\ell$ or $-S=U^\ell$ for some $\ell\geq 1$. By swapping the sign of each $x_i$ in the definition of $\mathbf x=(x_i)_{i\in [1,r]}$ (thus replacing  $\mathbf x$ by $-\mathbf x$) if need be, we may without loss of generality assume the former: $S=U^\ell$ and, in particular, $S\in \mathcal B(G_0)$. Since $S\neq T^\ell$ for any $T\in \mathcal B(G_0\cup -G_0)$ and $\ell\geq 2$ as observed above, it follows that $\ell=1$ and $S=U$.

Now \begin{equation*}|U|=|S|=\frac{1}{\delta}\Sum{i=1}{r}|x_i|=\frac{1}{\delta}\Sum{i=1}{r+1}
|\det(g_1,\ldots,g_{i-1},g_{i+1},\ldots,g_{r+1})|=\Delta(g_1,\ldots,g_{r+1}).\end{equation*}
Since $U\in \mathcal A(G_0)$ was an arbitrary elementary atom with $\supp(U)\subset \{g_1,\ldots,g_{r+1}\}$, we have
\ber\label{bound-tight}\mathsf D^{\mathsf{elm}}(G_0)&\leq&\sup\{\Delta(g_1,\ldots,g_{r+1}):\;g_1,\ldots,g_{r+1}\in G_0 \quad \und \quad \mathsf D(\{g_1,\ldots,g_{r+1}\})\geq 3\}\\
&\leq& \sup\{\Delta(g_1,\ldots,g_{r+1}):\;g_1,\ldots,g_{r+1}\in G_0\}. \label{thesymmetricboundII}
\eer

\smallskip
Let $g_1,\ldots,g_{r+1}\in G_0$ be vectors with $\mathsf D(\{g_1,\ldots,g_{r+1}\})\geq  3$ and such that $\Delta(g_1,\ldots,g_{r+1})>0$. Such vectors exist by \eqref{stiy} and \eqref{bound-tight}. The condition $\Delta(g_1,\ldots,g_{r+1})>0$ implies that $\rk(\la g_1,\ldots,g_{r+1}\ra)=r$. Since $\mathsf D(\{g_1,\ldots,g_{r+1}\})\geq  3$ and by Lemma \ref{lem-D3-existance}, there exists an elementary atom $U\in \mathcal A(\{g_1,\ldots,g_{r+1}\})$, such that, without loss of generality, $\supp(U)=\{g_1,\ldots,g_{x+1}\}$, where $x\leq r$. Now, repeating the above arguments using this particular elementary atom $U$, we find that $\mathsf D^{\mathsf{elm}}(G_0) \ge |U|=\Delta(g_1,\ldots,g_{r+1})$.
Taking the supremum over $\Delta(g_1,\ldots,g_{r+1})$ for all choices of $g_1,\ldots,g_{r+1}\in G_0$ with $\mathsf D(\{g_1,\ldots,g_{r+1}\})\geq 3$ and $\Delta(g_1,\ldots,g_{r+1})>0$, we see that equality holds in \eqref{bound-tight}.

\medskip
Next suppose that $-G_0=G_0$. To complete the proof we need to show that equality holds in \eqref{thesymmetricboundII}. We may assume $\mathsf D^{\mathsf{elm}}(G_0) < \infty$ as the claim is trivially true otherwise. Let $g_1,\ldots,g_{r+1}\in G_0$ be vectors that obtain the maximum in \eqref{bound-tight}. If $\rk(\la g_1,\ldots,g_{r+1}\ra)<r$, then $\Delta(g_1,\ldots,g_{r+1})=0$ and, by the lower bound $\mathsf D^{\mathsf{elm}}(G_0)\geq 3$ from \eqref{stiy}, the $g_1,\ldots,g_{r+1}\in G_0$ cannot maximize \eqref{thesymmetricboundII}. Thus we may assume that $\rk(\la g_1,\ldots,g_{r+1}\ra)=r$.
Note that replacing any $g_i$ with $-g_i$ does not alter the value of $\Delta(g_1,\ldots,g_{r+1})$ and that the hypothesis $G_0=-G_0$ ensures that $-g_i\in G_0$. Thus, to show equality in \eqref{thesymmetricboundII}, it suffices to show that $\mathsf D(\{\epsilon_1g_1,\ldots,\epsilon_{r+1}g_{r+1}\})\geq  3$ for some choice of $\epsilon_i\in \{1,\,-1\}$ which, by Lemma \ref{lem-D3-existance}, is equivalent to the existence of an elementary atom $U\in \mathcal A(\{\epsilon_1g_1,\ldots,\epsilon_{r+1}g_{r+1}\})$ for some choice of $\epsilon_i\in \{1,\,-1\}$ which in turn is equivalent to the existence of an elementary atom $U\in \mathcal A(G_0)$ with $\supp^+(U)\subset \{g_1,\ldots,g_{r+1}\}\cup -\{g_1,\ldots,g_{r+1}\}$. To show the later we will use Lemma \ref{lem-elementary-supp}.

Let $Y\subset \{g_1,\ldots,g_{r+1}\}$ be those $g_i\in \{g_1,\ldots,g_{r+1}\}$ such that $\det(g_1,\ldots,g_{i-1},g_{i+1},\ldots,g_{r+1})\neq 0$. It follows that $\{g_1,\ldots, g_{r+1}\}\setminus \{g_i\}$ is linearly independent for any $g_i\in Y$. Thus, any subset of $\{g_1,\ldots, g_{r+1}\}\setminus \{g_i\}$, including $Y\setminus \{g_i\}$, must also be linearly independent. This shows that all proper subsets of $Y$ are linearly independent.

Suppose the elements of $Y$ are linearly independent. Then clearly $|Y|\leq r$ and thus $Y$ must be a proper subset of $\{g_1,\ldots,g_{r+1}\}$. However, since $\rk(\la g_1,\ldots,g_{r+1}\ra)=r$, it follows that we can complete $Y$ to some full rank set $Y'\subset \{g_1,\ldots,g_{r+1}\}$ with $Y\subset Y'$. Then $\det(g_1,\ldots,g_{i-1},g_{i+1},\ldots,g_{r+1})\neq 0$ for the unique $g_i\in \{g_1,\ldots,g_{r+1}\}\setminus Y'$. By the definition of $Y$, this forces $g_i\in Y$, contradicting that $g_i\notin Y'$ with $Y\subset Y'$. Thus we conclude that the elements of $Y$ are linearly dependent.

In view of the conclusions of the previous two paragraphs, along with the hypothesis $G_0=-G_0$, we can now apply Lemma \ref{lem-elementary-supp} and conclude that there exists an elementary zero-sum sequence $S\in \mathcal B(G_0)$ with $\supp^+(S)=Y\subset \{g_1,\ldots,g_{r+1}\}\cup -\{g_1,\ldots,g_{r+1}\}$. By Lemma \ref{prop-char}, this ensures that there exists an elementary atom $U\in \mathcal A(G_0)$ with $\supp^+(U)=Y\subset \{g_1,\ldots,g_{r+1}\}\cup -\{g_1,\ldots,g_{r+1}\}$ which completes the proof.
\end{proof}

\smallskip
\begin{corollary}\label{cor-dav-upper}
Let $r \in \N$, let $(\mathsf e_1,\ldots, \mathsf e_{r+1})$  denote the standard basis of $\Z^{r+1}$, and
let $G_0\subset G =\la\mathsf e_1,\ldots,\mathsf e_r\ra$ be a nonempty subset with $\rk(\la G_0\ra)=r$ and
$\mathsf D(G_0)\geq 3$. Furthermore, let $\wtilde G_0,\,\wtilde G_1 \subset\Z^{r+1}$
be the subsets given by
\begin{eqnarray*}
\wtilde G_{0}&=&\{g+\mathsf e_{r+1}:g\in G_0\}\cup G_0\quad\und\nn\\
\wtilde G_{1}&=&\{g+\mathsf e_{r+1}:g\in G_0\}\cup \{g-\mathsf e_{r+1}:g\in G_0\},
\end{eqnarray*}
and let $\mathcal M(\wtilde G_i)$ be all those nonsingular  $(r+1)\times (r+1)$ matrices
with columns $\tilde g_1,\ldots,\tilde g_{r+1}\in \wtilde G_i\subset \Z^{r+1}$, for $i \in [0,1]$. Also, let $d_{r+1}(\tilde M)$ denote the largest elementary divisor of the matrix $\tilde M$.
\begin{enumerate}
\item \label{bound1} $\mathsf D^{\mathsf{elm}}(G_0)\leq2\sup\{d_{r+1}(\tilde M):\;\tilde M\in \mathcal M(\wtilde G_0)\}\leq2\sup\{|\det(\tilde M)|:\;\tilde M\in \mathcal M(\wtilde G_0)\}$ and
\smallskip
\item \label{bound2} $\mathsf D^{\mathsf{elm}}(G_0)\leq\sup\{d_{r+1}(\tilde M):\;\tilde M\in \mathcal M(\wtilde G_1)\}\leq \sup\{|\det(\tilde M)|:\;\tilde M\in \mathcal M(\wtilde G_1)\}$.
\end{enumerate}
\end{corollary}

\begin{proof}
 By Lemma \ref{lem-D3-existance} and Theorem \ref{thm-diaconis} we know that \begin{equation*}3 \le \mathsf D^{\mathsf{elm}}(G_0)\leq \sup\{\Delta(g_1,\ldots,g_{r+1}):\;g_1,\ldots,g_{r+1}\in G_0\}.
\end{equation*}
It follows that there exist $g_1,\ldots,g_{r+1}\in G_0\subset \Z^r$ with $\Delta(g_1,\ldots,g_{r+1}) > 0$ and that the supremum on the right hand side is necessarily obtained for such a choice of $g_1,\ldots,g_{r+1}$.
Let $g_1,\ldots,g_{r+1}\in G_0$ be such that $\Delta(g_1,\ldots,g_{r+1}) > 0$. Then $\rk(\la g_1,\ldots,g_{r+1}\ra)=r$.

\smallskip
For each $i\in [1,r+1]$, let $\tilde g_i=g_i\pm \mathsf e_{r+1}$, where an appropriate choice for the sign of $\mathsf e_{r+1}$ will be determined shortly, and let $\tilde M\in\mathcal M(\wtilde G_1)$ be the $(r+1)\times (r+1)$ integer matrix with columns $\tilde g_1,\ldots,\tilde g_{r+1}\in \tilde G_1\subset \Z^{r+1}$. Let \begin{align*}&\delta=\gcd\{|\det(g_1,\ldots,g_{i-1},g_{i+1},\ldots,g_{r+1})|:\; i\in [1,r+1]\}\quad\und
 \\ &\tilde \delta=\gcd\{|\det(M')|:\;M' \mbox{ is a $r\times r$ sub-matrix of $\tilde M$}\}.\end{align*} Then $\tilde \delta\mid \delta$. By using the cofactor expansion formula for the determinant and expanding along the final row of $\tilde M$, we see that by choosing the signs in each $\tilde g_i$ appropriately, $\frac{1}{\delta}\det(\tilde M)= \Delta(g_1,\ldots,g_{r+1})$. Now, since $\tilde \delta\mid \delta$ and by
\eqref{invariant-eq}, we have $$\Delta(g_1,\ldots,g_{r+1})= \frac{1}{\delta}\det(\tilde M)\leq \frac{1}{\tilde \delta}\det(\tilde M)=d_{r+1}(\tilde M).$$
Since the choice of $g_1,\ldots,g_{r+1}\in G_0$ was arbitrary, subject to the restriction $\Delta(g_1,\ldots,g_{r+1}) > 0$, the bound from Theorem \ref{thm-diaconis} establishes \ref*{bound2}.

\smallskip
For each $i\in [1,r+1]$, redefine $\tilde g_i$ as either $\tilde g_i=g_i+ \mathsf e_{r+1}$ or $\tilde g_i=g_i$, where the choice of coefficient for $\mathsf e_{r+1}$ will be determined shortly, and let $\tilde M\in\mathcal M(\wtilde G_0)$ be the $(r+1)\times (r+1)$ integer matrix with columns $\tilde g_1,\ldots,\tilde g_{r+1}\in \tilde G_0\subset \Z^{r+1}$. Let \begin{align*}&\delta=\gcd\{|\det(g_1,\ldots,g_{i-1},g_{i+1},\ldots,g_{r+1})|:\; i\in [1,r+1]\}\quad\und
 \\ &\tilde \delta=\gcd\{|\det(M')|:\;M' \mbox{ is a $r\times r$ sub-matrix of $\tilde M$}\}.\end{align*} Then $\tilde \delta\mid \delta$. By again using the cofactor expansion formula for the determinant and expanding along the final row of $\tilde M$, we see that by choosing the coefficients for $\mathsf e_{r+1}$ in each $\tilde g_i$ appropriately, we can achieve $$\det(\tilde M)=\Summ{i\in I}|\det(g_1,\ldots,g_{i-1},g_{i+1},\ldots,g_{r+1})|$$ for some subset $I\subset [1,r]$. By using the exact opposite choices for the $\tilde g_i$, we can instead achieve $$\det(\tilde M)=-\Summ{i\in [1,r+1]\setminus I}|\det(g_1,\ldots,g_{i-1},g_{i+1},\ldots,g_{r+1})|.$$ Between these two options, take the one where $|\det(\tilde M)|$ is larger. Then $$2|\det(\tilde M)|\geq \Sum{i=1}{r+1}|\det(g_1,\ldots,g_{i-1},g_{i+1},\ldots,g_{r+1})|,$$ in which case \eqref{invariant-eq} and $\tilde \delta\mid \delta$ give $$\Delta(g_1,\ldots,g_{r+1})\leq \frac{2}{\delta}|\det(\tilde M)|\leq \frac{2}{\tilde \delta}|\det(\tilde M)|=2d_{r+1}(\tilde M),$$
 establishing  \ref*{bound1}.
\end{proof}

\medskip
\noindent
{\bf The special case $G_0 = G_r^{+}\cup -G_r^{+}$.} Our goal for the remainder of this section is to apply the machinery above to the case when $$G_0=G_r^{+}\cup -G_r^{+}$$ where $G_r^{+}$ denotes the nonzero vertices of the $r$-dimensional hypercube as defined in Section \ref{2}. Specifically, we wish to obtain upper and lower bounds for $\mathsf D(G_r^{+}\cup -G_r^{+})$ which we will apply to the study of invariants of monoids of modules in Section \ref{5}.

\medskip
If $r=1$, then clearly $\mathsf D(G_0) = 2$. Thus we suppose for the remainder of this section that $r > 1$ in which case $\mathsf D(G_0) \ge 3$.
Note that $\Delta(g_1,\ldots,g_{r+1})$ is unaffected if any $g_i$ is replaced by $-g_i$. Thus we need only consider matrices with columns from $G_r^+$ when applying Theorem \ref{thm-diaconis} and Corollary \ref{cor-dav-upper}.
By Corollary \ref{cor-dav-upper} we see that $\mathsf D^{\mathsf{elm}}(G_0)$ is bounded from above by twice the maximal determinant of a $(r+1)\times (r+1)$ \ $(0,1)$-matrix. It is well-known that this value is $\frac{1}{2^{r+1}}$ times the maximal value of a $(r+2)\times (r+2)$ \ $(1,-1)$-matrix and this value, in turn, is equal to the maximal value of a $(r+2) \times (r+2)$ $(-1,0,1)$-matrix. Moreover, this maximal is bounded by Hadamard's bound and, together with Theorem \ref{thm-elm-atom-bound}, we obtain the following upper bound

$$\mathsf D(G_0)\leq (|G_0^+|-r)\mathsf D^{\mathsf{elm}}(G_0)\leq \frac{(2^r-r-1)}{2^{r}}(r+2)^{\frac{r+2}{2}}<(r+2)^{\frac{r+2}{2}}.$$

\medskip
We now construct a lower bound for $\mathsf D(G_0)$.
Suppose we have an $(r+1)\times (r+1)$ \ $(0,1)$-matrix $M$ whose determinant is a large prime $p$.
Then we must have $d_{r+1}=p$ and $d_r=1$ where the $d_1\mid\cdots\mid d_{r+1}$ are the elementary divisors of $M$. Therefore, by \eqref{invariant-eq}, we see that the greatest common divisor of the determinants of $r\times r$ sub-matrices of $M$ must be $1$. On the other hand, if we delete any row of $M$, we obtain, by applying the cofactor expansion formula to $\det(M)$ and expanding along the deleted, row $r+1$ vectors $g_1,\ldots,g_{r+1}\in G_0$ where $\gcd\{|\det(g_1,\ldots,g_{i-1},g_{i+1},\ldots,g_{r+1})|:\; i\in [1,r+1]\}$ divides $\det(M)=p$. As $p$ is prime, we can conclude that $\gcd\{\det(g_1,\ldots,g_{i-1},g_{i+1},\ldots,g_{r+1}):\; i\in [1,r+1]\}$ is either $1$ or $p$. Moreover, since the greatest common divisor of the determinants of $r\times r$ sub-matrices is $1$, we see that we can achieve $1$ rather than $p$ by choosing an appropriate row of $M$ to delete. Then, in this case, $\Delta(g_1,\ldots,g_{r+1})=d_{r+1}=|\det(M)|=p$ and we obtain the lower bound
$$\max\{|\det(M)|:\; \mbox{$M$ is a $(r+1)\times (r+1)$ \ $(0,1)$-matrix with $|\det(M)|$ prime}\}\leq \mathsf D^{\mathsf{elm}}(G_0)\leq \mathsf D(G_0).$$

\medskip
In general, the possible values of $|\det(M)|$ for an arbitrary $(0,1)$-matrix are not known, though this question is of great interest to many researchers (see \cite{Or12} for known results on the spectrum of the determinant). However, computational evidence obtained for small values of $r$ has led many to observe that there is (at least for $r$ small) a constant $C\cong \frac12$ and a large consecutive interval $[0,C\,2^{-r-1}(r+2)^{(r+2)/2}]$ of obtainable values for $|\det(M)|$, leading some to conjecture that the set of values of $\det(M)$ is dense in a interval whose length is a fraction of the maximal possible value (see \cite{Or10}). For any interval of obtainable values $[0,n]$, Bertrand's postulate ensures that a prime of size at least $\frac12 n$ will occur in that interval. Thus, if the intuitions gathered from examining small values of $r$ hold true for larger values in a very strong sense, we would expect a lower bound for $\mathsf D(G_0)$ of the form $$\Big(\frac{r+2}{C}\Big)^{(r+2)/2}\lesssim \mathsf D^{\mathsf{elm}}(G_0)\leq \mathsf D(G_0),$$ where $C\geq 1$ is some constant. We note that such a lower bound would very nearly match the upper bound in order of magnitude.

\medskip
Constructively, the best lower bounds we have been able to achieve involve the Fibonacci numbers.
We now present a construction giving a lower bound on $\mathsf D(G_0)$ regardless of concerns about the possible values of $\det(M)$. We note that for small values of $r$, examples of $(0,1)$-matrices with large prime determinant are known and can thus be used to improve upon this bound. For $r \in \mathbb N_0$, we denote by $\mathsf F_r$ the $r$th {\it Fibonacci number}. That is, $\mathsf F_0 = 0$, $\mathsf F_1 = 1$, and $\mathsf F_r = \mathsf F_{r-1}+\mathsf F_{r-2}$ for all $r \ge 2$.

\smallskip
\begin{proposition} \label{fib-bound}
Let $r \in \N$,  let $(\mathsf e_1,\ldots,\mathsf e_r)$ denote the standard basis of $\Z^r$, and set $H_r=\la \mathsf e_1+\cdots+ \mathsf e_r\ra$.
Then there exists a sequence $S_r\in \Fc(G_r^+)$  with \begin{equation*}\sigma(S_r)\in H_r\quad\mbox{ and }\quad \Sigma_{\leq |S_r|-1}(S_r)\cap H_r=\emptyset\end{equation*} such that \begin{eqnarray*} &&|\supp(S_r)|=r \; \mbox{ with \;  $\supp(S_r)$\; spanning $\Q^r$,}\\&&|S_r|=\mathsf F_{r+1} \quad\und\quad \sigma(S_r)=\mathsf F_{r}\mathsf e_1+\cdots+\mathsf F_{r}\mathsf e_r.\end{eqnarray*}
\end{proposition}

\begin{proof}
The sequence $S_1=\mathsf e_1$ is  easily seen to satisfy the conditions of the theorem. We proceed recursively to define $S_r$ for $r\geq 2$, assuming that $$S_{r-1}=g_1\bdot\ldots\bdot g_{\mathsf F_r}\in \Fc(G_{r-1}^+)$$ has already been constructed so as to satisfy the conclusions of the theorem.

Let $$S'_r= S_{r-1} \mathsf e_r^{\mathsf F_{r-1}}\in \Fc(G_r^+)$$ and let $$S_r=\varphi(S'_r)\in \Fc(\Z^r),$$ where $\varphi:\Z^r\rightarrow \Z^r$ is the map defined by $x\mapsto -(x-(\mathsf e_1+\cdots+\mathsf e_r))$ and extends to a affine linear isomorphism of $\Q^r$. The map $\varphi$ acts on an element $x\in G_r^+$ simply by exchanging each $0$ for a $1$ and each $1$ for a $0$. Therefore $\varphi(x)\in G_r^+$ for each $x\in G_r^+\setminus \{\mathsf e_1+\cdots+\mathsf e_r\}$. As a result, since $\mathsf e_1+\cdots+\mathsf e_r\notin \supp(S'_r)$ (as $\supp(S_{r-1})\subset G_{r-1}^+$ and $r\geq 2$), we see that $$S_r\in \Fc(G_r^+) \quad \mbox{ with \; $|S_r|=|S'_r|$ \; and \; $|\supp(S_r)|=|\supp(S'_r)|$}.$$

Observe that $$|S_r|=|S'_r|=|S_{r-1}|+\mathsf F_{r-1}=\mathsf F_r+\mathsf F_{r-1}=\mathsf F_{r+1}$$ and that

\ber\nn\sigma(S_r)&=&\sigma(\varphi(S'_r))=
 |S_r|\mathsf e_1+\cdots+|S_r|\mathsf e_r-\sigma(S'_r)\\\nn&=&
 \mathsf F_{r+1}\mathsf e_1+\cdots+\mathsf F_{r+1}\mathsf e_r-(\mathsf F_{r-1}\mathsf e_1+\cdots +\mathsf F_{r-1}\mathsf e_{r})\\\nn&=&
 \mathsf F_r\mathsf e_1+\cdots+\mathsf F_r\mathsf e_r\in H_r.\eer
 Moreover, since $\varphi:\Q^r\rightarrow \Q^r$ is a affine linear isomorphism, we have that $$|\supp(S_r)|=|\supp(\varphi(S'_r))|=|\supp(S'_r)|=|\supp(S_{r-1})|+1=r.$$ Since $\varphi:\Q^r\rightarrow \Q^r$ is a affine linear isomorphism, to show that $\supp(S_r)=\supp(\varphi(S'_r))$ spans $\Q^r$ it suffices to show that $\supp(S'_r)$ spans $\Q^r$. But this is clear since $\supp(S'_r)=\supp(S_{r-1})\cup \{\mathsf e_r\}$ with $\supp(S_{r-1})$ spanning $\Q \mathsf e_1+\cdots + \Q \mathsf e_{r-1}$ by hypothesis. It remains to show that $\Sigma_{\leq |S_r|-1}(S_r)\cap H_r=\emptyset$. Since $\sigma(\varphi(T))\in H_r$ if and only if $\sigma(T)\in H_r$ for any sequence $T\in \Fc(\Z^r)$, we see that in order to show $\Sigma_{\leq |S_r|-1}(S_r)\cap H_r=\emptyset$, it suffices to show that $\Sigma_{\leq |S'_r|-1}(S'_r)\cap H_r=\emptyset$.

Suppose that $T\mid S_r$ is a nontrivial subsequence with $\sigma(T)\in H_r$, that is, the coordinates of each entry of $T$ are equal. Then the coordinates of the first $r-1$ entries are equal. However, since $S'_r=S_{r-1} \mathsf e_r^{\mathsf F_{r-1}}$ and by the hypothesis that $\Sigma_{\leq |S_{r-1}|-1}(S_{r-1})\cap H_{r-1}=\emptyset$, this is only possible if either $T\mid \mathsf e_r^{\mathsf F_{r-1}}$ or $S_{r-1}\mid T$. In the former case, since $r\geq 2$, it is clear that $\sigma(T)\notin H$. In the latter case, since $\sigma(S_{r-1})=\mathsf F_{r-1}\mathsf e_1+\cdots +\mathsf F_{r-1}\mathsf e_{r-1}$ and $S'_r=S_{r-1} \mathsf e_r^{\mathsf F_{r-1}}$, $\sigma(T)\in H$ only if $T=S'_r$. Thus $\Sigma_{\leq |S'_r|-1}(S'_r)\cap H_r=\emptyset$ follows, completing the proof.
\end{proof}

\smallskip
\begin{theorem} \label{4.2}
Let $r \in \N_{\ge 2}$, let $(\mathsf e_1, \ldots, \mathsf e_r)$ denote the standard basis of $\Z^r$, and let $G_0 = G_r^{+} \cup -G_r^{+}$.
Then $$\mathsf F_{r+2} \le \mathsf D(G_0 ) \leq \frac{(2^r-r-1)}{2^{r}}(r+2)^{\frac{r+2}{2}} \leq (r+2)^{\frac{r+2}{2}}.$$
\end{theorem}

\begin{proof}
The upper bounds follow from Corollary \ref{cor-dav-upper} (see the discussion following the corollary).
It remains to show $\mathsf F_{r+2} \le \mathsf D(G_0 )$.
Let $S_0\in \Fc(G_r^{+})$ be a sequence satisfying the conclusions of
Proposition \ref{fib-bound} and define
$$U=S_r (-\mathsf e_1-\cdots-\mathsf e_r)^{\mathsf F_r}\in \Fc(G_0).$$
Then $|U|=|S_r|+\mathsf F_r = \mathsf F_{r+1}+\mathsf F_r=\mathsf F_{r+2}$ and
$\sigma(U)=\sigma(S_r)-(\mathsf F_r\mathsf e_1+\cdots+\mathsf F_r\mathsf e_r)=0$.
Moreover, by the definition of $U$ it is clear that if $T\mid U$ is a zero-sum
subsequence with $T=T^+ T^-$ where $\supp(T^+)\subset G_r^{+}$ and $\supp(T^-)\subset -G_r^{+}$,
then we must have $\sigma(T^+)\in H$ with $T^+\mid S_r$. However, since $\sigma(S_r)\in H_r$ and $\Sigma_{\leq |S_r|-1}(S_r)\cap H_r=\emptyset$,
this is only possible if either $T^+$ is trivial or $T^+=S_r$. If $T^+$ is trivial, then
$T=(-\mathsf e_1-\cdots-\mathsf e_r)^{|T|}$ which is a zero-sum sequence only if $T$ is trivial.
If $T^+=S_r$, then $\sigma(T^+)=\sigma(S_r)=\mathsf F_r \mathsf e_1+\cdots+\mathsf F_r \mathsf e_r$
and it then follows from the definition of $U$ that the only way $T$ can be a zero-sum is
if $T=U$. Therefore $U$ is a zero-sum sequence of length $\mathsf F_{r+2}$ having no proper
nontrivial zero-sum subsequences. Thus $\mathsf D(G_0)\geq \mathsf F_{r+2}$ as desired.
\end{proof}

\medskip
\begin{remark} \label{final}~
Let $r \in \N$, $(\mathsf e_1, \ldots, \mathsf e_r)$, and $G_0$ be as in Theorem \ref{4.2} but with $r > 1$. From Theorem \ref{thm-elm-atom-bound} we know that $\mathsf D^{\mathsf{elm}}(G_0)\leq \mathsf D(G_0)\leq \frac12 |G_0|\mathsf D^{\mathsf{elm}}(G_0)$. But the expected value of $\mathsf D^{\mathsf{elm}}(G_0)$ is much larger than $|G_0|$, meaning that $\mathsf D(G_0)\approx \mathsf D^{\mathsf{elm}}(G_0)$. Indeed, by a computer based search we  have verified that $\mathsf D(G_0)= \mathsf D^{\mathsf{elm}}(G_0)$ for $r \in \{2,3\}$. However, whether $\mathsf D(G_0)= \mathsf D^{\mathsf{elm}}(G_0)$ remains true for $r\geq 4$ is not known.
For the module-theoretic relevance of this question see \cite[Lemma 6.9 and Corollary 6.10]{Ba-Ge14b}.
\end{remark}

\bigskip
\section{Monoids of modules over commutative Noetherian local rings} \label {5}
\bigskip

In this section we study direct-sum decompositions of finitely generated modules over commutative Noetherian local rings. With $R$ a commutative Noetherian local ring and $\mathcal C$ a class of finitely generated  modules over $R$, it is well-known that the monoid of modules $\mathcal V (\mathcal C)$ is Krull. Our first lemma summarizes some basic information about $\mathcal C$.

\smallskip
\begin{lemma} \label{5.1}
Let $(R, \mathfrak m)$ be a commutative Noetherian local ring with maximal ideal $\mathfrak m$ and let $(\widehat R, \widehat{\mathfrak m})$ denote its $\mathfrak m$-adic completion. Let $\mathcal C$ denote the class of all finitely generated $R$-modules and let $\mathcal C'$ denote a subclass of $\mathcal C$ such that $\mathcal V (\mathcal C') \subset \mathcal V (\mathcal C)$ is a divisor-closed submonoid.

\begin{enumerate}
\item The embedding $\mathcal V (\mathcal C) \hookrightarrow \mathcal V (\widehat{\mathcal C})$, defined by
$[M] \mapsto [M \tensor \widehat R ]$, is a divisor homomorphism into the free abelian monoid $\mathcal V (\widehat{\mathcal C})$, where $\widehat{\mathcal C}$ denotes the class of finitely generated $\widehat R$-modules. In particular, $\mathcal V (\mathcal C)$ is a Krull monoid.

\smallskip
\item $\mathcal V (\mathcal C')$ is a Krull monoid whose class group is an epimorphic image of a subgroup of the class group of $\mathcal V (\mathcal C)$. If $\mathcal V (\mathcal C')$ is tame, then each of the arithmetical finiteness results of Proposition \ref{2.1} hold.

\smallskip
\item Suppose, in addition, that $R$ is one-dimensional and reduced (no non-zero nilpotent elements). Let $G$ denote the class group of $\mathcal V (\mathcal C')$ and let $G_P \subset G$ denote the set of classes containing prime divisors.
      \begin{enumerate}
      \item The class group $G$ of $\mathcal V ( \mathcal C')$ is a finitely generated abelian group.
\smallskip
      \item $\mathcal V (\mathcal C')$ is tame if and only if $\mathsf D (G_P) < \infty$ if and only if $G_P$ is finite.
      \end{enumerate}
\end{enumerate}
\end{lemma}

\begin{proof}
The monoid $\mathcal V (\widehat{\mathcal C})$ is free abelian by the Theorem of Krull-Remak-Schmidt-Azumaya. Also, the embedding defined by $[M]\mapsto [M \tensor \hat R]$ is a divisor homomorphism by \cite{Wi01} or \cite[Corollary 1.15]{Le-Wi12a}, and hence $\mathcal V (\mathcal C)$ is a Krull monoid, proving 1.

Let $H = \mathcal V (\mathcal C)$ and suppose that the inclusion $H \hookrightarrow F = \mathcal F (P)$ a divisor theory. Then $\mathcal C (H) = \mathsf q (F)/\mathsf q (H)$. If $H' \subset H$ is divisor-closed, the inclusion $H' \hookrightarrow F' = \mathcal F (P')$, where $P' = \{ p \in P : p \ \text{divides} \newline  \text{some} \ a  \in H'\ \text{in}\ F \}$, is a divisor homomorphism. Note that $H' = F' \cap \mathsf q (H')$ and that $\mathsf q (H') = \mathsf q (F') \cap \mathsf q (H)$. Therefore the homomorphism $\mathsf q (F') \to \mathsf q (F)/\mathsf q (H)$, defined by $a \mapsto a \mathsf q (H)$ for each $a \in \mathsf q (F')$, has kernel $\mathsf q (F') \cap \mathsf q (H) = \mathsf q (H')$. Consequently, there exists a monomorphism $\mathsf q (F')/\mathsf q (H') \to \mathsf q (F)/\mathsf q (H) = \mathcal C (H)$. Finally, \cite[Theorem 2.4.7]{Ge-HK06a} implies that the class group of $H'$ is an epimorphic image of a subgroup of $\mathsf q (F')/\mathsf q (H')$, and hence of a subgroup of $\mathcal C (H)$, proving 2.

We first consider the first statement of 3. The  class group of $\mathcal V (\mathcal C)$ is free abelian of finite rank by \cite[Theorem 6.3]{H-K-K-W07}. Therefore, by statement 2, the class group $G$ of $\mathcal V (\mathcal C')$ is finitely generated. We now consider the second statement of 3. Since $G$ is finitely generated, \cite[Theorem 4.2]{Ge-Ka10a} implies that  $\mathcal V (\mathcal C')$ is tame if and only if $\mathsf D (G_P) < \infty$. Again, since $G$ is finitely generated, \cite[Theorem 3.4.2]{Ge-HK06a} implies that $\mathsf D (G_P) < \infty$ if and only if $G_P$ is finite.
\end{proof}

\medskip
\begin{remark} \label{5.2}
We now make two brief remarks on certain hypothesis in Statement 3 of Lemma \ref{5.1}.
\begin{remenumerate}
\item The assumption that $R$ is reduced can be slightly weakened (see \cite[Section 6]{H-K-K-W07}). However, the assumption that $R$ is one-dimensional is essential for guaranteeing that the class group of $\mathcal V(\mathcal C')$ is finitely generated (see \cite[Lemma 2.16]{Le-Wi12a}). Indeed, there is a  two-dimensional complete Noetherian local Krull domain $S$ whose class group is not finitely generated. By a result of Heitman, there is a factorial two-dimensional Noetherian local domain $R$ whose completion is isomorphic to $S$. Then, if $\mathcal C$ is the class of finitely generated torsion-free $R$-modules, the class group of $\mathcal V (\mathcal C)$ is isomorphic to the class group of $S$ (see \cite[Section 5]{Ba-Ge14b}), and hence not finitely generated.

\item Both characterizations in 3(b) strongly depend on the fact that the class group of $\mathcal V (\mathcal C')$ is finitely generated and thus the hypothesis that $R$ is one-dimensional in 3(a) is critical for the results of this section pertaining to monoids of modules.
\end{remenumerate}
\end{remark}

\medskip
Let $R$ be a one-dimensional reduced commutative Noetherian local ring and let $\mathcal C'$ be a class of finitely generated $R$-modules such that $\mathcal V(\mathcal C')$ is a divisor-closed submonoid of the monoid $\mathcal V(\mathcal C)$ of all finitely generated $R$-modules. Then $\mathcal V (\mathcal C')$ is a Krull monoid with finitely generated class group $G$ and we let $G_P \subset G$ denote the set of classes containing prime divisors. By Proposition \ref{3.1}, sets of lengths in $\mathcal V(\mathcal C')$ can be studied in the monoid $\mathcal B (G_P)$ of zero-sum sequences over $G_P$. Using the combinatorial results of Section \ref{4}, we study sets of lengths in such Krull monoids in Theorem \ref{5.3} in the case where $G_P$ contains the set $G_r^{+}\cup G_r^{-}$ where $G_r^{+}$ denotes the nonzero vertices of the hypercube in a finitely generated free abelian group. We then investigate finer arithmetical invariants of $\mathcal V(\mathcal C')$ for small values of $r$. We conclude this section with an explicit example of a monoid of modules realizing such a Krull monoid (see Corollary \ref{5.6}).

\smallskip
\begin{theorem} \label{5.3}
Let $H$ be a Krull monoid whose class group $\mathcal C (H)$ is free abelian with basis $(e_1, \ldots, e_r)$ for some $r \ge 2$, and let $G_P \subset \mathcal C (H)$ denote the set of classes containing prime divisors. Suppose that $G_P$ is finite and that $G_P = -G_P$ with $G_P \supset G_r^+$.
\begin{enumerate}
\item There exists $M \in \N_0$ such that, for each $k \ge 2$, \ $\mathcal U_k (H) = L' \cup L^* \cup L''$, where $L^*$ is an interval, $L' \subset \min L^* + [-M, -1]$, and  $L'' \subset \max L^* + [1,M]$.

\smallskip
\item For each $k \in \N$,
\[
      \rho_{2k} (H) = k \mathsf D (G_P) \ge k \mathsf F_{r+2} \quad \text{and} \quad k \mathsf D (G_P) + 1 \le \rho_{2k+1} (H) \le k \mathsf D (G_P) + \frac{\mathsf D (G_P)}{2} \,.
      \]
       \end{enumerate}
Upper bounds for $\mathsf D(G_P)$, and hence for each $\rho_k(H)$, then follow from Theorem \ref{4.2}.
\end{theorem}

\begin{proof}
By Proposition \ref{3.1}, $\mathcal B (G_P)$ and $H$ are tame and thus, by Lemma \ref{2.2}, it suffices to prove all assertions about $H$ for the monoid $\mathcal B (G_P)$.

\smallskip
By Proposition \ref{2.1}.3, it suffices to show that $\min \Delta (G_P) = 1$. If $U = e_1 e_2 (-e_1-e_2)$, then $U \in \mathcal A (G_P)$ with $\mathsf L \big( (-U)U \big) = \{2, 3\}$. Therefore $1 \in \Delta (G_P)$, proving 1.

\smallskip
Since $G_P \supset G_r^+ \cup -G_r^+$, the monoid $\mathcal B(G_r^+ \cup -G_r^+)$ is a divisor closed submonoid of $\mathcal \cB(G_P)$ and thus $\mathsf D(G_P) \ge \mathsf D \big(G_r^+ \cup -G_r^+ \big)$. Theorem \ref{4.2} now implies that $\mathsf D(G_P) \ge \mathsf D \big(G_r^+ \cup -G_r^+ \big) \ge \mathsf F_{r+2}$.
The inequalities involving $\rho_k  (G_P) $ now follow easily from the definitions. Indeed, if $k \in \N$ and $A \in \mathcal B (G_P)$ with
\[
  A = 0^{\vp_0(A)} U_1 \cdot \ldots \cdot U_k = 0^{\vp_0(A)} V_1 \cdot \ldots \cdot V_l \,,
\]
where $U_1, \ldots, U_k, V_1, \ldots, V_l \in \mathcal A (G_P) \setminus \{0\}$ are minimal zero-sum sequences, then
\[
  2l \le \sum_{i=1}^l |V_i| = |A| - \vp_0(A) = \sum_{i=1}^k |U_i| \le k \mathsf D (G_P) \,,
\]
whence $l \le k \mathsf D (G_P)/2 $ and thus $\rho_k  (G_P)  \le k \mathsf D (G_P)/2 $. If $U = g_1 \cdot \ldots \cdot g_l$ is a minimal zero-sum sequence of length $|U| = l = \mathsf D (G_P)$, then
\begin{equation}\label{eqn-exm-rho}
U^k (-U)^k = \prod_{i=1}^l \big( (-g_i)g_i \big)^k
\end{equation}
and hence $\rho_{2k} (G_P)  \ge k \mathsf D (G_P)$.
Multiplying each side of \eqref{eqn-exm-rho} by any fixed atom, we find that $\rho_{2k+1}(G_P) \ge k \mathsf D(G_P) + 1$.
\end{proof}

\medskip
We now make a conjecture that claims that, at least for sufficiently large $k\in \N$, the first statement of Theorem \ref{5.3} holds with $M=0$.

\smallskip
\begin{conjecture} \label{5.4}
Let $H$ be a Krull monoid as in Theorem \ref{5.3} and suppose, in addition, that $G_P = G_r^+ \cup -G_r^+$. Then there exists $k^* \in \N$ such that for each $k \ge k^*$, $\mathcal U_k (H)$ is an interval.
\end{conjecture}

\medskip
We now discuss one possible strategy for proving Conjecture \ref{5.4}. If one could show that there exists $A^* \in \mathcal B (G_P)$ with $\mathsf L (A^*)$ an interval and $\max \mathsf L (A^*) / \min \mathsf L (A^*) = \mathsf D (G_P)/2$, then there must exist $k^* \in \N$ such that for each $k\ge k^*$, $\mathcal U_k (H)$ is an interval (\cite[Theorem 3.1]{Fr-Ge08}). Unfortunately, this strategy seems to require knowledge of the precise value of the Davenport constant, which is currently known only for $r \in [2,3]$. Even for $r=4$, it seems to be computational infeasible to compute the Davenport constant. However, for small values of $r$ we can provide a direct proof of Conjecture \ref{5.4}. Indeed, we are even able to show in Corollary \ref{5.5} that the conjecture holds when $r\in[2,3]$ for $k^* = 2$. We first provide a simple lemma.

\smallskip
\begin{lemma} \label{special-observation}
Let $G$ be a free abelian group of rank $r \in \N$, let $(e_1, \ldots, e_r)$ denote a basis for $G$, and set $G_0 = G_r^+ \cup -G_r^+$. For every $U \in \mathcal A (G_0)$ with $|U| \ge 3$ and any $g \in G_0$, we have $\mathsf v_g (U) < |U|/2$.
\end{lemma}

\begin{proof}
Let $U = g^k g_1 \cdot \ldots \cdot g_l \in \mathcal A (G_0)$ where $k = \mathsf v_g (U)$ and $|U| = k+l$. By symmetry we may suppose that $g \in G_r^+$. Clearly, we have $k \le l = |U|-k$ and hence $k \le |U|/2$. Assume for the sake of contradiction that $2k = |U|$. Then $g_1, \ldots, g_l \in - G_r^+$. For each $h \in G_0$ with $h = \sum_{i=1}^r a_i e_i$ for some $a_1, \ldots, a_r \in \Z$, we denote by $S(h)$ the set $S (h) = \{i \in [1,r] : a_i \ne 0\}$. We now obtain that
\[
S (g_1) \cup \cdots \cup S (g_l) \subset S (g) \subset \bigcap_{i=1}^l S (g_i) \,,
\]
and hence $g_1 = \cdots = g_l = -g$. Unless $U=g(-g)$, this contradicts the fact that $U \in \mathcal A (G_0)$.
\end{proof}

\smallskip
\begin{corollary} \label{5.5}
Let $H$ be a Krull monoid whose class group $G$ is free abelian with basis $(e_1, \ldots, e_r)$ for some $r \in [2,3]$. Let $G_P\subset G$ denote the set of classes containing prime divisors and suppose that $G_P = G_r^+ \cup -G_r^+$.
\begin{enumerate}
\item Suppose $r = 2$. Then $\mathsf c (H) = \mathsf t (H) = \mathsf D (G_P)= 3$ and, for each $k \ge 2$, the set $\mathcal U_k (H)$ is an interval. Moreover, for all $k \ge 2$ and $j\in [0,1]$, $\rho_{2k+j} (H) = 3k+j$.

\smallskip
\item Suppose $r=3$. Then $\mathsf c (H) = \mathsf D (G_P) = 5$ and, for each $k \ge 2$, the set $\mathcal U_k (H)$ is an interval. Moreover, for all $k\ge 2$ and $j \in [0,1]$,  $\rho_{2k+j} (H) = 5k+j$.
\end{enumerate}
\end{corollary}

\begin{proof}
When $r=2$, the statement $\mathsf c (H) = \mathsf t (H) = \mathsf D (G_P) = 3$ is a simple observation. Indeed, this setting is a special case of \cite[Theorem 6.4]{Ba-Ge14b}. The remaining assertions follow from Theorem \ref{5.3}.

\smallskip
We now assume $r=3$. A lengthy technical proof or a computer search show that $\mathsf D (G_P)=5$ and that $\{ V \in \mathcal A (G_P) : |V| = 5 \} = \{V_i, -V_i : i \in [1,4]\}$ where
\[
\begin{aligned}
V_1 & = (e_1+e_2)(e_1+e_3)(e_2+e_3)(-e_1-e_2-e_3)^2 \,, \\
V_2 & = (e_1+e_2)(e_1+e_3)(-e_2-e_3)(-e_1)^2 \,, \\
V_3 & = (-e_1-e_3)(e_1+e_2)(e_2+e_3)(-e_2)^2  \,, \quad \text{and} \\
V_4 & = (-e_1-e_2)(e_1+e_3)(e_2+e_3)(-e_3)^2 \,.
\end{aligned}
\]
Lemma \ref{special-observation} then implies that each minimal zero-sum sequence of length four is squarefree. Clearly, $\mathsf L \big( (-V_1)V_1 \big) = \{2, 5\}$.  Proposition \ref{3.1} now implies that $5 \le 2 + \max \Delta (H) \le \mathsf c (H) \le \mathsf D (G_P)=5$, whence $\mathsf c (H) = 5$. Again, by Proposition \ref{3.1}, it suffices to prove the remaining assertions for $\mathcal B (G_P)$.

Theorem \ref{5.3} implies that $\rho_{2k} (G_P) = k \mathsf D (G_P) = 5k$ and that $5k+1 \le \rho_{2k+1}(G_P) \le 5k+2$ for all $k \in \N$. In order to prove that $\mathcal U_k (H)$ is an interval for each $k \in \N$, it suffices to show that $\mathcal U_k (G_P) \cap \N_{\ge k}$ is an interval for each $k \in \N$. Indeed, this follows from a simple symmetry argument (see \cite[Lemma 3.5]{Fr-Ge08}). We proceed conclude the proof by proving the following three claims {\bf A1}, {\bf A2}, and {\bf A3} which clearly imply the assertion.

\smallskip
\begin{enumerate}
\item[{\bf A1.}\,] For each $k \in \N$, $\mathcal U_{2k} (G_P) \cap \N_{\ge 2k} = [2k, 5k]$.

\smallskip
\item[{\bf A2.}\,] For each $k \in \N$, $\rho_{2k+1} (G_P) = 5k+1$.

\smallskip
\item[{\bf A3.}\,] For each $k \in \N$, $\mathcal U_{2k+1} (G_P) \cap \N_{\ge 2k+1} = [2k+1, 5 k+1]$.
\end{enumerate}

\smallskip
{\it Proof of \,{\bf A1}}.\, For any $k \in \N$, the inclusion $\mathcal U_{2k} (G_P) \cap \N_{\ge 2k} \subset [2k, \rho_{2k} (G_P)] =  [2k, 5k]$ is clear. To prove the reverse inclusion we proceed by induction on $k$. For every $j \in [3, 5]$ there exists $U \in \mathcal A (G_P)$ with $|U| = j$, and it follows that $\{2, j\} \subset \mathsf L \big( (-U)U \big)$. Therefore $\mathcal U_2 (G_P) = [2, 5]$ and, together with the induction hypothesis, we see that for $k \ge 2$,
\[
[2k, 5k] = [2,5]+[2k-2, 5k-5] = \mathcal U_2 (G_P)+\mathcal U_{2k-2} (G_P) \subset \mathcal U_{2k} (G_P) \,.
\]

\smallskip
{\it Proof of \,{\bf A2}}.\, Assume for the sake of contradiction that there exists $k' \in \N$ such that $\rho_{2k'+1} (G_P) = 5k'+2$, and let $k \in \N$ denote the smallest integer with this property. Let $B \in \mathcal B (G_P)$ be such that
\[
B = A_1 \cdot \ldots \cdot A_{2k+1} = B_1 \cdot \ldots \cdot B_{5k+2} \,, \quad \text{where} \quad A_1, \ldots, A_{2k+1}, B_1, \ldots, B_{5k+2} \in \mathcal A (G_P) \,.
\]
Then
\[
10k+4 \le \sum_{\nu=1}^{5k+2} |B_{\nu}| = |B| = \sum_{\nu=1}^{2k+1} |A_{\nu}| \le 5(2k+1) \,,
\]
and hence, after renumbering if necessary, $|A_1| = \cdots = |A_{2k}| = 5$, $A_{2k+1} \in [4,5]$, $|B_1|= \cdots = |B_{5k+1}|=2$, and $|B_{5k+2}| \in [2,3]$. Suppose there are  $i , j \in [1, 2k+1]$ with $i<j$ such that $A_i = - A_j$. Without loss of generality, $i=1$ and $j=2$ in which case $B' = A_3 \cdot \ldots \cdot A_{2k+1}$ satisfies $2k-1, 5(k-1)+ 2 \in \mathsf L (B')$, contradicting the minimality of $k$.

Thus there are distinct $U_1, \ldots, U_4 \in  \{V_i, -V_i : i \in [1,4]\}$ such that $\{ A_i : i \in [1, 2k]\} \subset \{U_1, \ldots, U_4\}$ with $U_i \ne -U_j$ for all $i, j \in [1,4]$. Suppose that $U_1 = g_1^2g_2g_3g_4$ such that $|\{ i \in [1,2k] : A_i = U_1 \}| \ge \lceil k/2 \rceil$. Then $\mathsf v_{g_1} (B) \ge \max \{k, 2\}$. By inspection of all $W \in \mathcal A (G_P)$ with $|W|=5$, it follows that $\mathsf v_{-g_1} (U_1 \cdot \ldots \cdot U_4) = 0 = \mathsf v_{-g_1} (A_1 \cdot \ldots \cdot A_{2k})$, and hence $\mathsf v_{-g_1}(B) = \mathsf v_{-g_1} (A_{2k+1})$. Since $|B_1|= \cdots = |B_{5k+1}|=2$,  $|B_{5k+2}| \in [2,3]$, and $\mathsf v_{g_1} (B_{5k+2}) \le 1$, it follows that
\[
\mathsf v_{-g_1} (A_{2k+1}) = \mathsf v_{-g_1} (B) \ge \mathsf v_{g_1} (B)-1 \ge \max \{k-1, 1\} \,,
\]
and thus $|A_{2k+1}| = 4$. Therefore $|B_{5k+2}| = 2$, $\mathsf v_{-g_1} (A_{2k+1}) = \mathsf v_{-g_1} (B) = \mathsf v_{g_1} (B)\ge \max \{k, 2\}$. However, as we noted before, each minimal zero-sum sequences of length four is squarefree, a contradiction.

{\it Proof of \,{\bf A3}}.\, Let $k \in \N$. By assertion {\bf A2}, $\mathcal U_{2k+1} (G_P) \cap \N_{\ge 2k+1} \subset [2k+1, \rho_{2k+1} (G_P)] =  [2k+1, 5k+1]$. On the other hand,
\[
[2k+1, 5k+1] = 1+[2k, 5k] = \mathcal U_1 (G_P)+\mathcal U_{2k} (G_P) \subset \mathcal U_{2k+1} (G_P) \,. \qedhere
\]
\end{proof}

\medskip
Module theory provides an abundance of examples of Krull monoids satisfying the assumptions of Theorem \ref{5.3} (see \cite[Section 4]{Ba-Ge14b}, or \cite[Chapter 1]{Le-Wi12a} for various realization results). Below we provide one specific example of a Krull monoid where the set $G_P$ of classes containing prime divisors is precisely $G_P=G_r^+\cup G_r^-$ where $G_r$ is the set of nonzero vertices of the $r$-dimensional hypercube.

\medskip
\begin{corollary} \label{5.6}
Let $(R, \mathfrak m)$ be a one-dimensional analytically unramified commutative Noetherian local domain with unique maximal ideal $\mathfrak m$. Further assume that the $\mathfrak m$-adic completion $\widehat R$ of $R$ has $r+1$ minimal primes $\mathfrak q_1, \ldots, \mathfrak q_{r+1}$, and let $M$ be a torsion-free $R$-module whose completion $\widehat M=M\tensor_R\widehat R$ satisfies
 \[
 \widehat M \cong \bigoplus_{\emptyset\not=I \subset [1,  r+1]} \frac{\widehat R}{\cap _{i \in I}\mathfrak q_i} \,.
 \]
Then $\mathcal V (\add (M))$ is a Krull monoid with class group $G\cong \Z^r$ and the set $G_P$ of classes containing prime divisors satisfies $G_P = G_r^+ \cup -G_r^+$, where $G, \,  G_P$, and $G_r^+$ are as in Theorem \ref{5.3}. Therefore all arithmetical invariants, including all   sets $\mathcal U_k \big( \mathcal V (\add (M)) \big)$, are as described in Theorem \ref{5.3} and Corollary \ref{5.5}.
\end{corollary}

\begin{proof}
The statements about $G$ and $G_P$ follow from \cite[Example 4.21]{Ba-Ge14b}. The arithmetical consequences then follow from Theorem \ref{5.3} and Corollary \ref{5.5}.
\end{proof}

\bigskip
\section{Monoids of modules over Pr\"ufer rings} \label{6}
\bigskip

In this section we study classes $\mathcal C$ of finitely presented modules over Pr\"ufer rings and characterize the algebraic structure of the monoid $\mathcal V (\mathcal C)$. Specifically, we study certain classes of projective modules over various types of Pr\"ufer rings, and show that they are always half-factorial. We also study the catenary and tame degrees of these monoids.  We first recall the definition of a Pr\"ufer ring and related topics as well as that of finitely primary monoids. For a general reference on modules over Pr\"ufer rings, the reader may wish to consult the monograph of Fuchs and Salce \cite{Fu-Sa01}. For modules over Pr\"ufer rings with zero-divisors, we refer the reader to \cite{Fe-La02a}. For additional information on finitely primary monoids, see \cite[Sections 2.7, 3.1, and 4.3]{Ge-HK06a}.

\medskip
A {\it Pr\"ufer ring} is a commutative ring in which every finitely generated regular ideal is invertible. A commutative ring $R$ has
\begin{itemize}
\item the {\it $1\frac{1}{2}$ generator property} if, for any invertible ideal $I \subset R$ and any regular element $a \in I \setminus \rad (R) I$, there exists an element $b \in R$ such that $I = Ra + Rb$.

\smallskip
\item  {\it small zero-divisors} if for every zero-divisor $a \in R$ and any ideal $A \subset R$, $A + aR = R$ implies that $A = R$.
\end{itemize}

\medskip
A monoid $H$ is  called  {\it finitely primary} if there exist $s,\, \alpha \in \N$ with the following properties:
\begin{itemize}
\item[]
$H$ is a submonoid of a factorial monoid $F= F^\times \times [p_1,\ldots,p_s]$ for  $s$ pairwise non-associated prime elements $p_1, \ldots, p_s$ satisfying
\[\qquad
H \setminus H^\times \subset p_1 \cdot \ldots \cdot p_sF \quad \text{and} \quad (p_1 \cdot \ldots \cdot p_s)^\alpha F \subset H  \,.
\]

\end{itemize}
\quad In this case we say that $H$ is finitely primary of  {\it rank}  $s$  and  {\it exponent}  $\alpha$.

\medskip
It is easy to show that the complete integral closure of such a finitely primary monoid $H$ is $F$, and hence $H$ is a Krull monoid if and only if $H$ is factorial. Moreover, $H_{\red}$ is finitely generated if and only if $s=1$ and $(\widehat H^{\times} \DP H^{\times}) < \infty$. The main (and motivating) examples of finitely primary monoids stem from ring theory. For example, if $R$ is a one-dimensional local Mori domain with nonzero conductor $(R \DP \widehat R)$ and $\widehat R$ denotes the complete integral closure of $R$, then $R^{\bullet}$ is finitely primary (\cite[Proposition 2.10.7]{Ge-HK06a}). The arithmetic of finitely primary monoids is well-studied (\cite[Sections 2.7, 3.1, and 4.3]{Ge-HK06a}). In particular, the sets $\mathcal U_k (H)$ are finite (for one $k \ge 2$, or equivalently for all $k \ge 2$) if and only if $s=1$. In our main results of this section, Theorems \ref{6.1} and \ref{6.3}, we apply these arithmetical results to monoids of modules of modules over Pr\"ufer rings.

\medskip
\begin{theorem} \label{6.1}
Let $R$ be a Pr\"ufer ring such that $R$ has the $1\frac{1}{2}$ generator property and $R$ has small zero-divisors. Let $\mathcal C_{\proj}$ be the class of finitely generated projective $R$-modules.
\begin{enumerate}
\item $\mathcal V (\mathcal C_{\proj})$ is a finitely primary monoid of rank $1$ and of exponent $1$. Moreover, $\mathcal V (\mathcal C_{\proj})$ is finitely generated if and only if $\Pic (R)$ is finite.

\smallskip
\item If $\mathcal C_{\proj}$ does not satisfy KRSA, then  $\mathcal V (\mathcal C_{\proj})$ is half-factorial with $\mathsf c \big( \mathcal V (\mathcal C_{\proj}) \big) = \mathsf t \big( \mathcal V (\mathcal C_{\proj}) \big)  = 2$.
\end{enumerate}
\end{theorem}

\begin{proof}
We first consider statement 1. By \cite[Corollary 4]{Fe-La02a}, every module $P$ in $\mathcal C_{\proj}$ is isomorphic to $R^{n-1} \oplus I$ where $n \in \N$ is the rank of $P$, $I$ is an invertible ideal, and the isomorphism class of $P$ is determined by that of $I$ and by the rank $n$. Thus the map $\varphi \colon \mathcal V (\mathcal C_{\proj}) \to H = (\Pic (R) \times \N) \cup \{(0,0)\}$, defined by $P \mapsto ([I],n)$, is an isomorphism. By definition, $H \subset  \Pic (R) \times (\N_0, +)$ is finitely primary of rank $1$ and exponent $1$. Since $H$ is reduced, we obtain that $(\widehat H^{\times} \DP H^{\times}) = |\Pic (R)|$ and thus $H$ is finitely generated if and only if $\Pic (R)$ is finite.

\smallskip
We now compute the catenary and tame degrees based on the monoid described in 1. Suppose that $\mathcal C_{\proj}$ does not satisfy KRSA. Then $H$ is not factorial and hence $\mathsf c (H) \ge 2$. By \cite[Theorem 3.1.5]{Ge-HK06a}, every finitely primary monoid of rank $1$ and exponent $\alpha$ satisfies $\mathsf c (H) \le \mathsf t (H) \le 3\alpha-1$, and hence $\mathsf c (H) = \mathsf t (H)=2$. Now Proposition \ref{2.1} implies that $\Delta (H) = \emptyset$, that is, $H$ is half-factorial.
\end{proof}

\medskip
We now restrict our attention to direct-sum decompositions of modules over Pr\"ufer domains. Specifically, we consider the class of all finitely presented modules, including torsion modules. Before proceeding, we recall the following characterization of Pr\"ufer domains. An integral domain $R$ is  Pr\"ufer if and only if the following two equivalent conditions are satisfied:
\smallskip
\begin{itemize}
\item[\bf (P1)] The torsion submodule of any finitely generated $R$-module $M$ is isomorphic to a direct summand of $M$ (\cite[Chap.V, Cor. 2.9]{Fu-Sa01}).
\smallskip
\item[\bf (P2)] Every finitely generated $R$-module is projective if and only if it is torsion-free (\cite[Chap.V, Theorem 2.7]{Fu-Sa01}).
\end{itemize}

 \medskip
In Theorem \ref{6.3} we can provide a more precise arithmetical description of $\mathcal V(\mathcal C)$, where $\mathcal C$ is the class of finitely presented modules over a Pr\"ufer domain, if we further assume that the domain is $h$-local. We now recall this class of integral domains. Let $R$ be a domain and, for an ideal $I \subset R$, let $\Omega (I)$ denote the set of maximal ideals of $R$ containing $I$. Note, that $|\Omega (I)|=1$ implies that $R/I$ is local and hence indecomposable as an $R$-module. Also recall that a domain $R$ has finite character if each nonzero element of $R$ is contained in at most finitely many maximal ideals of $R$. Now, we say that an integral domain $R$ is  {\it $h$-local} if the following equivalent conditions are satisfied (\cite[Theorem 2.1]{Ol08a}):
\smallskip
\begin{itemize}
\item[\bf (H1)] $R$ has finite character and each nonzero prime ideal of R is contained in a unique maximal ideal.
\smallskip
\item[\bf (H2)] For each nonzero ideal $I$ of $R$, $R/I$ has a decomposition $\oplus_{\nu=1}^m R/I_{\nu}$ with $|\Omega (I_1)| = \cdots = |\Omega (I_m)|=1$.
\smallskip
\item[\bf (H3)] Each torsion torsion $R$-module $M$ is canonically isomorphic to $\oplus_{\mathfrak p \in \max (R)} M_{\mathfrak p}$.
\end{itemize}
\smallskip
For additional information on  $h$-local domains we refer the reader to the survey article by Olberding \cite{Ol08a} and to the monograph by Fontana, Houston, and Lucas \cite{Fo-Ho-Lu13a}. The next proposition gathers together the module-theoretic results necessary for the arithmetical results we present in Theorem \ref{6.3}. We would like to thank Bruce Olberding for the short proof of Proposition~\ref{6.2}.1.

\smallskip
\begin{proposition} \label{6.2}
Let $R$ be a commutative ring.
\begin{enumerate}
\item Let $m \in \N$ and let $I, I_1, \ldots, I_m$ be ideals of $R$ such that $R/I \cong R/I_1 \oplus \cdots \oplus R/I_m$. If $I$ is finitely generated and projective as an $R$-module, then $I_1, \ldots, I_m$ are also finitely generated and projective as $R$-modules.

\smallskip
\item Let $m, n \in \N$ and let $I_1, \ldots, I_m, J_1, \ldots, J_n$ be ideals of $R$. If $I_m \subset \cdots \subset I_1$, $J_n \subset \cdots \subset J_1$, and
    \[
    R/I_1 \oplus \cdots \oplus R/I_m \cong R/J_1 \oplus \cdots \oplus R/J_n \,,
    \]
    then $m=n$ and $I_{\nu} = J_{\nu}$ for each $\nu \in [1, m]$.

\smallskip
\item Let $R$ be an $h$-local Pr\"ufer domain and let $M$ be a finitely presented $R$-module. Then, as an $R$-module, $M$ decomposes as
    \[
    M \cong R/I_1 \oplus \cdots \oplus R/I_m \oplus I_{m+1} \oplus \cdots \oplus I_n \,,
    \]
    where $n \in \N_0$, $m \in [0, n]$, $I_m \subset \cdots \subset I_1$ are proper invertible ideals of $R$, and $I_{m+1}, \ldots, I_n$ are invertible ideals of $R$.
\end{enumerate}
\end{proposition}

\begin{proof}
Clearly, it is sufficient to prove that  $I_1$ is finitely generated and projective. We set $J = I_2 \cap \cdots \cap I_m$ and, by \cite[Lemma 1.1, Chap. V]{Fu-Sa01}, obtain that $I = I_1 \cap J$ and $R = I_1 + J$. Therefore there is a short exact sequence
\[
0 \to I \to I_1 \oplus J \to R \to 0
\]
of $R$-modules where the second map is the embedding and where the third map is given by $(x,y) \mapsto x
- y$ for all $x \in I_1$ and all $y \in J$. This sequence splits, and thus $I_1 \oplus J \cong R \oplus I$. Now, since $I$ is finitely generated and projective, so is $I_1$. This proves 1.

\smallskip
For the proofs of statements 2 and 3, see Proposition 2.10 and Theorem 4.12 in \cite[Chap. V]{Fu-Sa01}.
\end{proof}

We now state our main results about finitely presented modules over Pr\"ufer domains.

\smallskip
\begin{theorem} \label{6.3}
Let $R$ be a Pr\"ufer domain, $\mathcal C$ the class of all finitely presented $R$-modules, $\mathcal C_{\tor}$ the class of finitely presented torsion modules, and $\mathcal C_{\proj}$ the class of finitely generated projective modules.
\begin{enumerate}
\item $\mathcal V (\mathcal C) = \mathcal V (\mathcal C_{\proj}) \times \mathcal V (\mathcal C_{\tor})$.

\smallskip
\item If $R$ has the $1\frac{1}{2}$ generator property, then $\mathcal V (\mathcal C_{\proj})$ is finitely primary of rank $1$ and  exponent $1$.

\smallskip
\item Assume, in addition, that $R$ is $h$-local. Then $\mathcal V (\mathcal C_{\tor})$ is free abelian and $\mathcal V (\mathcal C)$ is half-factorial. Furthermore, $\mathcal V (\mathcal C)$ is either factorial or $\mathsf c \big( \mathcal V (\mathcal C) \big) = \mathsf t \big( \mathcal V (\mathcal C) \big)  = 2$.
\end{enumerate}
\end{theorem}

\begin{proof}
The proof of 1 follows immediately from {\bf (P1)} and {\bf (P2)}, and statement 2 follows immediately from Theorem \ref{6.1}.

\smallskip
We now prove statement 3, and we begin by showing that $\mathcal V (\mathcal C_{\tor})$ is free abelian. Since $R$ has finite character, $R$ has the $1\frac{1}{2}$ generator property. Let $M$ be a finitely presented non-zero torsion $R$-module. We now argue that the module $M$ has a decomposition as a direct sum of indecomposable finitely presented $R$-modules, and that such a decomposition is unique up to isomorphism. By Proposition \ref{6.2}.3, $M$ has a decomposition
\[
M \cong R/I_1 \oplus \cdots \oplus R / I_m \,,
\]
with $m \in \N$ and invertible ideals $I_m  \subset \cdots \subset I_1 \subsetneq R$. Property {\bf (H2)} then implies that for each $\nu \in [1, m]$, $R/I_{\nu}$ has a direct-sum decomposition into indecomposables modules, each of the form $R/J$ where $J \subset R$ and $|\Omega (J)| = 1$. Since each $I_{\nu}$ is invertible, each $I_{\nu}$ is finitely generated and projective and, by Proposition \ref{6.2}.1, the same is true for all ideals $J$ with $R/J$ occurring in the direct sum decomposition of $R/I_{\nu}$. Thus, after replacing the $R/I_{\nu}$ with direct-sums of finitely generated indecomposable $R$-modules of the form $R/J$ and then renaming, we may suppose that each $R/I_{\nu}$ is an indecomposable $R$-module, that each $I_{\nu}$ is an invertible ideal, and that $|\Omega (I_{\nu})|=1$ for each $\nu \in [1, m]$.

\smallskip
Let $M \cong C_1 \oplus \cdots \oplus C_n$ be any direct-sum decomposition of $M$ into indecomposable finitely presented $R$-modules. Then Proposition \ref{6.2} and {\bf (H2)} imply that, for each $\nu \in [1, n]$, $C_{\nu} \cong R/J_{\nu}$ for some invertible ideal $J_{\nu} \subset R$ with $|\Omega (J_{\nu})| = 1$. Therefore $M \cong R/J_1 \oplus \cdots \oplus R/J_n$. Let $\mathfrak p$ be a maximal ideal of $R$. Then
\[
M_{\mathfrak p} \cong (R/I_1)_{\mathfrak p} \oplus \cdots \oplus (R/I_m)_{\mathfrak p} \cong (R/J_1)_{\mathfrak p} \oplus \cdots \oplus (R/J_n)_{\mathfrak p}
\]
and, by {\bf (H3)}, it suffices to prove uniqueness for the $R_{\mathfrak p}$-module $M_{\mathfrak p}$. Since the set of all ideals in the valuation domain $R_{\mathfrak p}$  form a chain, the uniqueness follows from Proposition \ref{6.2}.2.

\smallskip
Suppose now that $\mathcal V (\mathcal C)$ is not factorial. By 1. and 2., $\mathcal V (\mathcal C) \cong F \times D$ where $F$ is free abelian and where $D$ is not factorial, but is finitely primary of rank $1$ and exponent $1$. Then
\[
\mathsf c (F \times D) = \mathsf t (F \times D) = \mathsf c (D) = \mathsf t (D) = 2
\]
and hence $F \times D$ is half-factorial.
\end{proof}

\smallskip
\begin{remark}
Since a Noetherian Pr\"ufer domain is precisely a Dedekind domain and since in the Noetherian setting the concepts of finitely presented and finitely generated modules coincide, Theorem \ref{6.3} also describes the monoid of all finitely generated modules over a Dedekind domain. Of course, these results can be obtained even more simply from the classical results of Steinitz. In the following section we consider direct-sum decompositions of yet another class of rings that generalize Dedekind domains.
\end{remark}

\bigskip
\section{Monoids of modules over hereditary Noetherian prime rings} \label{7}
\bigskip

In this final section we study classes $\mathcal C$ of finitely generated right modules over hereditary Noetherian prime (HNP) rings, a generalization of Dedekind prime rings (see \cite[\S5.7]{Mc-Ro01a}). Module theory over HNP rings is carefully presented in the monograph of Levy and Robson \cite{Le-Ro11a}, and it is on this work that this section is based. We begin with the arithmetical preparations necessary to state the main result in this section, Theorem \ref{hnp-proj}. There we give a characterization of the monoid of stable isomorphism classes of finitely generated projective right modules over HNP rings and use this information to study its arithmetic.

\smallskip
\begin{proposition} \label{monext}
Let $H_0$ and $D$ be monoids and define
  \[
    H = H_0 \monext D = (H_0 \setminus H_0^\times) \times D \;\cup\; H_0^\times \times \{1_D\}.
  \]
  Then $H$ is a submonoid of $H_0 \times D$.
  If $D = \{1_D\}$  or $H_0$ is a group, then $H=H_0$.
  Suppose  that $D \ne \{1_D\}$  and that $H_0$ is not a group.
  \begin{enumerate}
    \item\label{monext:units}  $H^\times = H_0^\times \times \{1_D\}$.
    \item\label{monext:nocic}
       $\quo(H) = \quo(H_0) \times \quo(D)$, $H \subset H_0 \times D$ is not saturated, and $\widehat H = \widehat H_0 \times \widehat D$.
      In particular, $H$ is not completely integrally closed and hence not a Krull monoid.
    \item\label{monext:transfer}
      The projection $\theta\colon H \to H_0,\, (a,d) \mapsto a$ is a transfer homomorphism.

    \item\label{monext:omega}
      Suppose $H_0$ and $D$ are atomic.
      Let $u \in \cA(H_0)$ and let $d \in D$.
      Then
      \[
        \max\{ \omega(H_0,u),\, \omega(D,d) \} \;\le\; \omega(H, (u,d)) \;\le\; \omega(H_0,u) + \omega(D,d) + \epsilon
      \]
      where $\epsilon=1$ if $u$ is prime and $d \in D^\times$, and $\epsilon=0$ otherwise. In particular, $\omega(H,(u,d)) < \infty$ if and only if $\omega(H_0,u) <\infty$ and $\omega(D,d) < \infty$, and if $D$ is not a group, then $\omega(H) = \infty$ and $\st(H) = \infty$.
    \item\label{monext:group}
      Suppose $H_0$ is atomic, $D$ is a group, and let $d \in D$.
      If $u \in \cA(H_0)$ is a prime element, then
      \[
        \omega(H,(u,d)) = 2\text{, }\,\tau(H,(u,d)) = 1 \text{ and } \st(H,(u,d)H^\times) = 2.
      \]
      If $u \in \cA(H_0)$ is not a prime element, then
      \[
        \omega(H,(u,d)) = \omega(H_0,u)\text{, }\,\tau(H,(u,d)) = \tau(H_0,u) \text{ and } \st(H,(u,d)H^\times) = \mathsf t(H_0,uH_0^\times).
      \]
      In particular,
      \[
        \st(H) = \max\{ 2,\,\st(H_0)\}\text{, }\,\omega(H) = \max\{ 2,\, \omega(H_0) \} \text{ and } \tau(H) = \max\{ 1,\, \tau(H_0) \}.
      \]
    \item\label{monext:cat}
      Suppose $H_0$ and $D$ are atomic.
      Let $d \in D$ and $a \in H_0 \setminus (\cA(H_0) \cup H_0^\times)$.
      Then $\sc_H((a,d)) =0$ if
      \begin{itemize}
        \item $D^\times = \{ 1,\, d \}$, and $\sZ_{H_0}(a) = \{ (uH_0^\times)^2 \}$ for some $u \in \cA(H_0)$, or
        \item $D$ is reduced, $d=1$ and $\sc_{H_0}(a) = 0$, or
        \item $D$ is reduced, $d \in \cA(D)$ and $\sZ_{H_0}(a) = \{ (uH_0^\times)^k \}$ for some $u \in \cA(H_0)$ and $k \in \N_{\ge 2}$.
      \end{itemize}
      In any other case, $\mathsf c_H((a,d)) = \max\{ 2,\, \mathsf c_{H_0}(a) \}$.
      In particular,
      \[
        \sc(H) = \max\{ 2,\, \sc(H_0) \}.
      \]
  \end{enumerate}
\end{proposition}

\smallskip
Before proceeding with the proof, the reader may find it useful to note that $$H = H_0 \monext D=\{(h,d) \in H_0 \times D\colon h \in H_0^\times \text{ only if } d=1_D\}.$$ Also, it will be convenient in the proof to introduce the following notation. For a monoid $S$ and elements $a, b \in S$, we write $a \mid\mid b$ to denote that $a \mid b$ and $b \nmid a$, that is, $a$ is a strict divisor of $b$.

\begin{proof}
It is easily checked that $H$ is a submonoid of $H_0 \times D$ and that $H=H_0$ if $H_0 = H_0^{\times}$ or $D = \{1_D\}$. Assume now that $D \ne \{1_D\}$ and $H_0 \ne H_0^{\times}$. Fix $a_0 \in H_0 \setminus H_0^\times$ and $d_0 \in D \setminus \{\, 1_D \,\}$.

\smallskip
The proof of \ref*{monext:units} is clear.

\smallskip
For statement \ref*{monext:nocic}, note that we have $\quo(H) \subset \quo(H_0) \times \quo(D)$ and must show the reverse inclusion. Let $a,b \in H_0$ and $d,e \in D$. Then $aa_0,ba_0 \in H_0 \setminus H_0^\times$ and thus $(ab^{-1},de^{-1}) = (aa_0,d)(ba_0,e)^{-1} \in \quo(H)$. Therefore $\quo(H_0) \times \quo(D) \subset \quo(H)$. To see that the inclusion $H \subset H_0 \times D$ is not saturated, note that $(1,d_0) \in \quo(H) \,\cap\, (H_0 \times D)$, but $(1,d_0) \not \in H$.

We now show that $\widehat H = \widehat H_0 \times \widehat D$. Since $H \subset H_0 \times D$, and since both monoids have the same quotient group, it is certainly true that $\widehat H \subset \widehat{H_0 \times D} = \widehat H_0 \times \widehat D$. For the reverse inclusion, let $(x,y) \in \widehat H_0 \times \widehat D$. By definition, there exist $c \in H_0$ and $d \in D$ such that $(c,d)(x,y)^n \in H_0 \times D$ for all $n \in \N_0$. But then $ca_0x^n \in H_0 \setminus H_0^\times$ and $dy^n \in D$ for all $n \in \N_0$ and thus $(ca_0,d)(x,y)^n \in H$. Therefore $(x,y) \in \widehat H$. Since $H \ne \widehat H$, $H$ is not completely integrally closed and thus not a Krull monoid.

\smallskip
We now prove statement  \ref*{monext:transfer}. Clearly $\theta$ is surjective and, by \ref*{monext:units}, $\theta^{-1}(H_0^\times) = H^\times$. Let $(a,d) \in H$ and $b$, $c \in H_0$ such that $\theta((a,d)) = bc$. To establish that $\theta$ is a transfer homomorphism, it will suffice to find $e,f \in D$ such that $(b,e),(c,f) \in H$ and $ef=d$ since, in this case, $(a,d) = (b,e)(c,f)$, $\theta((b,e)) = b$ and $\theta((c,f)) = c$. If $b$ is not a unit in $H_0$, then $e=d$ and $f=1$ gives $(b,d)$, $(c,1) \in H$ with $ef=d$. Similarly, if $c$ is not a unit, then $e=1$ and $f=d$ gives $(b,1)$, $(c,d) \in H$ with $ef=d$. If both $b$, $c \in H_0^\times$, then $a \in H_0^\times$ and it necessarily follows that $d = 1$. Now $e=f=1$ gives $(b,1),(c,1) \in H$.

\smallskip
To show \ref*{monext:omega} we first observe that, for $(a,d)$, $(b,e) \in H$, $(a,d) \mid (b,e)$ if and only if $a \mid b$ and $d=e$ or $a \mid\mid b$ and $d \mid e$. We first prove the correctness of the lower bound and for this we may assume that $k=\omega(H,(u,d)) < \infty$. Suppose that $m,k \in \N$ with $k \le m$ and that $u_1,\ldots,u_m \in \cA(H)$ are such that $u \mid u_1\cdot \ldots \cdot u_m$. Then $(u,d) \mid (u_1,d)\cdot \ldots \cdot (u_m,d)$, and hence there exists a subproduct of at most $k$ elements that is divisible by $(u,d)$, say $(u,d) \mid (u_1,d)\cdot \ldots \cdot (u_k,d)$. But then $u \mid u_1\cdot \ldots \cdot u_k$ and consequently $\omega(H_0,u) \le k$. Similarly, suppose $m, k \in \N$ with $k \le m$ and suppose $d_1,\ldots,d_m \in \cA(D)$ are such that $d \mid d_1\cdot \ldots \cdot d_m$. Then $(u,d) \mid (u^2,d_1)\cdot \ldots \cdot(u^2,d_m)$, and so there is a subproduct of at most $k$ elements that is divisible by $(u,d)$, say $(u,d) \mid (u^2,d_1)\cdot \ldots \cdot (u^2,d_k)$. But then $d \mid d_1\cdot \ldots \cdot d_k$ and so $\omega(D,d) \le k$.

We now verify the upper bound, and for this we may assume that $k = \omega(H,u) <\infty$ and $l = \omega(D,d) < \infty$. Suppose that $m,k,l \in \N$ with $m \ge k + l + \epsilon$ and suppose $(u_1,d_1),\ldots, (u_m,d_m) \in \cA(H)$ are such that $(u,d) \mid (u_1,d_1)\cdot \ldots \cdot (u_m,d_m)$. Then $u \mid u_1\cdot \ldots \cdot u_m$ and $d \mid d_1\cdot \ldots \cdot d_m$, whence there are subsets $I$, $J \subset [1,m]$ with $\card{I} \le k$ and $\card{J} \le l$ such that $u \mid \prod_{i \in I} u_i$ and $d \mid \prod_{j \in J} d_j$. Now $\card{I \cup J} \le k + l$ and, after renumbering and possibly enlarging $I$ and/or $J$, we have that $I \cup J = [1,k+l]$ with $u \mid u_1\cdot \ldots \cdot u_{k+l}$ and $d \mid d_1\cdot \ldots \cdot d_{k+l}$. If $u$ is not a prime element, then $k \ge 2$. If $d$ is not a unit then $l \ge 1$. In either of these cases, $k+l \ge 2$ and thus $u \mid\mid u_1\cdot \ldots \cdot u_{k+l}$. Therefore $(u,d) \mid (u_1,d_1) \cdot \ldots \cdot (u_{k+l}, d_{k+l})$, and consequently $\omega(H,(u,d)) \le k + l$.
On the other hand, if $u$ is a prime element and $d$ is a unit (i.e., $k+l=1$), then $u \mid\mid u_{1}\cdot \ldots \cdot u_{k+l+1}$ and hence $(u,d) \mid (u_1,d_1)\cdot \ldots \cdot(u_{k+l+1},d_{k+l+1})$.

The claim that $\omega(H,(u,d)) < \infty$ if and only if $\omega(H_0,u) < \infty$ and $\omega(D,d) < \infty$ is clear from the now-verified inequalities. If $D$ is not a group, then (since $D$ is atomic) there exists some atom $d \in D$. Then $\omega(D,d^k) \ge k$ for all $k \in \N$, and thus $\omega(H, (u,d^k)) \ge k$ for all $u \in \cA(H_0)$. This implies that  $\omega(H) = \infty$ and therefore $\st(H) = \infty$ by Proposition \ref{2.1}.

\smallskip
We now prove statement  \ref*{monext:group} in which case we assume that $D$ is a group. Since, for any atomic monoid $S$ and any non-prime atom $u \in S$, $\mathsf t(S,uS^\times) = \max\{ \omega(S,u), \tau(S,u) + 1 \}$ (see Section \ref{2}), it will suffice to establish the claim for the $\omega$- and $\tau$-invariants.

First observe that  since every element of $D$ is a unit, if $(a,e),(b,f) \in H$ with $a \mid b$, then $(a,e) \mid (b,f)(c,g)$ for any $(c,g) \in H \setminus H^\times$, and if $a \mid\mid b$ then $(a,e) \mid (b,f)$. In particular, if $k \ge 2$, $u_1,\ldots,u_k \in \cA(H_0)$, and $u \in \cA(H_0)$, then $u \mid u_1\cdot \ldots \cdot u_k$ if and only if $(u,e) \mid (u_1,e_1)\cdot \ldots \cdot (u_k,e_k)$ for any (equivalently, all) $e,e_1,\ldots,e_k \in D$.

Since $d \in D^\times$, we have that $\omega(D,d) = 0$ and thus statement \ref*{monext:omega} implies that $\omega(H,(u,d)) = \omega(H_0,u)$ if $u$ is not a prime element, and $\omega(H,(u,d)) \in [1,2]$ if $u$ is a prime element (since then $\omega(H_0,u) = 1$). Moreover, $(u,d) \mid (u,dd_0)^2$ but $(u,d) \nmid (u,dd_0)$, and hence $\omega(H,(u,d)) \ge 2$ implying that if $u$ is a prime element, then $\omega(H,(u,d)) = 2$.

A similar argument shows that $\tau(H,(u,d)) \ge 1$ for all $u \in \cA(H_0)$. Suppose that $u$ is a prime element and that $k \in \N_{\ge 2}$, $(u_1,d_1),\ldots,(u_k,d_k) \in \cA(H)$, and $(u,d) \mid (u_1,d_1) \cdot \ldots \cdot (u_k,d_k)$. Then $u \mid u_1\cdot \ldots \cdot u_k$ and thus $u$ divides one of $u_1,\ldots,u_k$, say $u \mid u_1$. But then $(u,d) \mid (u_1,d_1)(u_2,d_2)$, showing $\tau(H,(u,d)) \le 1$.

Recall the definition of the $\tau$-invariant as the supremum of a certain set from equation~\eqref{eq:tau} in Section \ref{2}. If $u \in \cA(H_0)$ is not a prime element, then the supremum of this set is attained for $k \ge 2$, and since $\sL_H((a,e)) = \sL_{H_0}(a)$ for all $a \in H_0$ and $e \in D$ (by statement \ref*{monext:transfer}), it is immediate that $\tau(H,(u,d)) = \tau(H_0,u)$. We note that if $k=2$, then there may be factorizations in $H$ that contribute elements to this set that are not already contributed by factorizations in $H_0$. However, if $k=2$, then we necessarily have that $\min \sL(u^{-1}a) = 1$ and thus the result is the same as for $k>2$.

\smallskip
We assume now that $H_0$ and $D$ are atomic and make use of the transfer homomorphism from statement \ref*{monext:transfer} in order to prove statement \ref*{monext:cat}. By Lemma \ref{2.2}.3(b), we have that $\sc_{H_0}(a) \le \sc_H((a,d)) \le \max\{\, \sc(a,\theta),\, \sc_{H_0}(a)\,\}$. We first show that $\sc(a,\theta) \le 2$.
For an element $(b,e) \in H$ we write $\overline{(b,e)}=(b,e)H^\times$ for its class in $H_\red$. Suppose that $z=\overline{(u_1,d_1)}\cdot \ldots \cdot\overline{(u_k,d_k)}$ and $z'=\overline{(u_1,d_1')}\cdot \ldots \cdot\overline{(u_k,d_k')}$ with $k,k' \ge 2$, $u_1,\ldots,u_k \in \cA(H_0)$, and $d_1,\ldots, d_k$, $d_1',\ldots,d_k' \in D$ are two factorizations of $(a,d)$ lying in the same fibre of $\theta$. Then $\overline{(u_1,d_1)}\cdot \ldots \cdot\overline{(u_{k-1},d_{k-1}d_k)}\,\overline{(u_k,1)}$ is also a factorization of $(a,d)$ lying in the same fibre and
  \[
    \sd\big(\overline{(u_1,d_1)}\cdot \ldots \cdot\overline{(u_{k-1},d_{k-1})}\,\overline{(u_k,d_k)},\; \overline{(u_1,d_1)}\cdot \ldots \cdot\overline{(u_{k-1},d_{k-1}d_k)}\,\overline{(u_k,1)}\big) \le 2.
  \]
Inductively, we find a $2$-chain from $z$ to $\overline{(u_1,d_1\cdot \ldots \cdot d_k)}\,\overline{(u_2,1)}\cdot \ldots \cdot \overline{(u_k,1)}$. Similarly, we find a $2$-chain from $z'$ to $\overline{(u_1,d_1'\cdot \ldots \cdot d_k')}\,\overline{(u_2,1)}\cdot \ldots \cdot\overline{(u_k,1)}$ and hence  a $2$-chain from $z$ to $z'$. This shows that $\sc(a,\theta) \le 2$.

Recall that $\sc_H((a,d)) = 0$ is equivalent to $(a,d)$ having a unique factorization in $H$ and, in any other case, $\sc_H((a,d)) \ge 2$. Therefore, to establish the remaining claims, it will suffice to show that in each of the cases listed, the element $(a,d)$ has a unique factorization and that in any other case, $(a,d)$ has at least two distinct factorizations. From the inequalities already proven, we then obtain that $\sc_H((a,d)) = \max\{ 2,\,\sc_{H_0}(a) \}$.

We first suppose that $D$ is not reduced. Assume that there exists $\varepsilon \in D^\times \setminus \{1,\, d\}$ and let $k \in \N_{\ge 2}$ and $u_1,\ldots,u_k \in \cA(H_0)$ with $a = u_1\cdot \ldots \cdot u_k$. Then
  \[
    \overline{(u_1,d)}\,\overline{(u_2,1)}\cdot \ldots \cdot\overline{(u_k,1)} \quad\text{and}\quad \overline{(u_1,d\varepsilon^{-1})}\, \overline{(u_2,\varepsilon)}\, \overline{(u_3,1)} \cdot \ldots \cdot \overline{(u_k,1)}
  \]
are two distinct factorizations of $(a,d)$ since $(d\varepsilon^{-1}, \varepsilon) \ne \{ (d,1),\, (1,d) \}$. Therefore, in this case, $\sc_H((a,d)) \ge 2$.

We now consider the case where $D^\times = \{1,\,d\}$ with $d \ne 1$. If $\sZ_{H_0}(a) = \{ (uH_0^\times)^2 \}$ for some $u \in \cA(H_0)$, then the unique factorization of $(a,d)$ is $\overline{(u,d)}\,\overline{(u,1)}$. The remaining cases are as follows.
  \begin{itemize}
    \item If $a$ has two distinct factorizations, then so does $(a,d)$.
    \smallskip
    \item If $a$ has a unique factorization represented by $u_1\cdot \ldots \cdot u_k$ for some $k \in \N_{\ge 2}$ and $u_1,\ldots,u_k \in \cA(H_0)$ with $u_1 \not \simeq u_2$, then
      \[
        \overline{(u_1,d)}\,\overline{(u_2,1)}\,\overline{(u_3,1)}\cdot \ldots \cdot \overline{(u_k,1)} \quad\text{and}\quad \overline{(u_1,1)}\,\overline{(u_2,d)}\,\overline{(u_3,1)}\cdot \ldots \cdot\overline{(u_k,1)}
      \]
are two distinct factorizations of $(a,d)$.
      \smallskip
    \item If $a \simeq u^k$ for some $k \in \N_{\ge 3}$ and $u \in \cA(H_0)$, then $\overline{(u,d)}\, \overline{(u,1)}^{k-1}$ and $\overline{(u,d)}^{3}\, \overline{(u,1)}^{k-3}$ are two distinct factorizations of $(a,d)$ (here we use $d^2=1$).
  \end{itemize}

Now suppose that $D$ is reduced. If $d=1$ and, for some $k \in \N_{\ge 2}$ and $u_1,\ldots, u_k \in \cA(H_0)$, $a = u_1\cdot \ldots \cdot u_k$ represents the unique factorization of $a$ in $H_0$, then $\overline{(u_1,1)}\cdot \ldots \cdot \overline{(u_k,1)}$ is the unique factorization of $(a,d)$ in $H$. If $d \in \cA(D)$ and for some $u \in \cA(H_0)$ and $k \in \N_{\ge 2}$, $a \simeq u^k$ represents the unique factorization of $a$ in $H_0$, then $\overline{(u,d)}\,\overline{(u,1)}^{k-1}$ is the unique factorization of $(a,d)$ in $H$.

We now show that in each of the remaining cases there exist two distinct factorizations of $(a,d)$ in $H$. This is clear if $a$ has two distinct factorizations in $H_0$, and we may assume in the following that $a$ has a unique factorization represented by $a = u_1\cdot \ldots \cdot u_k$ for some $k \in \N_{\ge 2}$ and $u_1,\ldots, u_k \in \cA(H_0)$.
  \begin{itemize}
    \item If $d \ne 1$ is not an atom, there must exist $d_1$, $d_2 \in D \setminus \{1\}$ such that $d=d_1 d_2$. Then $(a,d) = (u_1,d) (u_2,1) \cdot \ldots \cdot (u_k,1)$ and $(a,d) = (u_1,d_1)(u_2,d_2)(u_3,1)\cdot \ldots \cdot (u_k,1)$ give rise to two distinct factorizations of $(a,d)$.
   \smallskip
    \item If $d \in \cA(D)$ and there exist two distinct factors in the factorization of $a$, say $u_1 \not \simeq u_2$, then
      $(a,d) = (u_1,d)(u_2,1)(u_3,1)\cdot \ldots \cdot (u_k,1)$ and $(a,d) = (u_1,1)(u_2,d)(u_3,1)\cdot \ldots \cdot (u_k,1)$ give rise to two distinct factorizations of $(a,d)$.
  \end{itemize}

\smallskip
Note that we have so far showed that $\sc(H_0) \le \sc(H) \le \max\{2,\,\sc(H_0)\}$. We now show that $2 \le \sc(H)$. Since $H_0$ is atomic, but not a group, there exists $a \in H_0 \setminus H_0^\times$. Taking a power of $a$ if necessary, we may further assume that $a \in H_0 \setminus  \cA(H_0)$. If $D$ is not reduced, then $2 \le \sc((a,1)) \le \sc(H)$. If $D$ is reduced, then, since $D \ne \{1_D\}$  by assumption, there exists $d \in D \setminus D^\times$ and, again taking a power of $a$ if necessary, we may assume $d \in D \setminus \cA(D)$. Thus $2 \le \sc((a,d)) \le \sc(H)$.
\end{proof}

\medskip
We note that if $H_0=\mathbb N_0$ and $D$ is a group, then $H = \N_0 \monext D$ as above is a finitely primary monoid of rank $1$ and exponent $1$ (as discussed at the beginning of Section \ref{6}). Also, if $H_0$, $D$ and $E$ are monoids, then $(H_0 \monext D) \monext E = H_0 \monext (D \times E)$. Before defining, in Definition \ref{definition:almost-constant-monoid}, a monoid that will later be used to describe the monoid of stable isomorphism classes of finitely generated projective right modules over an HNP ring, we give a simple lemma that will be used in the proof of Proposition \ref{acm-factorization}.

\smallskip
\begin{lemma} \label{some-zss}
Let $G$ be an additive abelian group and let $G_0 \subset G$ be a subset.
  \mbox{}
  \begin{enumerate}
    \item \label{some-zss:tor} If $G_0 = \{g\}$ for some torsion element $g \in G$, then $\mathcal A (G_0) = \{g^{\ord (g)} \}$.

    \item \label{some-zss:tf} Let $n \in \N_{\ge 2}$, let $(e_1, \ldots, e_{n-1})$ be a family of independent elements of $G$ each of infinite order, and let
    \[
      G_0 = \big\{ A_i e_i : i \in [1,n-1] \big\} \cup \Big\{ - \sum_{i=1}^{n-1} B_i e_i \Big\} \,,
    \]
    where  $A_i, B_i \in \N$ with $\gcd(A_i,B_i) = 1$ for each $i \in [1,n-1]$. Then $\mathcal A (G_0) = \{U\}$ where
    \[
      U = \prod_{i=1}^{n-1} (A_i e_i)^{\frac{B_i L}{A_i}} \Big(-\sum_{i=1}^{n-1} B_i e_i\Big)^L \quad \text{with} \quad L = \lcm\{ A_i : i \in [1,n-1] \} \,.
    \]
  \end{enumerate}
  In particular, $\cB(G_0)$ is factorial in each case.
\end{lemma}

\begin{proof}
If $\mathcal A (G_0) = \{S\}$, then $S$ is a prime element, in which case $\mathcal B (G_0)$ is factorial.

\smallskip
Statement \ref*{some-zss:tor} is clear and we now prove statement \ref*{some-zss:tf}. From the definition of $L$, we have $A_i \mid B_i L$. Moreover, $\sigma(U) = \sum_{i=1}^{n-1} \big( \frac{B_i L}{A_i} A_i  - B_i L\big) e_i =  0 \in G$ which shows that $U \in \cB(G_0)$. Suppose that $S = \prod_{i=1}^{n-1} (A_i e_i)^{k_i} \big(-\sum_{i=1}^{n-1} B_i e_i\big)^{k} \in \cB(G_0)$ for some $k_1,\ldots,k_n,k \in \N_0$. Note that the $k_1,\ldots,k_n$ are uniquely determined by $k$ and $\sigma(S) = 0$. Therefore, to establish that $U$ is an atom of $\cB(G_0)$, and in fact the unique atom, it will suffice to show that $L \mid k$. Since $\sigma(S) = 0$, we have that $A_i k_i = k B_i$ for each $i \in [1,n-1]$. Since $\gcd(A_i,B_i) = 1$, this implies $A_i \mid k$ and hence $L = \lcm\{ A_i : i \in [1,n-1] \} \mid k$.
\end{proof}

\medskip
\begin{definition} \label{definition:almost-constant-monoid}
  Let $\Omega$ be a set containing a designated element $0$ and let $\mathbf c \in \Q_{> 0}^\Omega$ such that $c_0 = 1$ and $c_i \in \N$ for all but finitely many $i \in \Omega$.
  Define
  \[
    \N_0^\Omega(\vec c) = \{\, \mathbf x \in \mathbb N_0^\Omega : x_0 > 0 \text{ and } \card{\supp_{\Z^\Omega}(\mathbf x - x_0 \mathbf c)} < \infty \,\} \;\cup\; \{\mathbf 0\},
  \]
that is, $\N_0^\Omega(\vec c)$ consists of those vectors which are almost everywhere equal to a non-zero multiple (determined by the coordinate $x_0$) of $\vec c$, together with the vector $\vec 0$. We write $\ell(\vec x) = x_0$. If $\card{\Omega} < \infty$, then $\N_0^\Omega(\vec c) \cong (\N \times \N_0^{\Omega \setminus \{0\}}) \cup \{\vec 0\} \cong (\N \times \N_0^{(\Omega \setminus\{0\})}) \cup \{ \vec 0 \}$. Let $\Lambda$ be a (possibly empty) set of finite, pairwise disjoint subsets of $\Omega \setminus \{0\}$ each containing at least two elements and, for each $I \in \Lambda$, let $C_I = \sum_{i \in I} c_i$. We assume that $C_I \in \N$ for all $I \in \Lambda$. Define
  \[
    \N_0^\Omega(\vec c, \Lambda) = \Big\{\, \vec x \in \N_0^\Omega(\vec c) \,:\, \sum_{i \in I} x_i = C_I \ell(\vec x) \text{ for all $I \in \Lambda$ }\,\Big\}.
  \]
\end{definition}

\smallskip
By definition,  $\N_0^\Omega(\vec c)$ and $\N_0^\Omega(\vec c, \Lambda)$ are reduced  submonoids of $(\N_0^\Omega, +)$. However, note that the inclusion $\N_0^\Omega(\vec c) \subset \N_0^\Omega$ is not saturated. Indeed, if $\vec x$, $\vec y \in \N_0^\Omega(\vec c)$, then $\vec x$ divides $\vec y$ if and only if $\vec y - \vec x \in \N_0^\Omega(\vec c)$, that is, $\vec x \le \vec y$ and either $x_0 < y_0$ or $\vec x = \vec y$.

\smallskip
\begin{proposition} \label{acm-factorization}
  Let $H = \N_0^\Omega(\vec c, \Lambda)$ with $\Omega$, $\vec c$ and $\Lambda$ as in Definition \ref{definition:almost-constant-monoid}.
  Then $H$ is a saturated submonoid of $\N_0^\Omega(\vec c)$ and the map $\ell\colon H \to (\N_0, +)$ is a transfer homomorphism.
  In particular, $\vec x \in H$ is an atom if and only if $\ell(\vec x) = 1$, and $H$ is half-factorial.

  \begin{enumerate}
    \item\label{acmf:1} If $\card{\Omega} = 1$, then $H = \N_0$.
    \item\label{acmf:cyc} Suppose $2 \le \card{\Omega} < \infty$ and $\bigcup \Lambda = \Omega \setminus \{0\}$. Write $\Omega = [0, r]$ with $r \in \N$, and $\Lambda = \{ I_1,\ldots, I_n \}$ with $n \in \N$ and $\Omega \setminus \{0\} = I_1 \uplus \cdots \uplus I_n$.
          Set $C_i = C_{I_i}$ for all $i \in [1,n]$.
          \begin{enumerate}[ref=\arabic{enumi}(\alph*)]
          \item \label{acmf:cyc-krull}
              The map $j\colon H \hookrightarrow \N_0^{r},\; \vec x \mapsto (x_1,\ldots, x_r)$  is a divisor theory (note that $x_0$ is omitted), and hence $H$ is a Krull monoid.
               Then $G = \Z^{r}/\quo(j(H))$ is the divisor class group and $G_P = \{\, \vec e_i + \quo(j(H)) \colon i \in [1,r] \,\} \subset G$ is the set of classes containing prime divisors. There is a monomorphism
              \[
               \varphi^* \colon  G \to \begin{cases}
                  \vec \Z/C_1\Z & \text{if $n=1$,} \\
                  \vec \Z^{n-1}  & \text{if $n \ge 2$}
                \end{cases} \qquad \text{with the following properties{\rm\,:}}
              \]
              If $n=1$, then  $\varphi^*$ is an isomorphism, $\varphi^* (G_P) = \{ 1 + C_1\Z \}$, and the unique class in $G_P$ contains precisely $\card{I_1}$ prime divisors.
               Let $n \ge 2$ and, for all $i \in [1,n-1]$, set $A_i = \frac{C_n}{\gcd(C_i,C_n)}$ and $B_i = \frac{C_i}{\gcd(C_i,C_n)}$.
               Then $\varphi^*(G)$ is a full rank subgroup of $\mathbb Z^{n-1}$ and
              \[
              \varphi^* ( G_P) = \big\{ A_i \vec e_i : i \in [1,n-1] \big\} \cup \Big\{ - \sum_{i=1}^{n-1} B_i \vec e_i \Big\} \,.
              \]
              For each $i \in [1,n-1]$, the class mapped onto the element $A_i \vec e_i$ contains precisely $\card{I_{i}}$ prime divisors and the class mapped onto $-\sum_{i=1}^{n-1} B_i \vec e_i$ contains precisely $\card{I_n}$ prime divisors.
            \item \label{acmf:cyc-atoms}
              An element $\vec u \in H$ is an atom if and only if $\sum_{i \in I} u_i = C_I$ for all (equivalently, any) $I \in \Lambda$.
            \item \label{acmf:cyc-cat}
              We have that $\mathsf c(H) \le 2$. Moreover, the monoid $H$ is factorial if and only if $\card{\Lambda} = 1$ and $C_1 = 1$.
              In this case, the prime elements are precisely the $\vec e_0 + \vec e_i$ for $i \in [1,r]$.
              If $H$ is not factorial, then $H$ contains no prime elements.
            \item \label{acmf:cyc-omega}
              For any atom $\vec u \in \cA(H)$ we have
              \[
                \sum_{I \in \Lambda} (C_I - \min_{i \in I} u_i) \;\le\; \omega(H,\vec u) \;\le\; \sum_{I \in \Lambda} C_I.
              \]
            \item \label{acmf:cyc-tame} If $H$ is not factorial, then $\mathsf t(H) = \omega(H) = \sum_{I \in \Lambda} C_I$ .
          \end{enumerate}

      \item \label{acmf:fin-generic}
        If $2 \le \card{\Omega} < \infty$ but $\bigcup \Lambda \subsetneq \Omega \setminus \{0\}$, set $\Omega' = \bigcup \Lambda \cup \{0\}$.
            Let $\vec c' \in \Q_{>0}^{\Omega'}$ be defined by $c'_i = c_i$ for all $i \in \Omega'$ and let $H' = \N_0^{\Omega'}(\vec c', \Lambda)$.
            Then $H \cong H' \monext \N_0^{\card{\Omega} - \card{\Omega'}}$.
            For each $(\vec u, \vec x) \in \cA(H)$,
            \[
              \max\{ \omega(H',\vec u),\, \length{\vec x} \} \;\le\; \omega(H,(\vec u,\vec x)) \;\le\; \omega(H',\vec u) + \length{\vec x} + 1
            \]
            and thus $\omega(H,(\vec u,\vec x)) < \infty$ and $\st(H,(\vec u,\vec x)) < \infty$.
            Moreover, $H$ is a half-factorial \FF-monoid with $\sc(H) = 2$ and $\st(H) = \infty$. Also, $H$ is not a Krull monoid.
      \item If $\card{\Omega} = \infty$, then
        \begin{enumerate}[ref=\arabic{enumi}(\alph*)]
          \item\label{acmf:inf-notff} $\card{\sZ_H(\vec x)} = \infty$ for all $\vec x \in H \setminus (\cA(H) \cup H^\times)$. In particular, $H$ is not a \FF-monoid and is therefore not a submonoid of a free abelian monoid,
          \item\label{acmf:inf-cat}  $\sc_H(\vec x) = 2$ for all $\vec x \in H \setminus (\cA(H) \cup H^\times)$ and
          \item\label{acmf:inf-tame} $\omega(H, \vec u) = \tau(H, \vec u) = \st(H, \vec u) = \infty$ for all $\vec u \in \cA(H)$.
        \end{enumerate}
  \end{enumerate}
\end{proposition}

\begin{proof}
Suppose that $\vec x$, $\vec y \in H$ such that $\vec x$ divides $\vec y$ in $\N_0^\Omega(\vec c)$. That is, $\vec y - \vec x \in \N_0^\Omega(\vec c)$. For each $I \in \Lambda$, we have $\sum_{i \in I} y_i - x_i = C_I (\ell(\vec y) - \ell(\vec x)) = C_I \ell(\vec{y - x})$, and thus $\vec y - \vec x \in H$. Therefore the inclusion $H \subset \N_0^\Omega(\vec c)$ is saturated. We now show that $\ell\colon H \to \N_0$ is a transfer homomorphism. Clearly this map is surjective and, since $H$ is reduced, it will suffice to show: If $\vec x \in H$ and $\ell(\vec x) = k + l$ with $k$, $l \in \N_0$, then there exist $\vec y$, $\vec z \in H$ such that $\vec x = \vec y + \vec z$ and $\ell(\vec y) = k$, $\ell(\vec z) = l$.

If one of $k$ or $l$ is $0$ we may, without restriction, assume that $l=0$. Then $\vec y = \vec x$ and $\vec z = \vec 0$ give the result. From now on we assume that $k$, $l> 0$. Let $\Lambda' \subset \Lambda$ consist of those finitely many $I \in \Lambda$ with $I \cap \supp(\vec x - \ell(\vec x) \vec c) \ne \emptyset$. For $i \in \bigcup \Lambda'$, let $y_i'$, $z_i' \in \N_0$ be such that $y_i' + z_i' = x_i$ and such that, for all $I \in \Lambda'$, it holds that $\sum_{i \in I} y_i' = C_I k$ and $\sum_{i \in I} z_i' = C_I l$.
Define $\vec y$ and $\vec z \in \N_0^\Omega$ by
  \begin{align*}
    y_i &= \begin{cases}
               y_i'   &\quad\text{if $i \in \bigcup \Lambda'$,} \\
               k c_i  &\quad\text{if $i \in \Omega \setminus \bigcup \Lambda'$,}
           \end{cases}
     & z_i &=  \begin{cases}
              z_i'   &\quad\text{if $i \in \bigcup \Lambda'$,} \\
              l c_i  &\quad\text{if $i \in \Omega \setminus \bigcup \Lambda'$.}
           \end{cases}
  \end{align*}
  Then $\vec y$, $\vec z \in \N_0^\Omega(\vec c, \Lambda)$ with $\ell(\vec y) = k$, $\ell(\vec z) = l$, and $\vec x = \vec y + \vec z$ as required.

\smallskip
Statement \ref*{acmf:1} is clear and thus we suppose that $2 \le \card{\Omega} < \infty$ and $\Omega \setminus \{0\} = \bigcup \Lambda$.

We first verify statement  \ref*{acmf:cyc-krull}. Since $\sum_{i \in I_1} x_i = C_1 x_0$ and $C_1 \ne 0$, $j$ is injective. We now check that $j$ is a divisor homomorphism. Let $\vec x$, $\vec y \in H$ with $j(\vec x) \le j(\vec y)$. If $\vec x = \vec y$, there is nothing to show. If $\vec x \ne \vec y$, then there necessarily exists $I \in \Lambda$ with $\sum_{i \in I} x_i < \sum_{i \in I} y_i$ and hence $x_0 < y_0$. Consequently, $\vec x$ divides $\vec y$ in $H$ and thus $j$ is a divisor homomorphism. To prove that $j$ is a divisor theory, we need to show that each standard basis vector $\vec e_i \in \N_0^r$ is the greatest common divisor of a finite, non-empty set in the image of $j$. Let $i \in [1,r]$, let $I_0 \in \Lambda$ be such that $i \in I_0$, and let $i' \in I_0 \setminus \{i\}$. Define
  \begin{align*}
    \vec x &= \vec e_0 + \vec e_i + (C_{I_0} - 1)\vec e_{i'} + \sum_{I \in \Lambda\setminus\{I_0\}} C_I \vec e_{\min I}, \text{ and} \\
    \vec y &= \vec e_0 + C_{I_0} \vec e_i                    + \sum_{I \in \Lambda\setminus\{I_0\}} C_I \vec e_{\max I}.
  \end{align*}
  Then $\vec x$, $\vec y \in H$ and, recalling that $\card{I} \ge 2$ for all $I \in \Lambda$, $\gcd_{(\N_0^r,+)}(j(\vec x), j(\vec y)) = \vec e_i$. Thus $j \colon H \hookrightarrow \N_0^r$ is a divisor theory, $H$ is a Krull monoid, and the class group of $H$ is $\quo(\N_0^r) / \quo(j(H))  = \Z^r/\quo(j(H))$.

\smallskip
 We now determine the structure of the divisor class group of $H$ and determine the set of classes containing prime divisors. Suppose that $n=1$ and define $\varphi^*\colon \Z^r \to \Z/C_1\Z$ by $\vec x \mapsto \sum_{i=1}^r x_i + C_1\Z$. We see immediately that $\varphi^*$ is an epimorphism and $\ker(\varphi^*) = \quo(j(H))$ follows easily. Therefore $G \cong \Z/C_1\Z$. Since $\varphi^*(\vec e_i) = 1 + C_1\Z$ for all $i \in [1,r]$, $G_P$ is as claimed. Now suppose that $n \ge 2$ and let $\varphi^*\colon \Z^r \to \Z^{n-1}$ be defined by
  \[
    \vec x \mapsto (A_1 \sum_{i \in I_1} x_i - B_1 \sum_{i \in I_n} x_i,\;\; A_2 \sum_{i \in I_2} x_i - B_2 \sum_{i \in I_n} x_i ,\;\; \ldots,\;\; A_{n-1} \sum_{i \in I_{n-1}} x_i - B_{n-1} \sum_{i \in I_{n}} x_i).
  \]
One easily checks that $\ker(\varphi^*) = \quo(j(H))$, showing that $G$ embeds into $\Z^{n-1}$ via $\varphi^*$. Considering the images of $\vec e_1,\ldots, \vec e_r$ under $\varphi^*$, the description of $G_P$ given in statement \ref*{acmf:cyc-krull} follows, and we see that $\varphi^*(G)$ is a subgroup of full rank of $\Z^{n-1}$.

\smallskip
For statement  \ref*{acmf:cyc-atoms}, note that an element $\vec u \in H$ is an atom if and only if $\ell(\vec u) = 1$, and this is the case if and only if $\sum_{i \in I} u_i = C_I$ for all (equivalently, any) $I \in \Lambda$.

\smallskip
Consider statement  \ref*{acmf:cyc-cat}. By Lemma \cref{some-zss}, $\mathcal B (G_P)$ is factorial and hence $\mathsf c (G_P) = 0$. Thus Proposition \ref{3.1} implies that $\sc(H) \le \max\{ 2,\, \sc( G_P) \} \le 2$. If $\card{\Lambda} = 1$ and $C_1 = 1$, then $j$ is surjective. Thus $H \cong \N_0^r$, showing that $H$ is factorial. The prime elements of $(\N_0^r,+)$ are simply the standard basis vectors $\vec e_i$ for $i \in [1,r]$, and their preimages under $j$ are precisely the elements $\vec e_0 + \vec e_i \in H$ for $i \in [1,r]$. If $\card{\Lambda} > 1$ or $C_1 > 1$, then no atom is prime and thus $H$ is not factorial. Let $\vec u \in \cA(H)$. The lower bound given in \ref*{acmf:cyc-omega}, which we will soon verify, implies that $\omega(H,\vec u) \ge 2$ unless $\card{\Lambda}=1$, $C_1=2$, and $u_1=u_2=1$. Note that $\min_{i \in I} u_i \le \lfloor C_I/2 \rfloor$ for each $I \in \Lambda$ and thus, in this case, $\vec u$ is not a prime element. In the remaining case, one easily checks that again, $\vec u$ is not a prime element.

\smallskip
We now verify the bounds on the omega invariant as given in  \ref*{acmf:cyc-omega}. If $\card{\Lambda} = 1$ and $C_1 = 1$, then $H$ is factorial and the inequalities hold trivially. We assume from now on that this is not the case and hence $\sum_{I \in \Lambda} C_I \ge 2$. We first show that $\omega(H,\vec u) \le \sum_{I \in \Lambda} C_I$.
Let $k \in \N_{\ge 2}$ and let $\vec v_1,\ldots, \vec v_k \in \cA(H)$ be such that $\vec u$ divides $\sum_{i=1}^k \vec v_i$. If $J \subset [1,k]$, then $\vec u$ divides $\sum_{j \in J} \vec v_j$ if and only if $\vec u \le \sum_{j \in J} \vec v_j$. Since $\vec u \le \sum_{i=1}^k \vec v_i$ and $\sum_{i \in \Omega\setminus\{0\}} u_i = \sum_{I \in \Lambda} C_I$, we can recursively construct a subset $J \subset [1,k]$ of size at most $\sum_{I \in \Lambda} C_I$ such that $\vec u \le \sum_{j \in J} \vec v_j$. This is done by adding, in each step, a vector $\vec v_j$ with $v_{j,i} > 0$ for some $i \in \Omega\setminus \{0\}$ for which $u_i < \sum_{j \in J} v_{j,i}$.

We now prove the lower bound. By renumbering the coordinates if necessary, we may assume $u_{\min I} = \min_{i \in I} u_i$ for all $I \in \Lambda$. For $i \in \Omega$, let $I \in \Lambda$ be the unique set containing $i$ and define
  \[
    \vec v_i = \vec e_0 + \vec e_i + (C_I - 1) \vec e_{\min I} \;+ \sum_{I' \in \Lambda\setminus\{\,I\,\}} C_{I'} \vec e_{\min I'} \;\in\; \cA(H).
  \]
Since $\vec u = \sum_{i \in \Omega} u_i \vec e_i = \sum_{I \in \Lambda} \sum_{i \in I} u_i \vec e_i$, it is clear that $\vec u$ divides $\sum_{I \in \Lambda} \sum_{i \in I \setminus \{ \min I \}} u_i \vec v_i$. For $I \in \Lambda$ with $C_I > 1$, this is clear from the definition of the $\vec v_i$'s. If $C_I=1$, then $u_{\min I} = 0$ since $\card{I} \ge 2$, yet $\vec u$ does not divide any proper subsum. Therefore
  \[
    \omega(H,\vec u) \ge \sum_{I \in \Lambda} \Big[ \big(\sum_{i \in I} u_i\big) - \min_{i \in I} u_i \Big]= \sum_{I \in \Lambda} (C_I - \min_{i\in I} u_i).
  \]

\smallskip
We now prove statement \ref*{acmf:cyc-tame}. By definition, $\omega(H) = \sup_{\vec u \in \cA(H)} \omega(H,\vec u)$. Since $H$ is half-factorial, $\st(H) = \omega(H)$.
The element
  \[
    \vec u = \vec e_0 + \sum_{I \in \Lambda} C_I \vec e_{\min I} \in H
  \]
  is an atom and the bounds in \ref*{acmf:cyc-omega} imply that $\omega(H,\vec u) = \sum_{I \in \Lambda} C_I$. Therefore $\omega(H) \ge \sum_{I \in \Lambda} C_I$ and the upper bound again follows from \ref*{acmf:cyc-omega}.

  \smallskip
Consider statement  \ref*{acmf:fin-generic} and let $D = \N_0^{\card{\Omega} - \card{\Omega_0}}$. The isomorphism $H \cong H' \monext D$ is immediate from the definitions.
Because $D$ is factorial, we have $\omega(D, \vec x) = \length{\vec x}$ for all $\vec x \in D$. \Subref{monext:cat} together with \ref*{acmf:cyc-cat} implies that $\sc(H) = 2$, and \Subref{monext:omega} implies the remaining inequalities. As a submonoid of a free abelian monoid, $H$ is an \FF-monoid (\cite[Corollary 1.5.7]{Ge-HK06a}). However, since it is not completely integrally closed in its quotient group (by \Subref{monext:nocic}), it cannot be a Krull monoid.

\smallskip
We suppose for the remainder of the proof that $\card{\Omega} = \infty$. For notational ease we assume $\N_0 \subset \Omega$.

  \smallskip
Consider statement \ref*{acmf:inf-notff}. To show $\card{\sZ_H(\vec x)} = \infty$ it suffices to show that there are infinitely many atoms of $H$ dividing $\vec x$. By definition of $H$, there exists a finite subset $\Omega' \subset \Omega$ having the property that $I \cap \Omega' \ne \emptyset$ already implies $I \subset \Omega'$ for all $I \in \Lambda$ and such that for each $i \in \Omega \setminus \Omega'$ it holds that $x_i = \ell(x) c_i$ and $c_i \in \N$. We may assume that $\N_0 \subset \Omega \setminus \Omega'$. Moreover, we may assume that for each $I \in \Lambda$ we have $\card{I \cap \N} \le 1$ and, if $i \in I \cap \N$, then $c_i = \min_{i' \in I} c_{i'}$.

For each $i \in \N$, define $\vec u_i \in \N_0^\Omega$ as follows: If there exists $I \in \Lambda$ with $i \in I$, then let $i' \in I \setminus \{i\}$. Note also that $i' \in \Omega \setminus \Omega'$ by the choice of $\Omega'$. We set $u_{i,i} = 0$, $u_{i,i'} = c_i + c_{i'}$, and $u_{i,j} = c_j$ for each $j \in \Omega \setminus (\Omega' \cup \{i,\, i'\})$.
On the other hand, if $i$ is not contained in any $I \in \Lambda$, we set $u_{i,i} = 0$ and  $u_{i,j} = c_j$ for each $j \in \Omega \setminus (\Omega' \cup \{i\}) $.
In either case, we can choose $u_{i,j} \le x_{i,j}$ for each $j \in \Omega'$ such that $\sum_{j \in I} u_{i,j} = C_I$ for all $I \in \Lambda$ with $I \subset \Omega'$.
Therefore $\sum_{j \in I} u_{i,j} = C_I$ for all $I \in \Lambda$ and $u_{i,j} = c_j$ for all but finitely many $j \in \Omega$. Therefore $\vec u_i \in \cA(H)$. By construction, $\vec u_i \le \vec x$. In the first case, due to the minimal choice of $c_i$, we have $u_{i,i'} = c_i + c_{i'} \le 2c_{i'} \le \ell(\vec x) c_{i'} = x_{i'}$, with the last equality holding since $i' \in \Omega\setminus \Omega'$. Since $\ell(\vec u_i) = 1 < \ell(\vec x)$, $\vec u_i$ divides $\vec x$. Clearly these atoms are pairwise distinct. By what we have just shown, $H$ is not an \FF-monoid and therefore cannot a submonoid of a free abelian monoid by \cite[Corollary 1.5.7]{Ge-HK06a}.

We now compute the catenary degree of $H$ as stated in  \ref*{acmf:inf-cat}. From \ref*{acmf:inf-notff}, every nonzero, non-atom element has at least two distinct factorizations and hence $\sc_H(\vec x) \geq 2$ for all $\vec x \in H$. We now show that $\sc_H(\vec x) \leq 2$ by projecting to the finite case. Suppose that $\vec x = \vec u_1 + \cdots + \vec u_k = \vec v_1 + \cdots + \vec v_k$ for some $k \in \N_{\ge 2}$ and that $\vec u_1,\ldots, \vec u_k$, $\vec v_1,\ldots,\vec v_k \in \cA(H)$. Let $\Omega' \subset \Omega$ be the smallest subset containing
  \[
    \big\{ j \in \Omega : u_{i,j} \ne c_j\text{ or }v_{i,j} \ne c_j\text{ for some $i \in [1,k]$} \big\} \cup \{0\}
  \]
  and having the property that whenever $I \cap \Omega' \ne \emptyset$ for $I \in \Lambda$, then already $I \subset \Omega'$.
  Observe that $\Omega'$ is finite and that $\vec u_{i,j} = \vec v_{i',j} = c_j$ for all $i$, $i' \in [1,k]$ and $j \in \Omega\setminus \Omega'$.
  Set $\Lambda' = \{ I \in \Lambda : I \subset \Omega' \}$ and let $\vec c' \in \N_0^{\Omega'}$ be defined by $c_i' = c_i$ for all $i \in \Omega'$.
Define $H' = \N_0^{\Omega'}(\vec c', \Lambda')$ and note that there is a canonical projection $\pi\colon H \to H'$. By the finiteness of $\Omega'$ we immediately have that $\sc(H') \le 2$. Thus there exists a sequence of factorizations $z_1',\ldots,z_l' \in \sZ(H')$ of $\pi(\vec x)$ with $z_1'=\pi(\vec u_1)\cdot\ldots\cdot\pi(\vec u_k)$,\ $z_l'=\pi(\vec v_1)\cdot\ldots\cdot\pi(\vec v_k)$ and such that $\mathsf d(z_i',z_{i+1}') \le 2$ for all $i \in [1,l-1]$. For any factor $\vec w'$ occurring in some $z_j'$, we lift it to $\vec w \in H$ by setting $w_{j} = u_{1,j}$ for all $j \in \Omega \setminus \Omega'$. Then $\vec w \in \cA(H)$ and we can lift the factorizations $z_1',\ldots,z_l'$ of $\pi(\vec x)$ to factorizations $z_1,\ldots,z_l$ of $\vec x$. These factorizations form a sequence connecting the two factorization of $\vec x$ that we began with and have the property that $\sd(z_i,z_{i+1}) \le 2$ for all $i \in [1,l-1]$.

\smallskip
Finally, we  consider the infinitude of the invariants given in \ref*{acmf:inf-tame}. Let $\vec u \in \cA(H)$ and $k \in \N_{\ge 2}$. We will show that $\omega(H,\vec u) \ge k$, as the other invariants are then equal to $\omega(H,u)$ by half-factoriality. Note that $\vec u$ is necessarily non-zero in all but finitely many coordinates and so we may assume $u_i \ne 0$ for $i\in[1,k]$. Moreover, since $\Omega$ is infinite and each $I \in \Lambda$ is finite, we may suppose that no two $i$, $j \in [1,k]$ belong to the same set $I \in \Lambda$. (For the following argument it would suffice that $I \not \subset [1,k]$.) Modifying finitely many coordinates of $\vec u$ as needed, we can therefore construct $\vec u_1,\ldots, \vec u_k \in \cA(H)$ such that $u_{i,j} = u_j \delta_{i,j}$ for all $i,j \in [1,k]$, and $u_{i,j} \ge u_j$ for all $i \in [1,k]$ and $j \in \Omega \setminus [1,k]$. Then $\vec u$ divides $\vec u_1 + \cdots + \vec u_k$, but for all $j \in [1,k]$, $u_{j} > u_{1,j} + \cdots + u_{k,j} - u_{j,j} = 0$, showing that $\vec u$ divides no proper subsum.
\end{proof}

\medskip
Now with the appropriate arithmetical results, we apply Propositions \ref{monext} and \ref{acm-factorization} to study direct-sum decompositions of modules over HNP rings. For a hereditary Noetherian prime ring (HNP ring) $R$, we consider the class $\cC$ of finitely generated right $R$-modules. In the noncommutative setting there are two invariants (the \emph{genus} and the \emph{Steinitz class}) which describe the stable isomorphism class of a finitely generated projective right $R$-module $P$. In general, however, the isomorphism class of $P$ is determined by its stable isomorphism class only if $\udim(P) \ge 2$ and such a module is indecomposable if and only if $\udim(P) = 1$. Thus the forthcoming description of the direct-sum decomposition of finitely generated projective right $R$-modules is one where the indecomposable factors are determined up to stable isomorphism. If $R$ has the additional property that any two finitely generated projective right $R$-modules that are stably isomorphic are already isomorphic, then this result is a description up to isomorphism.

\medskip
Let $V$ and $W$ be two simple right $R$-modules. Then $W$ is called a \emph{successor} of $V$ and $V$ is called a \emph{predecessor} of $W$ if $\Ext_R^1(V,W) \ne \vec 0$. Let $\cW$ be a set of representatives (of isomorphism classes) of the simple unfaithful right $R$-modules.
We note that every $V \in \cW$ is contained in a unique \emph{tower} of $R$ (\cite[\S19]{Le-Ro11a}).
A \emph{tower} $\cT$ is a finite set of simple right $R$-modules, ordered with respect to the successor relationship, and having the following structure: Every tower  $\cT$ is either cyclically ordered and each simple module in $\cT$ is unfaithful, in which case we say that $\cT$ is a \emph{cycle tower}, or $\cT$ is linearly ordered and only the first module (the only module in the tower not having a predecessor) is faithful, in which case $\cT$ is said to be a \emph{faithful tower}.
The length of a tower is the number of distinct modules contained in it, and a tower is \emph{non-trivial} if it contains more than one module.

\medskip
We briefly recall the notions of rank, genus and the Steinitz class of a finitely generated projective right $R$-module, as these invariants are used to describe stable isomorphism classes of finitely generated projective right $R$-modules. We refer the reader to \cite[\S33 and \S35]{Le-Ro11a} for additional details. Let $P$ be a finitely generated projective right $R$-module and let $V \in \cW$. Then $M = \ann(V)$ is a maximal ideal of $R$, $R/M$ is simple artinian, and the \emph{rank of $M$ at $V$}, denoted by $\rho(P, V) \in \N_0$, is defined to be the length of the $R/M$-module $P/PM$. If $\mathcal T$ is a cycle tower, we set $\rho(P, \mathcal T) = \sum_{V \in \mathcal T} \rho(P,V)$. Let $\modspec(R)$ denote the set of isomorphism classes of all unfaithful simple right $R$-modules together with the trivial module $\vec 0$.
For a finitely generated projective right $R$-module $P$, we set $\Psi_V(P) = \rho(P,V)$ if $V$ is an unfaithful simple right $R$-module, and $\Psi_{\vec 0}(P) = \udim(P)$.
Then
\[
  \Psi(P) = (\Psi_{V}(P))_{V \in \modspec(R)} \in \N_0^{\modspec(R)}
\]
is called the \emph{genus} of $P$. Two finitely generated projective right $R$-modules $P$ and $Q$ are \emph{stably isomorphic} if there exists a finitely generated projective right $R$-module $X$ such that $P \oplus X \cong Q \oplus X$. We denote by $[P]$ the stable isomorphism class of $P$. The direct sum operation on modules induces the structure of a commutative semigroup on the set of stable isomorphism classes, and by $\mathbf K_0(R)$ we denote its quotient group. The genus $\Psi$ induces a homomorphism $\Psi^+\colon \mathbf K_0(R) \to \Z^{\modspec(R)}$, and $G(R) = \Ker(\Psi^+)$ is called the \emph{ideal class group of $R$}. By choosing a base point set $\cB$ of non-zero finitely generated projective right $R$-modules consisting of exactly one module in each genus, and such that $\cB$ is closed, up to isomorphism, under direct sums, we can associate to any nonzero finitely generated projective right $R$-module $P$ a class $\mathcal S(P) = [P] - [B] \in G(R)$, where $B \in \cB$. We call $\mathcal S(P)$ the \emph{Steinitz class} of $P$ and we set $\mathcal S(\vec 0) = \vec 0$.

\medskip
We are now ready to characterize the semigroup of stable isomorphism classes of finitely generated projective right modules over an HNP ring. We note that the characterization of factoriality in Theorem \ref{hnp-proj} was already obtained by Levy and Robson in \cite[Theorem 39.5]{Le-Ro11a}.

\smallskip
\begin{theorem} \label{hnp-proj}
  Let $R$ be an HNP ring and let $H$ be the semigroup of stable isomorphism classes of finitely generated projective right $R$-modules with operation induced by the direct sum of modules.
  \begin{enumerate}
  \item
  Let $\Omega$ denote the set of isomorphism classes of all unfaithful simple right $R$-modules which are contained in a non-trivial tower, together with the trivial module, denoted by $\vec 0$.
  For $V \in \Omega \setminus \{\vec 0\}$, let $c_V = \frac{\rho(R,V)}{\udim(R)} \in \Q_{> 0}$ and $c_{\vec 0} = 1$.
  Finally, let $\Lambda$ be the set of all non-trivial cycle towers.
  Then
  \[
    H \cong \N_0^\Omega(\vec c, \Lambda) \monext G(R)
  \]
  with the product $\monext$ as defined in Proposition \ref{monext}.
  In particular, $H$ is half-factorial, $\sc(H) \le 2$ and the following are equivalent{\rm \,:}
  \begin{enumerate}
    \enumequiv
    \item $H$ is factorial (i.e., $R$ satisfies stable uniqueness).
    \item $G (R) = \vec 0$ and either $R$ has no non-trivial towers (i.e., $R$ is a Dedekind prime ring), or $R$ has a unique non-trivial tower $\mathcal T$ which is a cycle tower and $\rho(R, \mathcal T) = \udim(R)$.
  \end{enumerate}

  \smallskip
\item If $R$ has infinitely many non-trivial towers, then $\st(H,[U]) = \omega(H,[U]) = \infty$ for each indecomposable finitely generated projective right $R$-module $U$.

  \smallskip
  \item
  Suppose  that $R$ has only finitely many non-trivial towers.
  Then $\st(H,[U])$, $\omega(H,[U]) < \infty$ for each indecomposable finitely generated projective right $R$-module $U$.
  If $R$ has at least one non-trivial faithful tower, then $\st(H) = \omega(H) = \infty$.
  If $R$ has no non-trivial faithful tower, only finitely many non-trivial cycle towers $\cT_1, \ldots, \cT_n$ with $n \in \N_0$, and $H$ is not factorial, then
  \[
    \st(H) = \omega(H) = \frac{1}{\udim(R)} \sum_{i=1}^n \rho(R,\cT_i).
  \]
  \end{enumerate}
\end{theorem}

\begin{proof}
It will suffice to establish that $H \cong \N_0^\Omega(\vec c, \Lambda) \monext G(R)$ as the remaining claims will then follow from Proposition \ref{acm-factorization}. The genus and the Steinitz class are both additive on direct sums and thus give rise to a monoid homomorphism $H \to \N_0^{\modspec(R)} \propto G(R),\; [P] \mapsto (\Psi(P), \mathcal S(P))$. We note that $\mathcal S(\vec 0) = \vec 0$ and that $\vec 0$ is the only module of uniform dimension zero. The genus satisfies a number of necessary conditions, namely that $\Psi_V(P) = c_V \Psi_0(P)$ for almost all $V \in \modspec(R)$ and that it has standard rank at every cycle tower, that is, $\rho(P,\mathcal T) = \sum_{V \in \mathcal T} c_V \Psi_0(P)$ for all cycle towers $\mathcal T$. In particular, if $V$ is contained in a trivial tower, then this is necessary a cycle tower and hence implies $\Psi_V(P) = c_V \Psi_0(P)$. Thus instead of $\Psi(P)$ we may consider $\Psi'(P) \subset \N_0^\Omega$ were we omit the components corresponding to unfaithful simple right $R$-modules that are contained in a trivial tower. We obtain a homomorphism
  \[
    \Phi\colon H \to \N_0^{\Omega}(\vec c, \Lambda) \propto G(R),\; [P] \mapsto (\Psi'(P), \mathcal S(P)).
  \]
The main theorem of Levy and Robson (\cite[Theorem 35.13]{Le-Ro11a}) implies that the genus and the Steinitz class are independent invariants, and that up to the stated conditions on the rank, all values can be obtained. In other words, $\Phi$ is an isomorphism.
\end{proof}

\medskip
\begin{remarks}
\begin{remenumerate}
\item If $R$ is such that each two stably isomorphic right $R$-modules are isomorphic, then $H = \mathcal V(\mathcal C_\proj)$. If this is not the case, then $H$ still provides information on direct sum decompositions of finitely generated projective modules with the summands determined up to stable isomorphism. Specifically, we have the following: Clearly, if $P = U_1 \oplus \cdots \oplus U_k$ for some $k \in \N$ and indecomposable right $R$-modules $U_1, \ldots, U_k$, then $[P] = [U_1] + \cdots + [U_k]$ is a factorization of $[P]$ in $H$. Let $P$ be a right $R$-module with $\udim(P) \ge 2$ (that is, $P$ is neither the zero module nor indecomposable). If $[P] = [U_1] + \cdots + [U_k]$ for some $k \in \N_{\ge 2}$ and atoms $[U_1], \ldots, [U_k]$ of $H$, then $[P] = [U_1 \oplus \cdots \oplus U_k]$ and, since $\udim(P) \ge 2$, \cite[Theorem 34.6]{Le-Ro11a} implies that $P \cong U_1 \oplus \cdots \oplus U_k$. Therefore, a factorization of $[P]$ in $H$ gives rise to one of $P$ into indecomposables, with the stable isomorphism classes of the indecomposable summands determined by the factorization of $[P]$ in $H$.

\smallskip
\item An HNP ring $R$ has finitely many non-trivial towers if and only if it is a multichain idealizer from a Dedekind prime ring $S$ (\cite[Proposition 30.5]{Le-Ro11a}). A sufficient (but not necessary) condition for there to be no non-trivial faithful towers is for $R$ to be right (equivalently, left) bounded (\cite[Lemma 18.2]{Le-Ro11a}).

\smallskip
\item Let $R$ be an HNP ring and let $\mathcal C$ denote the class of finitely generated right $R$-modules. Then $\cV(\cC) \cong \cV(\cC_{\tor}) \times \cV(\cC_{\proj})$ and $\cV(\cC_{\tor})$ is factorial.

\begin{proof}
By \cite[Corollary 12.16(ii), Theorem 12.18]{Le-Ro11a}, every finitely generated right $R$-module decomposes uniquely as a direct sum of a torsion module and a torsion-free module. The first has finite length and hence has a unique decomposition into indecomposables, while the second is projective.
\end{proof}

\smallskip
A detailed description of $\cV(\cC_{\tor})$ is given in \cite[\S41]{Le-Ro11a}.
\end{remenumerate}
\end{remarks}

\bigskip

\providecommand{\bysame}{\leavevmode\hbox to3em{\hrulefill}\thinspace}
\providecommand{\MR}{\relax\ifhmode\unskip\space\fi MR }
\providecommand{\MRhref}[2]{%
  \href{http://www.ams.org/mathscinet-getitem?mr=#1}{#2}
}
\providecommand{\href}[2]{#2}

\end{document}